\documentclass{article}
\usepackage{a4wide}
\usepackage{graphicx} 
\usepackage{comment}
\usepackage{subcaption}
\usepackage{float}
\usepackage{amsmath}
\usepackage{amsthm}
\usepackage{amsfonts}
\usepackage{thmtools}
\usepackage{thm-restate} 
\newtheorem{theorem}{Theorem}[section]
\newtheorem{lemma}[theorem]{Lemma}
\newtheorem{definition}[theorem]{Definition}

\newtheorem{corollary}[theorem]{Corollary}

\newtheorem{remark}[theorem]{Remark}
\usepackage[numbers,sort]{natbib}

\newtheorem*{question*}{Question}
\newtheorem*{theorem*}{Theorem}
\newtheorem{proposition}[theorem]{Proposition}
\usepackage{hyperref}
\usepackage[dvipsnames]{xcolor}
\usepackage{cleveref}
\usepackage[colorinlistoftodos]{todonotes}
\usepackage{amsmath, amssymb, amsfonts} 

\usepackage{ifthen}

\usepackage{tikz}
\usetikzlibrary{arrows.meta}
\usetikzlibrary{calc}
\usetikzlibrary{positioning}
\usetikzlibrary{shapes}
\usetikzlibrary{patterns}
\usepackage{circuitikz}

\makeatletter
\edef\texforht{TT\noexpand\fi
  \@ifpackageloaded{tex4ht}
    {\noexpand\iftrue}
    {\noexpand\iffalse}}
\makeatother

\makeatletter
\newif\iftikz@node@phantom
\tikzset{
  phantom/.is if=tikz@node@phantom,
  text/.code=%
    \edef\tikz@temp{#1}%
    \ifx\tikz@temp\tikz@nonetext
      \tikz@node@phantomtrue
    \else
      \tikz@node@phantomfalse
      \let\tikz@textcolor\tikz@temp
    \fi
}
\usepackage{etoolbox}
\patchcmd\tikz@fig@continue{\tikz@node@transformations}{%
  \iftikz@node@phantom
    \setbox\pgfnodeparttextbox\hbox{}
  \fi\tikz@node@transformations}{}{}
\makeatother

\newcommand{\tikzAngleOfLine}{\tikz@AngleOfLine}
\def\tikz@AngleOfLine(#1)(#2)#3{%
  \pgfmathanglebetweenpoints{%
    \pgfpointanchor{#1}{center}}{%
    \pgfpointanchor{#2}{center}}
  \pgfmathsetmacro{#3}{\pgfmathresult}%
}

\tikzset{ 
    vertexNodePlain/.style = {fill=none, shape=circle, inner sep=0pt, minimum size=2pt, text=none},
    vertexNodePlain/.default=white,
    vertexPlain/labels/.style = {
        vertexNode/.style={vertexNodePlain=##1},
        vertexLabel/.style={gray}
    },
    vertexPlain/nolabels/.style = {
        vertexNode/.style={vertexNodePlain=##1},
        vertexLabel/.style={text=none}
    },
    vertexPlain/.style = vertexPlain/#1,
    vertexPlain/.default=labels
}
\tikzset{
    vertexNodeNormal/.style = {fill=none, shape=circle, inner sep=0pt, minimum size=4pt, text=none},
    vertexNodeNormal/.default = blue,
    vertexNormal/labels/.style = {
        vertexNode/.style={vertexNodeNormal=##1},
        vertexLabel/.style={blue}
    },
    vertexNormal/nolabels/.style = {
        vertexNode/.style={vertexNodeNormal=##1},
        vertexLabel/.style={text=none}
    },
    vertexNormal/.style = vertexNormal/#1,
    vertexNormal/.default=labels
}
\tikzset{
    vertexNodeBallShading/pdf/.style = {ball color=#1},
    vertexNodeBallShading/svg/.style = {fill=#1},
    vertexNodeBallShading/.code = {
        \if\texforht
            \tikzset{vertexNodeBallShading/svg=white}
        \else
            \tikzset{vertexNodeBallShading/pdf=white}
        \fi
    },
    vertexNodeBall/.style = {shape=circle, vertexNodeBallShading=#1, inner sep=2pt, outer sep=0pt, minimum size=3pt, font=\tiny},
    vertexNodeBall/.default = white,
    vertexBall/labels/.style = {
        vertexNode/.style={vertexNodeBall=##1, text=black},
        vertexLabel/.style={text=none}
    },
    vertexBall/nolabels/.style = {
        vertexNode/.style={vertexNodeBall=##1, text=none},
        vertexLabel/.style={text=none}
    },
    vertexBall/.style = vertexBall/#1,
    vertexBall/.default=labels
}
\tikzset{ 
    vertexStyle/.style={vertexNormal=#1},
    vertexStyle/.default = labels
}

\newcommand{\vertexLabelR}[4][]{
    \ifthenelse{ \equal{#1}{} }
        { \node[vertexNode] at (#2) {#4}; }
        { \node[vertexNode=#1] at (#2) {#4}; }
    \node[vertexLabel, #3] at (#2) {#4};
}
\newcommand{\vertexLabelA}[4][]{
    \ifthenelse{ \equal{#1}{} }
        { \node[vertexNode] at (#2) {#4}; }
        { \node[vertexNode=#1] at (#2) {#4}; }
    \node[vertexLabel] at (#3) {#4};
}

\newcommand{\edgeLabelColor}{blue!20!white}
\tikzset{
    edgeLineNone/.style = {draw=none},
    edgeLineNone/.default=black,
    edgeNone/labels/.style = {
        edge/.style = {edgeLineNone=##1},
        edgeLabel/.style = {fill=\edgeLabelColor,font=\small}
    },
    edgeNone/nolabels/.style = {
        edge/.style = {edgeLineNone=##1},
        edgeLabel/.style = {text=none}
    },
    edgeNone/.style = edgeNone/#1,
    edgeNone/.default = labels
}
\tikzset{
    edgeLinePlain/.style={line join=round, draw=#1},
    edgeLinePlain/.default=black,
    edgePlain/labels/.style = {
        edge/.style={edgeLinePlain=##1},
        edgeLabel/.style={fill=\edgeLabelColor,font=\small}
    },
    edgePlain/nolabels/.style = {
        edge/.style={edgeLinePlain=##1},
        edgeLabel/.style={text=none}
    },
    edgePlain/.style = edgePlain/#1,
    edgePlain/.default = labels
}
\tikzset{
    edgeLineDouble/.style = {very thin, double=#1, double distance=.8pt, line join=round},
    edgeLineDouble/.default=gray!90!white,
    edgeDouble/labels/.style = {
        edge/.style = {edgeLineDouble=##1},
        edgeLabel/.style = {fill=\edgeLabelColor,font=\small}
    },
    edgeDouble/nolabels/.style = {
        edge/.style = {edgeLineDouble=##1},
        edgeLabel/.style = {text=none}
    },
    edgeDouble/.style = edgeDouble/#1,
    edgeDouble/.default = labels
}
\tikzset{
    edgeStyle/.style = {edgePlain=#1},
    edgeStyle/.default = labels
}

\newcommand{\faceColorY}{yellow!60!white}   
\newcommand{\faceColorB}{blue!60!white}     
\newcommand{\faceColorC}{cyan!60}           
\newcommand{\faceColorR}{red!60!white}      
\newcommand{\faceColorG}{green!60!white}    
\newcommand{\faceColorO}{orange!50!yellow!70!white} 

\newcommand{\faceColor}{\faceColorY}
\newcommand{\faceColorSwap}{\faceColorC}

\tikzset{
    face/.style = {fill=#1},
    face/.default = \faceColor,
    faceY/.style = {face=\faceColorY},
    faceB/.style = {face=\faceColorB},
    faceC/.style = {face=\faceColorC},
    faceR/.style = {face=\faceColorR},
    faceG/.style = {face=\faceColorG},
    faceO/.style = {face=\faceColorO}
}
\tikzset{
    faceStyle/labels/.style = {
        faceLabel/.style = {}
    },
    faceStyle/nolabels/.style = {
        faceLabel/.style = {text=none}
    },
    faceStyle/.style = faceStyle/#1,
    faceStyle/.default = labels
}
\tikzset{ face/.style={fill=#1} }
\tikzset{ faceSwap/.code=
    \ifdefined\swapColors
        \tikzset{face=\faceColorSwap}
    \else
        \tikzset{face=\faceColor}
    \fi
}

\DeclareMathOperator{\Aut}{Aut}
\DeclareMathOperator{\Pow}{Pow}

\title{Strong Embeddings of Regular Graphs with Prescribed Automorphism Groups}
\author{Reymond Akpanya, Tom Goertzen, Meike Weiß}
\date{}

\begin{document}
\maketitle
\begin{abstract}
A classical theorem of Frucht states that every finite group occurs as the automorphism group of a finite graph. We prove an embedded analogue for regular
graphs of arbitrary degree. In particular, we show that for every $d\geq 3$ and every finite group $G$, there exists a $d$-regular graph $\Gamma$ with a strong embedding $\beta$ such that
$
\mathrm{Aut}(\Gamma) \cong \mathrm{Aut}(\beta(\Gamma)) \cong G.
$
Further, we prove that for every such $d$ and $G$ there exists a sequence of $d$-regular graphs with corresponding strong embeddings whose genera form an unbounded sequence and whose automorphism groups are isomorphic to $G$. Along the way, we identify an oversight in Sabidussi's classical construction of regular graphs with prescribed automorphism group. We give an alternative construction that corrects this issue and strengthens Sabidussi's result by producing an automorphism group-invariant proper $d$-edge-colouring.
\end{abstract}

\section{Introduction}

    A classical realisation problem asks whether every finite group can be realised as the automorphism group of a mathematical structure of a given class. We focus on this question in the context of graphs that admit embeddings on closed surfaces, i.e., compact 2-dimensional manifolds without boundary.
In particular, we are interested in graph embeddings that are strong, namely embeddings in which the given graph is drawn on a surface so that every induced cell is bounded by a cycle in the graph.
In order to provide a precise framework for our main investigation, we follow the standard terminology of (topological) graph theory as presented in \cite{TopologicalGraphTheory,GraphsOnSurfaces}. Unless stated otherwise, all graphs in this study are assumed to be undirected, connected, simple, and finite.

The construction of graphs and graph embeddings with prescribed automorphism groups has gained a lot of interest in graph theoretical studies. A fundamental observation in this context is that, given a group $G$ together with a generating set $S$, the corresponding directed edge-coloured Cayley graph $\mathrm{Cay}(G,S)$ has a colour-preserving automorphism group that is isomorphic to $G$. This observation has been widely used as a general construction principle: by suitably choosing $S$ and modifying the Cayley graph $\mathrm{Cay}(G,S)$, one can obtain graphs with desired properties while preserving $G$ as the automorphism group. In \cite{MR1557026}, Frucht exploits this fact to construct a graph whose automorphism group is a prescribed finite group. 
Frucht later extended this result by providing a construction of cubic graphs with prescribed automorphism groups \cite{Frucht}. The original construction contained an oversight which Frucht noticed in \cite{FruchtHow}. In recent work of the first two authors, this issue is corrected by modifying Frucht's construction, see \cite{SurfacesWithAuto}.
An alternative construction of a cubic graph with a given finite group as automorphism group introduced by Babai is provided in \cite{babaiBook}. Further, \v{S}ir\'{a}\v{n} and \v{S}koviera show that for every finite group $G$ there exists a cubic graph that admits an embedding on a surface such that the automorphism group of the resulting embedded graph is isomorphic to $G,$ see \cite{MR1211265}.
For more studies on the construction of combinatorial objects with prescribed automorphism group, we refer the reader to \cite{Cori, MR316298,MR406855,MR4866595}.

In \cite{Sabidussi}, Sabidussi studies several generalisations of Frucht’s theorem. As one of many results, Sabidussi proposes a construction of a $d$-regular graph, where $d\geq 3$ is a natural number, whose automorphism group is isomorphic to an arbitrary finite group $G$.
The present work concentrates on this construction of $d$-regular graphs with prescribed automorphism group.
While the overall strategy presented by Sabidussi is very elegant, a closer look suggests that certain aspects of the presented arguments need additional clarification, see Remark~\ref{sabidussi:construction}.
 Here, we correct this oversight by extending Sabidussi's result in the following way.

\begin{theorem*}[\ref{theorem:dregular}]
        For every finite group $G$ and every $d\geq 3$ there exists a $d$-regular graph $\Gamma$ and a proper $d$-edge-colouring $\kappa:E(\Gamma)\to \{1,\ldots,d\}$ of $\Gamma$ such that $\Aut(\Gamma)\cong G$ and $\kappa$ is $\Aut(\Gamma)$-invariant, i.e.,  $\kappa(e)=\kappa(\phi(e))$ for all $e\in E(\Gamma)$ and $\phi\in \Aut(\Gamma).$ 
\end{theorem*}

We establish this extension of \cite[Theorem~3.7]{Sabidussi} by combining the proof strategies of Sabidussi in \cite{Sabidussi} and Frucht in \cite{Frucht}. Following Frucht’s approach, we distinguish three cases: (1) $|G| \leq 2$, (2) $G \cong C_\ell$ for $\ell \geq 3$, where $C_\ell$ denotes the cyclic group of order $\ell$, and (3) $G$ is $n$-generated with $n\geq 2$. 
In particular, we adopt Sabidussi’s use of graph products and replace Frucht’s cubic graph construction in Sabidussi's original argument with the construction from \cite{SurfacesWithAuto}, refining the arguments where necessary.

Additionally, in \cite{SurfacesWithAuto}, the first two authors of this work exploit the concept of automorphism group–invariant edge-colourings to show that, for every finite group $ G $, there exists a cubic graph $ \Gamma $ together with a strong embedding $ \beta $ such that both $ \Aut(\Gamma) $ and the automorphism group of the embedded graph $ \beta(\Gamma) $ are isomorphic to $ G $. This naturally leads to the question whether the ideas introduced by Sabidussi in \cite{Sabidussi,graphmult} can be adapted to extend this result. In this work, we address this question by establishing the following main results.

\begin{theorem*}[\ref{theorem:autom_strongEmbedding}]
    For every finite group $G$ and every $d\geq 3$ there exists a $d$-regular graph $\Gamma$ and a strong embedding $\beta$ of $\Gamma$ on some surface such that $\Aut(\Gamma)\cong \Aut(\beta(\Gamma))\cong G.$
\end{theorem*}

We obtain this result by exploiting the fact that $ d $-regular graphs admitting proper automorphism group–invariant edge-colourings admit symmetric strong embeddings. This observation then enables us to establish the following statement.

\begin{theorem*}[\ref{thm:infinite}]
    For every finite group $G$ and every $d\geq 3$ there exists a sequence of $d$-regular graphs $(\Gamma_\ell)_{\ell \in \mathbb{N}}$ with corresponding strong embeddings $(\beta_\ell)_{\ell\in \mathbb{N}}$ such that 
\begin{enumerate}
\item $\Aut(\Gamma_\ell)\cong\Aut(\beta_\ell(\Gamma_\ell))\cong G$ for all $\ell\in \mathbb{N},$
    \item $\lim_{\ell\to \infty} g(\beta_\ell(\Gamma_\ell)) =\infty.$
\end{enumerate}
\end{theorem*}

The paper is organised as follows: In \Cref{sec:prel}, we introduce the necessary notions on graphs and their embeddings that are used throughout this work. In \Cref{sec:trunc}, we study truncations of graphs with particular emphasis on their role in the existence of strong embeddings and proper edge-colourings. 
In \Cref{sec:regular}, we revisit Sabidussi’s theorem and establish our extension (see Theorem~\ref{theorem:dregular}). Moreover, we present one of our main results, namely Theorem~\ref{theorem:autom_strongEmbedding}, showing that for every finite group $ G $, there exists a $ d $-regular graph together with a corresponding strong embedding such that both the automorphism group of the graph and that of the embedded graph are isomorphic to $ G $. 
Finally, we show that for every finite group $ G$, there exists an infinite sequence of $d$-regular graphs and corresponding strong embeddings whose automorphism groups are isomorphic to $ G $, and whose genera tend to infinity, see Theorem \ref{thm:infinite}.

Implementations of the construction in Magma~\cite{Bosma} and GAP~\cite{GAP4}, together with scripts for verifying the resulting automorphism groups on examples from the SmallGroups Library~\cite{SmallGroupPaper,MR1935567}, are available in~\cite{AkpanyaKregularSymmetricGraphs}.

\section{Preliminaries}\label{sec:prel}

We begin our investigation by introducing some preliminary notions on graphs and graph embeddings. As mentioned above, we refer the reader to \cite{TopologicalGraphTheory, GraphsOnSurfaces, graphmult}, for more details on graph embeddings and graph products.

Let $\Gamma$ be a graph with vertices and edges denoted by $V(\Gamma)$ and $E(\Gamma)$, respectively. Since we assume $\Gamma$ to be simple, we define $E(\Gamma)$ to be a subset of $ \Pow_2(V(\Gamma)).$ For a vertex $v$, we denote its \emph{vertex degree} by $\deg_\Gamma(v).$ If $\Gamma$ is clear from the context, we simply write $\deg(v).$ We say that $\Gamma$ is $d$-\emph{regular}, if all vertex degrees are equal to $d.$ In the case that $\Gamma$ is $3$-regular, we say that $\Gamma$ is \emph{cubic}. Furthermore, we denote the \emph{automorphism group} of $\Gamma$ by $\Aut(\Gamma)$.
A \emph{cycle} in $\Gamma$ is a sequence $
(v_1,e_1,v_2,e_2,v_3, \dots, v_\ell, e_\ell)$ of distinct vertices and edges in $\Gamma$ such that $v_i \in e_{i-1}\cap e_i$ holds for all $1\leq i \leq \ell$ where we identify the edges $e_0$ and $e_\ell$.
Since $\Gamma$ is simple, we can identify the above cycle by the sequence of its vertices, namely $(v_1,\ldots,v_\ell).$ 
 A map $\kappa:E(\Gamma)\to \{1,\ldots,d\}$ is said to be a \emph{proper $d$-edge-colouring} of $\Gamma$, 
if the restriction of $\kappa$ to the edges that are incident to any vertex $v\in V(\Gamma)$ is injective. In the above case, we say that $\Gamma$ admits a proper $d$-edge-colouring. From now on we always assume an edge-colouring to be proper and thus only say edge-colouring meaning proper edge-colouring. We refer to the edge-colouring $\kappa$ as $\Aut(\Gamma)$-invariant, if $\kappa(e)= \kappa(\phi(e))$ for all $e\in E(\Gamma)$ and $\phi\in \Aut(\Gamma).$

Next, we recall the notion of the (Cartesian) product of graphs from \cite{graphmult}. For this, let $\Gamma_1$ and $\Gamma_2$ be two graphs. We define the (Cartesian) product of $\Gamma_1$ and $\Gamma_2$ as the graph $\Gamma_1\times \Gamma_2$ whose vertex set is given by $V(\Gamma_1)\times V(\Gamma_2)$ and two vertices $(x_1,y_1),(x_2,y_2)\in V(\Gamma_1\times \Gamma_2)$ satisfy $\{(x_1,y_1),(x_2,y_2)\}\in E(\Gamma_1\times \Gamma_2)$ if and only if (1) $x_1 = x_2$ and $\{y_1,y_2\}\in E(\Gamma_2)$ or (2) $y_1 = y_2$ and $\{x_1,x_2\}\in E(\Gamma_1).$ 
We observe that several properties of $\Gamma_1$ and $\Gamma_2$ extend to properties of $\Gamma_1\times \Gamma_2$. For instance, $\Gamma_1\times \Gamma_2$ is connected, if and only if $\Gamma_1$ and $\Gamma_2$ are connected. Further, if $\Gamma_1$ is $d_1$-regular and $\Gamma_2$ is $d_2$-regular, then $\Gamma_1 \times \Gamma_2$ is $(d_1+d_2)$-regular. Additionally, if $\Gamma_1\times \Gamma_2$ is regular, then it follows that $\Gamma_1$ and $\Gamma_2$ have to be regular.
Graph products allows us to introduce the notion of a prime graph.
For this, we define the \emph{trivial graph} $U$ as the graph consisting of a single vertex with no edges. We say that a graph is \emph{non-trivial}, if it is not isomorphic to the trivial graph $U$. Furthermore, a graph $\Gamma$ is called \emph{prime} if $\Gamma$ is non-trivial and if $\Gamma \cong \Gamma_1 \times \Gamma_2$ implies directly $\Gamma_1 \cong U \text{ or } \Gamma_2 \cong U $.
Two graphs $\Gamma_1$ and $\Gamma_2$ are called \emph{relatively prime} if and only if for all graphs $\Gamma_1', \Gamma_2', Z$ such that $\Gamma_1 \cong \Gamma_1' \times Z$ and $\Gamma_2 \cong \Gamma_2' \times Z$, it follows that $Z \cong U$.

In \cite[Lemma 2.8]{Sabidussi} Sabidussi proved that every graph $\Gamma$ containing a vertex or an edge which is not contained in a cycle of length four of $\Gamma$ is prime. By using a similar proof, we give the following result.

\begin{corollary}\label{cor:prime}
Let $\Gamma$ be a $d$-regular graph. If $\Gamma$ contains a vertex that is contained in fewer than $d-1$ cycles of length four, then $\Gamma$ is prime. 
\end{corollary}
\begin{proof}
    Let $\Gamma$ be a graph that is not prime. That means that there exist graphs $\Gamma_1$ and $\Gamma_2$ such that $\Gamma= \Gamma_1\times \Gamma_2.$ Since $\Gamma$ is regular, there exists $d_1,d_2\in \mathbb{N}$ such that $\Gamma_1$ is $d_1$-regular and $\Gamma_2$ is $d_2$-regular. Furthermore, let $\{x_1,x_2\}$ be an edge in $E(\Gamma_1)$ and $\{y_1,y_2\}$ an edge in $E(\Gamma_2)$. Thus, we see that $$\gamma=((x_1,y_1),(x_1,y_2),(x_2,y_2),(x_2,y_1))$$ forms a cycle of length four in $\Gamma$. Since every choice of neighbours of the vertices $x_1\in V(\Gamma_1)$ and $y_1\in V(\Gamma_2)$ in the above argument gives us a cycle of length four in $\Gamma=\Gamma_1\times \Gamma_2,$ we obtain that the number of cycles of length four that are incident to the vertex $(x_1,y_1)$ is greater or equal to $d_1\cdot d_2.$ Since 
    ${d_1\cdot d_2}\geq d-1$, the result follows.
\end{proof}

In the case that $\Gamma_1$ and $\Gamma_2$ are two arbitrary graphs, we know that  $\Aut(\Gamma_1)\times \Aut(\Gamma_2)$ can be embedded into $\Aut(\Gamma_1\times \Gamma_2).$ In \cite[Theorem 3.1]{graphmult}, Sabidussi proves that $$\Aut(\Gamma_1)\times \Aut(\Gamma_2)\cong \Aut(\Gamma_1\times \Gamma_2)$$ holds, if $\Gamma_1$ and $\Gamma_2$ are relatively prime. Note that this result can also be found in \cite[Corollary 4.17]{MR1788124}. However, for graphs $\Gamma_1$ and $\Gamma_2$ that are not relatively prime, these two groups are not necessarily isomorphic.

Lastly, we introduce embeddings of graphs on topological surfaces. An \emph{embedding} of $\Gamma$ on a surface $S$ is an injective and continuous map $\beta :\Gamma\rightarrow S$.  The \emph{cells} of $\beta$ are the connected components of $S\setminus\beta(\Gamma)$. 
We say that $\beta$ is a \emph{$2$-cell} embedding, if the cells of $\beta$ are all homeomorphic to open discs. The boundaries of the cells can be described by closed walks in $\Gamma$ and are called \emph{facial walks}.
The \emph{Euler characteristic} of an embedded graph $\beta(\Gamma)$ is defined as $\chi(\beta(\Gamma))=\vert V(\Gamma)\vert-\vert E(\Gamma)\vert+\vert F(\beta(\Gamma))\vert$, where $F(\beta(\Gamma))$ denotes the set of facial walks of the embedding $\beta$. The genus of $\beta(\Gamma)$ is defined in the usual way.
We refer to the embedding $\beta$ as \emph{strong}, if $\beta$ is a 2-cell embedding and all facial walks are cycles in $\Gamma$. A strong embedding of $\Gamma$ gives rise to a cycle double cover of $\Gamma$, see \cite{szekeres, seymour}.
A \emph{cycle double cover} $\mathcal{C}$ of $\Gamma$ is a set of cycles in $\Gamma$ such that each edge $e\in E(\Gamma)$ is contained in exactly two cycles in $\mathcal{C}$.
A cycle double cover is called \emph{facial} if the contained cycles describe the boundaries of the 2-cells of a strong embedding. Note that if $\Gamma$ is a cubic graph, every cycle double cover of $\Gamma$ is facial; however, this does not hold for arbitrary graphs. 
A strong embedding and therefore also its corresponding facial cycle double cover gives rise to a \emph{local rotation} at $v\in V(\Gamma)$, i.e., a cyclic ordering $\rho_v:=(e_1,\ldots,e_\ell)$ of the edges incident to $v$, see \cite{TopologicalGraphTheory,GraphsOnSurfaces}. Lastly, the \emph{automorphism group} of the embedded graph $\beta(\Gamma)$ is the maximal subgroup of $\Aut(\Gamma)$ that preserves the cells of $\beta(\Gamma)$, i.e., 
$$\Aut(\beta (\Gamma)) =\{\phi \in \Aut(\Gamma)\mid \{\phi(f)\mid f \in F(\beta( \Gamma))\} =F(\beta(\Gamma))\}.$$

\section{Truncation of regular embedded graphs}\label{sec:trunc}

In order to establish our main results in Section~\ref{sec:regular} and Section~\ref{sec:genus}, we make use of truncations of regular graphs that are embedded on surfaces via strong embeddings. Here, we define truncations of a given graph by employing the local rotations of its corresponding vertices. Additionally, we make several preliminary remarks that will be useful in this work.

Let therefore $\Gamma$ be a $d$-regular graph with $d\geq 3,$ and $\beta$ a strong embedding of $\Gamma$. Moreover, let $v\in V(\Gamma)$ be a vertex and $e_1=\{v,v_1\},\ldots,e_d=\{v,v_d\}\in E(\Gamma)$ the edges incident to $v$ such that $\rho_v=(e_1,\ldots,e_d)$ is the local rotation at $v$ defined by $\beta$. 
With this, we define the graph $\Gamma^{(v,\beta)}$ via the vertices 
    \begin{align*}
        V(\Gamma^{(v,\beta)}):=&(V(\Gamma)\setminus \{v\})\,\cup\, \{w_{(v,e_1)},\ldots,w_{(v,e_d)}\},
        \end{align*}
with $w_{(v,e_i)}\notin V(\Gamma)$ for $1\leq i\leq d$, and edges 
    \begin{align*}
        E(\Gamma^{(v,\beta)}):=&(E(\Gamma)\setminus \{e_1,\ldots,e_d\})\,\cup\, \bigcup_{i=1}^{d}\{\{w_{(v,e_i)},w_{(v,e_{i+1})}\},\{v_{i},w_{(v,e_{i})}\}\}
    \end{align*}
with $w_{(v,e_{d+1})}:=w_{(v,e_1)}$.
We say that $\Gamma^{(v,\beta)}$ is obtained from applying a \emph{truncation} to the vertex $v$ in $\Gamma$. Note that $\Gamma^{(v,\beta)}$ is no longer a $d$-regular graph if $d>3$. In the proposition below, we establish that the strong embedding $\beta$ of $\Gamma$ induces a strong embedding of $\Gamma^{(v,\beta)}$.

\begin{proposition}\label{prop:embedding_trunc}
Let $\Gamma$ be a $d$-regular graph, where $d\geq 3$, $\beta$ a strong embedding of $\Gamma$ and $v\in V(\Gamma)$ a vertex. Then $\Gamma':=\Gamma^{(v,\beta)}$ has a strong embedding $\beta'$ satisfying 
    $\chi(\beta(\Gamma))=\chi(\beta'(\Gamma'))$.
\end{proposition}

\begin{proof}
Let $\mathcal{C}$ be the facial cycle double cover of $\Gamma$ that corresponds to the strong embedding $\beta$.
Moreover, let $e_1=\{v,v_1\},\ldots,e_d=\{v,v_d\}\in E(\Gamma)$ be exactly the edges that are incident to $v$ such that $(e_1,\ldots,e_d)$ forms the local rotation at $v$ induced by $\beta$. Thus, for every $1\leq i \leq d$ the cycle double cover $\mathcal{C}$ contains a cycle $\gamma_i$ of the form $(\ldots,v_i,e_i,v, e_{i+1},v_{i+1},\ldots)$ with $v_{d+1}:=v_1$. 
Since a cycle in $\Gamma$ is determined by the ordered sequence of its vertices, there exist suitable vertices in $\Gamma$ such that $\gamma_i$ can be written as $\gamma_i=(a_i,\ldots,\bar{a}_i, v_i,v,v_{i+1},\bar{b}_i,\ldots,b_i)$.

We use this information to construct a cycle double cover of $\Gamma'$ as follows: First, we translate the vertex $v$ in $V(\Gamma)$ into the cycle $\gamma_v:=(w_{(v,e_1)},\ldots,w_{(v,e_d)})$ in $\Gamma'$.
Moreover, for $1\leq i \leq d$ the cycle $\gamma_i\in \mathcal{C}$ can be translated into the cycle $\gamma'_i$ in $\Gamma'$ defined by $$(a_i,\ldots,\bar{a}_i,v_i,w_{(v,e_i)},w_{(v,e_{i+1})},v_{i+1},\bar{b}_i,\ldots,b_i).$$ 
So, the set $$\mathcal{C}':=\{\gamma_v\}\cup \{\gamma_i'\mid 1\leq i \leq d\} \cup (\mathcal{C}\setminus \{\gamma_i\mid 1\leq i \leq d\})$$ forms a cycle double cover of $\Gamma',$ see Figure~\ref{fig:trunc} for an illustration. Note that cycles that do not pass through the truncated vertex $v$ are kept unchanged.
Since $\mathcal{C}$ is a facial cycle double cover, we deduce that $\mathcal{C}'$ is facial. This follows as described below:
If we consider a vertex $v'\in V(\Gamma)$ with $v'\neq v$ as a vertex in $V(\Gamma'),$ then we see that the rotation system at $v'$ with respect to $\beta$ corresponds to the rotation system at $v'$ with respect to $\beta'.$
Furthermore, at each new vertex $w_{(v,e_i)}$ the cyclic order is induced by the rotation system at $v$ with respect to $\beta$. Consequently, the cycles in $\mathcal{C}'$ are precisely the facial walks of the induced rotation system, and hence $\mathcal{C}'$ is a facial cycle double cover.
Hence, the cycle double cover $\mathcal{C}'$ of $\Gamma'$ constructed above defines a strong embedding $\beta'$ of $\Gamma'.$
Finally, we obtain:
\begin{align*}
    \chi(\beta'(\Gamma'))=\vert V(\Gamma')\vert-\vert E(\Gamma')\vert + \vert \mathcal{C}' \vert
= (\vert V(\Gamma)\vert +d-1 )- (\vert E(\Gamma)\vert +d) +(\vert \mathcal{C}\vert +1)= \chi(\beta(\Gamma)).
\end{align*}
\end{proof}

\begin{figure}[H]
    \centering
    \begin{tikzpicture}[vertexBall, edgeDouble=nolabels, scale=2]

\coordinate (V1) at (0 , 0);
\coordinate (V2) at (1 , 0);
\coordinate (V3) at (0 , 1);
\coordinate (V4) at (-1 , 0);
\coordinate (V5) at (0, -1);

\draw[-,very thick] (V1) to (V2);
\draw[-,very thick] (V1) to (V3);
\draw[-,very thick] (V1) to (V4);
\draw[-,very thick] (V1) to (V5);

\vertexLabelR[]{V1}{left}{$ $}

\draw[-{To[scale=2]}] (1.3,0) to (1.8,0);

\coordinate (V12) at (0.5+3 , 0);
\coordinate (V22) at (0 +3, 0.5);
\coordinate (V32) at (-0.5+3 , 0);
\coordinate (V42) at (3, -0.5);

\coordinate (V13) at (1 +3, 0);
\coordinate (V23) at (3, 1);
\coordinate (V33) at (-1 +3, 0);
\coordinate (V43) at (3, -1);

\draw[-,very thick,red] (V12) to (V13);
\draw[-,very thick,red] (V22) to (V23);
\draw[-,very thick,red] (V32) to (V33);
\draw[-,very thick,red] (V42) to (V43);

\draw[-,very thick] (V12) to (V22);
\draw[-,very thick] (V22) to (V32);
\draw[-,very thick] (V32) to (V42);
\draw[-,very thick] (V42) to (V12);

 \draw[very thick, blue, smooth ]
  plot coordinates {
    (0.1,1)
    (0.2,0.2)
    (1,0.1)
  };

\draw[very thick, green!50!black, smooth ]
  plot coordinates {
    (-0.1,1)
    (-0.2,0.2)
    (-1,0.1)
  };

  \draw[very thick, orange, smooth ]
  plot coordinates {
    (-0.1,-1)
    (-0.2,-0.2)
    (-1,-0.1)
  };

  \draw[very thick, Fuchsia, smooth ]
  plot coordinates {
    (0.1,-1)
    (0.2,-0.2)
    (1,-0.1)
  };

\vertexLabelR[]{V12}{left}{$ $}
\vertexLabelR[]{V22}{left}{$ $}
\vertexLabelR[]{V32}{left}{$ $}
\vertexLabelR[]{V42}{left}{$ $}

\draw[very thick, black!50!white, smooth cycle ]
  plot coordinates {
    (0.35+3,0)
    (3,-0.35)
    (-0.35+3,0)
    (3,0.35)
  };

 \draw[very thick, blue, smooth ]
  plot coordinates {
    (0.1+3,1)
    (0.2+3,0.5)
        (0.5+3,0.1+0.1)
    (1+3,0.1)
  };

\draw[very thick, green!50!black, smooth ]
  plot coordinates {
    (-0.1+3,1)
    (-0.2+3,0.5)
    (-0.5+3,0.1+0.1)
    (-1+3,0.1)
  };

  \draw[very thick, orange, smooth ]
  plot coordinates {
    (-0.1+3,-1)
        (-0.2+3,-0.5)
    (-0.5+3,-0.2)
    (-1+3,-0.1)
  };

    \draw[very thick, Fuchsia, smooth ]
  plot coordinates {
    (0.1+3,-1)
    (0.2+3,-0.5)
    (0.5+3,-0.2)
    (1+3,-0.1)
  };

\node[vertexBall,shift={(-0.25,0.25)}] at (V1) {$v$};

\node[vertexBall,shift={(-1.5,0.4)}] at (V1) {$e_4$};

\node[vertexBall,shift={(1.5,-0.4)}] at (V1) {$e_2$};

\node[vertexBall,shift={(0.4,1.5)}] at (V1) {$e_1$};

\node[vertexBall,shift={(-0.4,-1.5)}] at (V1) {$e_3$};

\node[vertexBall,shift={(0.3,-.75)}] at (V12) {$w_{(v,e_2)}$};

\node[vertexBall,shift={(1.,0)}] at (V22) {$w_{(v,e_1)}$};

\node[vertexBall,shift={(-.4,0.65)}] at (V32) {$w_{(v,e_4)}$};

\node[vertexBall,shift={(-1.,0)}] at (V42) {$w_{(v,e_3)}$};

\end{tikzpicture} 
    \caption{Truncating a vertex $v$ of degree 4 with its incident edges $e_1,\ldots,e_4$ in $\Gamma$ resulting in four vertices $w_{(v,e_1)},\ldots,w_{(v,e_4)}$ in $\Gamma'$ and the corresponding facial cycle double cover that induces a strong embedding with the same Euler characteristic.}
    \label{fig:trunc}
\end{figure}
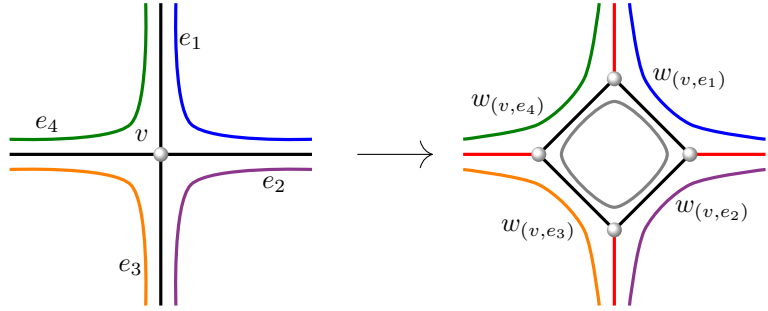
In the setting of the above proof, we see that $\beta$ and $\beta'$ describe strong embeddings of $\Gamma$ and $\Gamma'$ onto the same surface.

Next, we exploit the truncation of a vertex of $\Gamma$ and the above proposition to define the complete truncation of $\Gamma$.
For this, we assume that the vertices of $\Gamma$ are given by $V(\Gamma)=\{v_1,\ldots,v_k\}$. We first set $\Gamma_0:=\Gamma$ and furthermore assume that $\Gamma_0$ has a strong embedding $\beta_0:=\beta$. With this, we construct graphs $\Gamma_1,\ldots,\Gamma_k$ recursively as follows: For $0\leq i\leq k-1$ let the graph $\Gamma_i$  together with a strong embedding $\beta_i$ already be constructed as described in the proof of Proposition~\ref{prop:embedding_trunc}.
Now, the vertex $v_{i+1}$ can be interpreted as a vertex of the graph $\Gamma_i$.
Thus, we define the graph $\Gamma_{i+1}$ via  $\Gamma_{i+1}:={\Gamma_i}^{(v_{i+1},\beta_i)}$ with corresponding strong embedding $\beta_{i+1}$ obtained from $\beta_i$ with Proposition~\ref{prop:embedding_trunc}.
We call $\mathcal{T}(\Gamma, \beta):= \Gamma_k$ the \emph{complete truncation} of $\Gamma.$
In \Cref{subfig:completetrunc2} we illustrate the graph resulting from a complete truncation of the cubical graph induced by its planar embedding illustrated in \Cref{subfig:completetrunc1}.
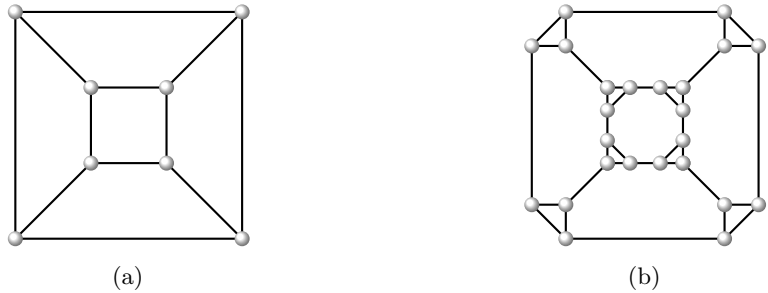
\begin{figure}[H]
    \begin{subfigure}{.45\textwidth}
        \centering
        \begin{tikzpicture}[vertexBall, edgeDouble=nolabels, faceStyle=nolabels, scale=3]

\coordinate (V1) at (0.,0);
\coordinate (V2) at (0 , 1.);
\coordinate (V3) at (1,1);
\coordinate (V4) at (1,0);
\coordinate (V5) at (1/3,1/3);
\coordinate (V6) at (1/3,2/3);
\coordinate (V7) at (2/3,2/3);
\coordinate (V8) at (2/3,1/3);
\draw[thick] (V1) -- (V2);
\draw[thick] (V1) -- (V5);
\draw[thick] (V1) -- (V4);
\draw[thick] (V2) -- (V3);
\draw[thick] (V2) -- (V6);
\draw[thick] (V3) -- (V4);
\draw[thick] (V3) -- (V7);
\draw[thick] (V4) -- (V8);
\draw[thick] (V5) -- (V6);
\draw[thick] (V5) -- (V8);
\draw[thick] (V6) -- (V7);
\draw[thick] (V7) -- (V8);

\vertexLabelR[]{V1}{left}{$ $}
\vertexLabelR[]{V2}{left}{$ $}
\vertexLabelR[]{V3}{left}{$ $}
\vertexLabelR[]{V4}{left}{$ $}
\vertexLabelR[]{V5}{left}{$ $}
\vertexLabelR[]{V6}{left}{$ $}
\vertexLabelR[]{V7}{left}{$ $}
\vertexLabelR[]{V8}{left}{$ $}

\end{tikzpicture}
        \subcaption{}
        \label{subfig:completetrunc1}
    \end{subfigure}
    \begin{subfigure}{.45\textwidth}
        \centering
        \begin{tikzpicture}[vertexBall, edgeDouble=nolabels, faceStyle=nolabels, scale=3.]

\coordinate (V1) at (0.,0);
\coordinate (V11) at (0.15,0);
\coordinate (V12) at (0,0.15);
\coordinate (V13) at (0.15,0.15);

\coordinate (V2) at (0 , 1.);
\coordinate (V21) at (0.15 , 1.);
\coordinate (V22) at (0 , 0.85);
\coordinate (V23) at (0.15 , 0.85);

\coordinate (V3) at (1,1);
\coordinate (V31) at (0.85,1);
\coordinate (V32) at (1,0.85);
\coordinate (V33) at (0.85,0.85);

\coordinate (V4) at (1,0);
\coordinate (V41) at (0.85,0);
\coordinate (V42) at (1,0.15);
\coordinate (V43) at (0.85,0.15);

\coordinate (V51) at (1/3,1/3);
\coordinate (V52) at (1/3+0.1,1/3);
\coordinate (V53) at (1/3,1/3+0.1);

\coordinate (V61) at (1/3,2/3);
\coordinate (V62) at (1/3+0.1,2/3);
\coordinate (V63) at (1/3,2/3-0.1);

\coordinate (V71) at (2/3,2/3);
\coordinate (V72) at (2/3-0.1,2/3);
\coordinate (V73) at (2/3,2/3-0.1);

\coordinate (V81) at (2/3,1/3);
\coordinate (V82) at (2/3-0.1,1/3);
\coordinate (V83) at (2/3,1/3+0.1);

\draw[thick] (V12) -- (V22);
\draw[thick] (V13) -- (V51);
\draw[thick] (V11) -- (V41);
\draw[thick] (V21) -- (V31);
\draw[thick] (V23) -- (V61);
\draw[thick] (V32) -- (V42);
\draw[thick] (V33) -- (V71);
\draw[thick] (V43) -- (V81);
\draw[thick] (V73) -- (V83);
\draw[thick] (V72) -- (V62);
\draw[thick] (V53) -- (V63);
\draw[thick] (V52) -- (V82);
\draw[thick] (V12) -- (V11);
\draw[thick] (V13) -- (V12);
\draw[thick] (V13) -- (V11);

\draw[thick] (V22) -- (V21);
\draw[thick] (V23) -- (V22);
\draw[thick] (V23) -- (V21);

\draw[thick] (V32) -- (V31);
\draw[thick] (V33) -- (V32);
\draw[thick] (V33) -- (V31);

\draw[thick] (V42) -- (V41);
\draw[thick] (V43) -- (V42);
\draw[thick] (V43) -- (V41);

\draw[thick] (V52) -- (V51);
\draw[thick] (V53) -- (V52);
\draw[thick] (V53) -- (V51);

\draw[thick] (V62) -- (V61);
\draw[thick] (V63) -- (V62);
\draw[thick] (V63) -- (V61);

\draw[thick] (V72) -- (V71);
\draw[thick] (V73) -- (V72);
\draw[thick] (V73) -- (V71);

\draw[thick] (V82) -- (V81);
\draw[thick] (V83) -- (V82);
\draw[thick] (V83) -- (V81);

\vertexLabelR[]{V51}{left}{$ $}
\vertexLabelR[]{V61}{left}{$ $}
\vertexLabelR[]{V71}{left}{$ $}
\vertexLabelR[]{V81}{left}{$ $}

\vertexLabelR[]{V11}{left}{$ $}
\vertexLabelR[]{V12}{left}{$ $}
\vertexLabelR[]{V13}{left}{$ $}

\vertexLabelR[]{V21}{left}{$ $}
\vertexLabelR[]{V22}{left}{$ $}
\vertexLabelR[]{V23}{left}{$ $}

\vertexLabelR[]{V31}{left}{$ $}
\vertexLabelR[]{V32}{left}{$ $}
\vertexLabelR[]{V33}{left}{$ $}

\vertexLabelR[]{V41}{left}{$ $}
\vertexLabelR[]{V42}{left}{$ $}
\vertexLabelR[]{V43}{left}{$ $}

\vertexLabelR[]{V51}{left}{$ $}
\vertexLabelR[]{V52}{left}{$ $}
\vertexLabelR[]{V53}{left}{$ $}

\vertexLabelR[]{V61}{left}{$ $}
\vertexLabelR[]{V62}{left}{$ $}
\vertexLabelR[]{V63}{left}{$ $}

\vertexLabelR[]{V71}{left}{$ $}
\vertexLabelR[]{V72}{left}{$ $}
\vertexLabelR[]{V73}{left}{$ $}

\vertexLabelR[]{V81}{left}{$ $}
\vertexLabelR[]{V82}{left}{$ $}
\vertexLabelR[]{V83}{left}{$ $}

\end{tikzpicture}
        \subcaption{}
        \label{subfig:completetrunc2}
    \end{subfigure}
    \caption{(a) The planar embedding of the cubical graph and (b) its corresponding complete truncation.}
    \label{fig:completetrunc}
\end{figure}
Note that the complete truncation $\Gamma':=\mathcal{T}(\Gamma,\beta)$ of $\Gamma$ is a cubic graph and has vertex set $$V(\Gamma') =\{w_{(v,e)} \mid v\in V(\Gamma),e\in E(\Gamma)\text{ with }v\in e\}.$$
Furthermore, we observe that $\Gamma'$ has two types of edges, namely
\begin{enumerate}
    \item $\vert E(\Gamma)\vert$ edges of the form $\{w_{(v_1,e)}, w_{(v_2 ,e)}\}$, where $e\in E(\Gamma)$ is an edge with $e= \{v_1,v_2\}$, and
\item  $d\cdot \vert V(\Gamma)\vert$ edges of the form $\{w_{(v,e_1)},w_{(v,e_2)}\}$, where $e_1,e_2\in E(\Gamma)$ are edges satisfying $v\in e_1\cap e_2$ and $(e_1,v,e_2)$ is a subsequence that is contained in a cycle of the facial cycle double cover corresponding to $\beta.$
\end{enumerate}
We notice that since $\Gamma$ is simple, every cycle in $\Gamma'$ that does not correspond to a vertex in $\Gamma$ has length at least 6. For the complete truncation of cubic graphs, we give the following remark.
\begin{remark}
Let $\Gamma$ be a cubic graph and $\beta,\beta'$ two strong embeddings of $\Gamma$. Then we know that the graphs $\mathcal{T}(\Gamma,\beta)$ and $\mathcal{T}(\Gamma,\beta')$ are isomorphic. This follows from the fact that truncating a vertex of degree $3$ always produces a subgraph isomorphic to the complete graph $K_3$. Hence, we obtain the same incidences in $\mathcal{T}(\Gamma,\beta)$ and $\mathcal{T}(\Gamma,\beta')$ independently of the chosen strong embedding. However, this is no longer true for the corresponding graph embeddings induced by $\beta$ and $\beta'$ as described in Proposition~\ref{prop:embedding_trunc}. Therefore, when we are only concerned with the underlying graph and not with the induced embedding, we write $\mathcal{T}(\Gamma)$ instead of $\mathcal{T}(\Gamma,\beta)$ for cubic graphs. From the observation above, we can also apply the complete truncation to a cubic graph without any knowledge of a strong embedding. 
\end{remark}

In the case where $\Gamma$ is $d$-regular with $d\in\{3,4,5\}$ and admits a strong embedding $\beta$, we show that the automorphism group of $\mathcal{T}(\Gamma,\beta)$ can be recovered from the automorphism group of $\beta(\Gamma)$.

\begin{proposition}\label{prop:autiso}
Let $\Gamma$ be a $d$-regular graph with $3\le d\le 5$, and let $\beta$ be a strong embedding of $\Gamma$. Then $\Aut(\beta(\Gamma))\cong \Aut(\mathcal{T}(\Gamma,\beta)).$
\end{proposition}

\begin{proof}
First, let $\Gamma':=\mathcal{T}(\Gamma,\beta)$. We begin by defining a homomorphism $\Phi:\Aut(\beta(\Gamma))\longrightarrow \Aut(\Gamma').$
Let $\phi\in \Aut(\beta(\Gamma))$ be an automorphism. Since $\phi$ preserves the embedding, it maps vertices, edges, and faces of $\beta(\Gamma)$ to vertices, edges, and faces, respectively, and preserves the cyclic order of the
incident edges at each vertex, possibly reversing its orientation. For every vertex $w_{(v,e)}\in V(\Gamma')$, define
\[
\Phi(\phi)(w_{(v,e)})
   = w_{(\phi(v),\phi(e))}.
\]
The preservation of the local rotation (up to orientation) at each vertex implies that adjacent vertices of $\Gamma'$ are mapped to adjacent vertices of $\Gamma'$, so $\Phi(\phi)\in\Aut(\Gamma')$. If $\Phi(\phi_1)=\Phi(\phi_2)$, then $\phi_1$ and $\phi_2$ agree on all vertices and edges of $\Gamma$, and hence $\phi_1=\phi_2$. Therefore, $\Phi$ is injective.

It remains to show that $\Phi$ is surjective. Let $\psi\in\Aut(\Gamma')$. For each vertex $v\in V(\Gamma)$, the
truncation construction produces a cycle
\[
C_v=(w_{(v,e_1)},\ldots,w_{(v,e_d)})
\]
of length $d$ in $\Gamma'$.
Since $3\le d\le5$, every cycle in $\Gamma'$ that does not arise from a vertex of $\Gamma$ has length at least $6$. Hence, the cycles $C_v$ are precisely the cycles of length $d$ in $\Gamma'$. It follows that $\psi$ permutes the set
\[
\mathcal C=\{C_v\mid v\in V(\Gamma)\}.
\]
Consequently, for every $v\in V(\Gamma)$ there exists a unique vertex $\phi(v)\in V(\Gamma)$ such that
\[
\psi(C_v)=C_{\phi(v)}.
\]
This defines a permutation $\phi$ of $V(\Gamma)$.
The edges of $\Gamma$ can also be recovered from $\Gamma'$. Indeed, if $e=\{u,v\}\in E(\Gamma)$, then the truncation contains a unique edge of $\Gamma'$ joining a vertex of $C_u$ to a vertex of $C_v$. These are precisely the edges of $\Gamma'$ that do not lie in any cycle $C_w$. Since $\psi$ preserves adjacency and maps vertex-cycles to vertex-cycles, it maps such an edge to the unique edge joining $C_{\phi(u)}$ and $C_{\phi(v)}$. Therefore, $\phi$ preserves adjacency and induces an automorphism of $\Gamma$.
Finally, the cyclic order of the edges around a vertex $v$ is encoded by the cycle $C_v$. Since $\psi$ maps $C_v$ onto $C_{\phi(v)}$ as a cycle, it preserves the cyclic arrangement of the incident edges. Thus, the induced automorphism $\phi$ preserves the rotation system determined by the embedding $\beta$, and hence preserves the facial structure of
$\beta(\Gamma)$.
Therefore, $\phi\in\Aut(\beta(\Gamma))$, and by construction $\Phi(\phi)=\psi$. Thus, $\Phi$ is surjective and so the statement follows.
\end{proof}

For cubic graphs, the above result can be strengthened by an analogous argument.
\begin{corollary}
    \label{prop:autiso2}
If $\Gamma$ is a cubic graph, then $\Aut(\Gamma)\cong \Aut(\mathcal{T}(\Gamma)).$
\end{corollary}

We conclude this section by showing that an $\Aut(\Gamma)$-invariant 3-edge-colouring of a cubic graph $\Gamma$ can be translated into an automorphism group-invariant 3-edge-colouring of its complete truncation.
\begin{proposition}\label{prop:truncproper}
    Let $\Gamma$ be a cubic graph. If $\Gamma$ has an $\Aut(\Gamma)$-invariant $3$-edge-colouring, then $\mathcal{T}(\Gamma)$ has also a $3$-edge-colouring which is $\Aut(\mathcal{T}(\Gamma))$-invariant.
\end{proposition}
\begin{proof}
First, let $\kappa:E(\Gamma)\to \{1,2,3\}$ be the $\Aut(\Gamma)$-invariant 3-edge-colouring of $\Gamma$. 
Recall that the complete truncation $\Gamma':=\mathcal{T}(\Gamma)$
has the following structure: The vertices of $\Gamma'$ are given by
\[
V(\Gamma')=\{w_{(v,e)}\mid v\in V(\Gamma),\ e\in E(\Gamma)\ \text{with}\ v\in e\}.
\]

Furthermore, there are two different types of edges in $E(\Gamma')$, namely edges of the form
\begin{itemize}
    \item[1.] $\{w_{(v_1,e)},w_{(v_2,e)}\}$, where $e=\{v_1,v_2\}\in E(\Gamma)$, and 
    \item[2.] $\{w_{(v,e_1)},w_{(v,e_2)}\}$, where $e_1,e_2\in E(\Gamma)$ are edges with $v\in e_1\cap e_2$.
\end{itemize}

For $e_1,e_2\in E(\Gamma)$ with $e_1\cap e_2\neq \emptyset$ we define the number $\kappa[e_1,e_2]=\kappa[e_2,e_1]$ to be the unique integer $i\in \{1,2,3\}\setminus \{\kappa(e_1),\kappa(e_2)\}$.
Using this, we define a 3-edge-colouring $\kappa'$ of $\Gamma'$ by
\[
\kappa'(e')=
\begin{cases}
\kappa(e) & \text{if } e'=\{w_{(v_1,e)},w_{(v_2,e)}\} \text{ for an edge $e=\{v_1,v_2\}\in E(\Gamma)$} \\
\kappa[e_1,e_2]& \text{if } e'=\{w_{(v,e_1)},w_{(v,e_2)}\} \text{ else}
\end{cases}
\]
where $e'\in E(\Gamma')$ is an edge. This edge-colouring of $\Gamma'$ is shown in Figure~\ref{fig:inducedthreecolouring}, where we illustrate the colours $1,2,3$ of the different edge-colourings with the colours red, green and blue, respectively.
\begin{figure}[H]
    \centering
    \begin{tikzpicture}[vertexBall, faceStyle=nolabels, scale=1.5]

\coordinate (V1) at (0+9 , 0);
\coordinate (V2) at (1+9 , 0);
\coordinate (V3) at (0.5+9 , 0.866025);
\coordinate (V4) at (1/2+9 , 1.866025);
\coordinate (V5) at (1.86603 +9, -0.5);
\coordinate (V6) at (-0.86603+9 , -0.5);

\draw[very thick,blue] (V1) -- (V3);
\draw[very thick,green] (V2) -- (V3);
\draw[very thick,red] (V1) -- (V2);
\draw[very thick,green] (V1) -- (V6);
\draw[very thick,red] (V3) -- (V4);
\draw[very thick,blue] (V2) -- (V5);

\vertexLabelR[]{V1}{left}{$ $}
\vertexLabelR[]{V2}{left}{$ $}
\vertexLabelR[]{V3}{left}{$ $}

\draw[-{To[scale=2]}] (7,0.433013) to (8,0.433013);

\coordinate (V7) at (0.5+5 , 0.433013);
\coordinate (V8) at (1/2+5 , 1.433013);
\coordinate (V9) at (1.36603+5 , -0.0669873);
\coordinate (V10) at (-0.36603+5 , -0.0669873);

\draw[very thick,red] (V7) -- (V8);
\draw[very thick,blue] (V7) -- (V9);
\draw[very thick,green] (V7) -- (V10);

\vertexLabelR[]{V7}{left}{$ $}

\node at (5.5, 0.2) {$v$};

\node at (5.7, 0.9) {$e_1$};

\node at (4.9, 0.3) {$e_2$};
\node at (6.1, 0.3) {$e_3$};

\node at (10., 0.8) {$w_{(v,e_1)}$};
\node at (10.5, 0.05) {$w_{(v,e_3)}$};
\node at (8.5, 0.05) {$w_{(v,e_2)}$};

\end{tikzpicture}
    \caption{A 3-edge-colouring of a cubic graph and the induced 3-edge-colouring of its complete truncation}
    \label{fig:inducedthreecolouring}
\end{figure}
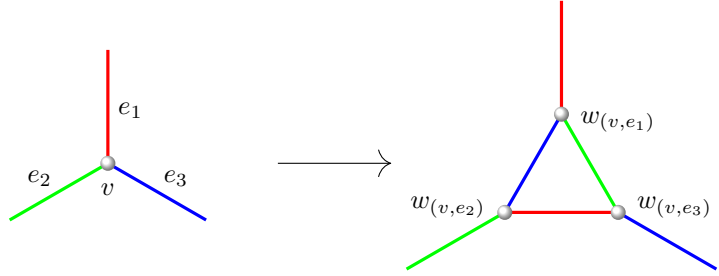
It is straightforward to verify that $\kappa'$ is a valid 3-edge-colouring of $\Gamma'$.
It remains to prove that this colouring is $\Aut(\Gamma')$-invariant. For this, let $\psi\in \Aut(\Gamma')$ be an automorphism of $\Gamma'$.
Since $\Aut(\Gamma)\cong \Aut(\Gamma')$ by Corollary~\ref{prop:autiso2}, there exists $\phi\in \Aut(\Gamma)$ such that
\[
\psi(w_{(v,e)})=w_{(\phi(v),\phi(e))}
\]
for all $v\in V(\Gamma)$ and $e\in E(\Gamma)$.
First, we assume that $e'$ is an edge of the form $e'=\{w_{(v_1,e)},w_{(v_2,e)}\}$ with $e=\{v_1,v_2\}\in E(\Gamma)$. Then we conclude that
\[
\psi(e')=\{w_{(\phi(v_1),\phi(e))},w_{(\phi(v_2),\phi(e))}\}.
\]
Since the colouring $\kappa$ is $\Aut(\Gamma)$-invariant, we have $\kappa(e)=\kappa(\phi(e))$, and therefore $\kappa'(e')=\kappa'(\psi(e'))$.
Now, we assume that $e'=\{w_{(v,e_1)},w_{(v,e_2)}\}$, where $e_1,e_2$ are edges in $E(\Gamma)$ with $v\in e_1\cap e_2$. Then we know that
\[
\psi(e')=\{w_{(\phi(v),\phi(e_1))},w_{(\phi(v),\phi(e_2))}\}.
\]
Since $\kappa$ is $\Aut(\Gamma)$-invariant, we have $\kappa(e_1)=\kappa(\phi(e_1))$ and $\kappa(e_2)=\kappa(\phi(e_2))$.
Hence, the third colour different from $\kappa(e_1)$ and $\kappa(e_2)$ coincides with the third colour different from 
$\kappa(\phi(e_1))$ and $\kappa(\phi(e_2))$, and therefore
\[
\kappa[e_1,e_2]=\kappa[\phi(e_1),\phi(e_2)].
\]
It follows that $\kappa'(e')=\kappa'(\psi(e'))$ for all $\psi \in \Aut(\Gamma')$. Thus, $\kappa'$ is $\Aut(\Gamma')$-invariant.
\end{proof}

\section{Regular graphs with prescribed automorphism groups via Cartesian products}\label{sec:regular}

In this section, we show that for any finite group $G$ and every $d\geq 3$, there exists a $d$-regular graph $\Gamma$ satisfying the following properties:
\begin{enumerate}
    \item $\Aut(\Gamma)\cong G$ and
    \item $\Gamma$ has an $\Aut(\Gamma)$-invariant $d$-edge-colouring.
\end{enumerate}
By establishing this result, we extend Sabidussi's classical result. Before proceeding further, we revisit Sabidussi's original construction.
\begin{remark}
\label{sabidussi:construction}
Let $ G $ be an arbitrary finite group, $ S $ a generating set of $ G $ with $n:=\vert S\vert $, and $ d \geq 3 $ a natural number. In \cite[Theorem 3.7]{Sabidussi}, Sabidussi describes a construction of a $ d $-regular graph with prescribed automorphism group $ G $ as follows:

For the cases $ d\in\{3,4,5\} $, Sabidussi constructs $ d $-regular prime graphs $ \Gamma^d $ such that $ \Aut(\Gamma^d) \cong G $. In the cubic case ($ d=3 $), the graph $\Gamma^3$ is the graph arising from Frucht's cubic graph construction in \cite[Theorem 4.1]{Frucht}. 
In particular, Sabidussi chooses $\Gamma^3$ to be given with incidences as described in \cite[p.~374]{Frucht}.
In this construction, Frucht defines a cubic graph associated with a group generated by $n\geq 2$ elements, with vertex set $\{x_{(i,g)} \mid i = 1,\dots,2n+4\}$ and adjacencies as illustrated in Figure~\ref{frucht1}.
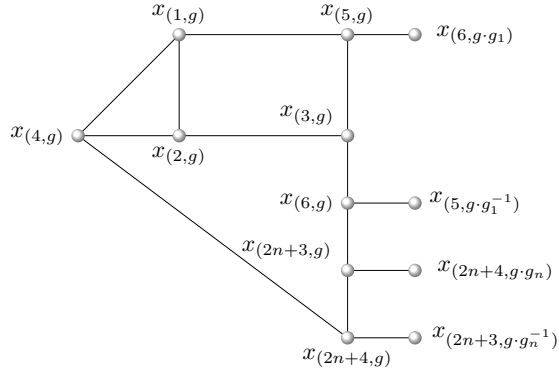
\begin{figure}[H]
    \centering
    \resizebox{!}{5cm}{\begin{tikzpicture}[vertexBall, faceStyle=nolabels, scale=0.5]

	\coordinate (V1) at (3 , 3);
	\coordinate (V2) at (3 , 0);
	\coordinate (V3) at (8 , 0);
	\coordinate (V4) at (0 , 0);
	\coordinate (V5) at (8 , 3);
	\coordinate (V6) at (8 , -2);
	\coordinate (V7) at (8 , -4);
	\coordinate (V8) at (8 , -6);
	\coordinate (V9) at (10 , 3);
	\coordinate (V10) at (10 , -2);
	\coordinate (V11) at (10, -4);
	\coordinate (V12) at (10 , -6);

	\draw (V1) -- (V4);
	\draw (V2) -- (V1);
	\draw (V2) -- (V4);

	\draw (V4) -- (V8);
	\draw (V8) -- (V7);
	\draw (V2) -- (V3);

	\draw (V5) -- (V3);
	\draw (V1) -- (V5);
	\draw (V5) -- (V9);

	\draw (V6) -- (V7);
	\draw (V3) -- (V6);
	\draw (V7) -- (V11);
	\draw (V8) -- (V12);
	\draw (V6) -- (V10);

	\vertexLabelR[]{V1}{left}{$ $}
	\vertexLabelR[]{V2}{left}{$ $}
	\vertexLabelR[]{V3}{left}{$ $}
	\vertexLabelR[]{V4}{left}{$ $}
	\vertexLabelR[]{V5}{left}{$ $}
	\vertexLabelR[]{V6}{left}{$ $}
	\vertexLabelR[]{V7}{left}{$ $}
	\vertexLabelR[]{V8}{left}{$ $}
	\vertexLabelR[]{V9}{left}{$ $}
	\vertexLabelR[]{V10}{left}{$ $}
	\vertexLabelR[]{V11}{left}{$ $}
	\vertexLabelR[]{V12}{left}{$ $}

	\def\x{0.6}
	\node at (3 , 3+\x) {$x_{(1,g)}$};
	\node at (3 , 0-\x) {$x_{(2,g)}$};
	\node at (8-2*\x , 0+\x) {$x_{(3,g)}$};
	\node at (0-2*\x , 0) {$x_{(4,g)}$};
	\node at (8 , 3+\x) {$x_{(5,g)}$};
	\node at (8-2*\x , -2) {$x_{(6,g)}$};
	\node at (8-3*\x , -4+\x) {$x_{(2n+3,g)}$};
	\node at (8 , -6-\x) {$x_{(2n+4,g)}$};
	\node at (10+3*\x , 3) {$x_{(6,g\cdot g_1)}$};
	\node at (10+3*\x , -2) {$x_{(5,g\cdot g_1^{-1})}$};
	\node at (10+4*\x, -4) {$x_{(2n+4,g\cdot g_n)}$};
	\node at (10+4*\x , -6) {$x_{(2n+3,g\cdot g_n^{-1})}$};

	\end{tikzpicture}}

\caption{Frucht's original cubic graph construction for a group $G$ generated by $n\geq 2$ elements}
\label{frucht1}
\end{figure}

For $ d=4 $ and $ d=5 $, the construction builds upon the cubic graph $ \Gamma^3 $ together with auxiliary graphs illustrated in \Cref{fig:auxSabidussi1,fig:auxSabidussi2}. 
Loosely speaking, for $d=4$, Sabidussi chooses the graph $Y$ illustrated in Figure~\ref{fig:auxSabidussi1} to consist of $2n+4$ vertices. Moreover, the vertices of $Y$ are labelled with $y_{(i,g)}$, where $g\in G$ and $i =1,\ldots, 2n+4$. Thus, for each $g\in G$ we obtain a graph that is illustrated in \Cref{fig:auxSabidussi1}. Sabidussi then constructs a 4-regular graph $\Gamma^4$ with vertex set $\{x_{(i,g)}, y_{(i,g)}\mid g\in G , i=1,\ldots,2n+4\}$. The incidences of $\Gamma^4$ are then given by the incidences of the graphs $\Gamma^3$ and $Y$ together with the edges $$\{\{x_{(i,g)},y_{(i,g)}\}\mid g\in G, i=1,\ldots,2n+4 \}.$$
For further details and the construction of the 5-regular graph with the desired properties, we refer the reader to Sabidussi’s original paper.

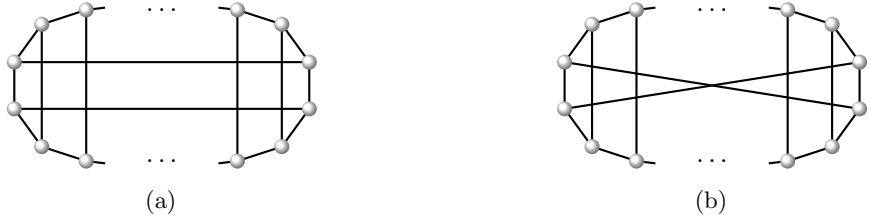
\begin{figure}[H]
    \centering
    \begin{subfigure}{.48\textwidth}
        \centering
        \begin{tikzpicture}[vertexBall, edgeDouble=nolabels, faceStyle=nolabels, scale=1.]
    
\coordinate (V11)  at ( 0.0000-2,  1.0000);
\coordinate (V11c)  at ( 0.0000-1.75,  1.0000+0.03);
\coordinate (V2)  at (-0.5878-2,  0.8090);
\coordinate (V3)  at (-0.9511-2,  0.3090);
\coordinate (V4)  at (-0.9511-2, -0.3090);
\coordinate (V5)  at (-0.5878-2, -0.8090);
\coordinate (V61)  at ( 0.0000-2, -1.0000);
\coordinate (V61c)  at ( 0.0000-1.75, -1.0000-0.03);

\coordinate (V62)  at ( 0.0000, -1.0000);
\coordinate (V62c)  at ( 0.0000-0.25, -1.0000-0.03);
\coordinate (V7)  at ( 0.5878, -0.8090);
\coordinate (V8)  at ( 0.9511, -0.3090);
\coordinate (V9)  at ( 0.9511,  0.3090);
\coordinate (V10) at ( 0.5878,  0.8090);
\coordinate (V12)  at ( 0.0000,  1.0000);
\coordinate (V12c)  at ( 0.0000-0.25,  1.0000+0.03);

\draw[thick] (V11) -- (V61);
\draw[thick] (V12) -- (V62);

\draw[thick] (V2) -- (V5);
\draw[thick] (V10) -- (V7);
\draw[thick] (V9) -- (V3);
\draw[thick] (V8) -- (V4);

\draw[thick] (V11) -- (V2);
\draw[thick] (V3) -- (V2);
\draw[thick] (V3) -- (V4);
\draw[thick] (V5) -- (V4);
\draw[thick] (V61) -- (V5);

\draw[thick] (V62) -- (V7);
\draw[thick] (V8) -- (V7);
\draw[thick] (V8) -- (V9);
\draw[thick] (V10) -- (V9);
\draw[thick] (V10) -- (V12);

\draw[thick] (V11) -- (V11c);
\draw[thick] (V12) -- (V12c);
\draw[thick] (V61) -- (V61c);
\draw[thick] (V62) -- (V62c);

\node at (-1, 1.) {$\ldots$};
\node at (-1, -1.) {$\ldots$};

\vertexLabelR[]{V11}{left}{$ $}
\vertexLabelR[]{V12}{left}{$ $}
\vertexLabelR[]{V2}{left}{$ $}
\vertexLabelR[]{V3}{left}{$ $}
\vertexLabelR[]{V4}{left}{$ $}
\vertexLabelR[]{V5}{left}{$ $}
\vertexLabelR[]{V61}{left}{$ $}
\vertexLabelR[]{V62}{left}{$ $}
\vertexLabelR[]{V7}{left}{$ $}
\vertexLabelR[]{V8}{left}{$ $}
\vertexLabelR[]{V9}{left}{$ $}
\vertexLabelR[]{V10}{left}{$ $}
\end{tikzpicture}
        \subcaption{}
        \label{fig:auxSabidussi1}
    \end{subfigure}
    \begin{subfigure}{.48\textwidth}
        \centering
        \begin{tikzpicture}[vertexBall, edgeDouble=nolabels, faceStyle=nolabels, scale=1]

\coordinate (V11)  at ( 0.0000-2,  1.0000);
\coordinate (V11c)  at ( 0.0000-1.75,  1.0000+0.03);
\coordinate (V2)  at (-0.5878-2,  0.8090);
\coordinate (V3)  at (-0.9511-2,  0.3090);
\coordinate (V4)  at (-0.9511-2, -0.3090);
\coordinate (V5)  at (-0.5878-2, -0.8090);
\coordinate (V61)  at ( 0.0000-2, -1.0000);
\coordinate (V61c)  at ( 0.0000-1.75, -1.0000-0.03);

\coordinate (V62)  at ( 0.0000, -1.0000);
\coordinate (V62c)  at ( 0.0000-0.25, -1.0000-0.03);
\coordinate (V7)  at ( 0.5878, -0.8090);
\coordinate (V8)  at ( 0.9511, -0.3090);
\coordinate (V9)  at ( 0.9511,  0.3090);
\coordinate (V10) at ( 0.5878,  0.8090);
\coordinate (V12)  at ( 0.0000,  1.0000);
\coordinate (V12c)  at ( 0.0000-0.25,  1.0000+0.03);

\draw[thick ] (V11) -- (V61);
\draw[thick ] (V12) -- (V62);

\draw[thick ] (V2) -- (V5);
\draw[thick ] (V10) -- (V7);
\draw[thick ] (V9) -- (V4);
\draw[thick ] (V8) -- (V3);

\draw[thick ] (V11) -- (V2);
\draw[thick ] (V3) -- (V2);
\draw[thick ] (V3) -- (V4);
\draw[thick ] (V5) -- (V4);
\draw[thick ] (V61) -- (V5);

\draw[thick ] (V62) -- (V7);
\draw[thick ] (V8) -- (V7);
\draw[thick ] (V8) -- (V9);
\draw[thick ] (V10) -- (V9);
\draw[thick ] (V10) -- (V12);

\draw[thick ] (V11) -- (V11c);
\draw[thick ] (V12) -- (V12c);
\draw[thick ] (V61) -- (V61c);
\draw[thick] (V62) -- (V62c);

\node at (-1, 1.) {$\ldots$};
\node at (-1, -1.) {$\ldots$};

\vertexLabelR[]{V11}{left}{$ $}
\vertexLabelR[]{V12}{left}{$ $}
\vertexLabelR[]{V2}{left}{$ $}
\vertexLabelR[]{V3}{left}{$ $}
\vertexLabelR[]{V4}{left}{$ $}
\vertexLabelR[]{V5}{left}{$ $}
\vertexLabelR[]{V61}{left}{$ $}
\vertexLabelR[]{V62}{left}{$ $}
\vertexLabelR[]{V7}{left}{$ $}
\vertexLabelR[]{V8}{left}{$ $}
\vertexLabelR[]{V9}{left}{$ $}
\vertexLabelR[]{V10}{left}{$ $}
\end{tikzpicture}
        \subcaption{}
        \label{fig:auxSabidussi2}
    \end{subfigure}
    \caption{Sabidussi’s auxiliary graphs for constructing 4-regular (a) and 5-regular (b) graphs with prescribed automorphism group.}
    \label{fig:auxSabidussi}
\end{figure}

In the general case $ d > 5 $, Sabidussi employs graph products. Writing $ d = d' + 3\ell $, where $d  \equiv d' \pmod{3}$ and $ \ell \in \mathbb{N}_0 $, the construction begins with a $ d' $-regular graph $ \Gamma^{d'} $ satisfying $ \Aut(\Gamma^{d'}) \cong G $. In addition, pairwise non-isomorphic cubic prime graphs $ \Gamma_1, \dots, \Gamma_\ell $ with trivial automorphism group are used. These are chosen such that $\Gamma^{d'},\Gamma_1,\ldots,\Gamma_{\ell}$ are relatively prime. The desired graph is then defined as
\[
\Gamma' := \Gamma^{d'} \times \prod_{i=1}^\ell \Gamma_i.
\]
Sabidussi asserts that this construction yields a $ d$-regular graph satisfying $ \Aut(\Gamma') \cong G $.
While the overall strategy is elegant and highly influential, a closer inspection suggests that certain aspects of the argument merit additional clarification.

First, the cubic construction of Frucht, as used by Sabidussi, does not explicitly address the case where $(n=1)$ and misses the case $\vert G\vert \leq 2$. Consequently, the construction is most naturally interpreted in the case $n\geq 2.$

Secondly, there is a subtle point in Frucht’s construction concerning the choice of generating set. For groups generated by two elements, the construction behaves as intended. However, for generating sets of size at least three, the resulting graph need not have automorphism group isomorphic to $ G $ (see \cite{FruchtHow}). For example, for $ G = A_5 $ with generating set
\[
S = \{ (1,5)(2,4),\ (1,2,4,3,5),\ (2,5,3) \},
\]
the resulting cubic graph admits an automorphism group isomorphic to $ C_2 \times A_5 $. This phenomenon and more precisely this example has been addressed in \cite{SurfacesWithAuto}, where suitable modifications are proposed. Hence, for $d\equiv 0 \pmod{3}$, there exist finite groups $G$ and $d$-regular graphs $\Gamma'$ arising from Sabidussi's construction satisfying $\Aut(\Gamma')\ncong G$. This indicates that additional care is required when applying the construction in full generality.

Finally, Sabidussi indicates that the identification $ \Aut(\Gamma') \cong G $ follows by an argument analogous to Frucht’s cycle-counting method. In light of the considerations above, this step may benefit from a more detailed justification. 
\end{remark}

To prove Theorem~\ref{theorem:dregular}, namely that for every finite group $G$ and every degree $d\geq 3$, there exists a $d$-regular graph that has an automorphism group-invariant $d$-edge-colouring and whose automorphism group realises $G$, we follow the general strategy of Sabidussi with suitable modifications to ensure that the argument applies in full generality. In particular, inspired by Frucht’s approach in \cite{Frucht}, we distinguish three cases: (1) $\lvert G \rvert \leq 2$, (2) $G \cong C_\ell$ for $\ell \geq 3$, and (3) $G$ is $n$-generated with $n\geq 2$.

In Section~\ref{section:atmost2}, we treat the first case. We then address the second case in Section~\ref{subsection:cyclic} for $d=3$ and hence construct cubic graphs with the desired properties. The result for $d=3$ in the case (3) is covered by the modification presented in \cite{SurfacesWithAuto}. This approach is justified by the fact that, once cases (2) and (3) are established in the cubic setting, the general cases can be handled uniformly. This is described in more detail in \Cref{section:regular}.

Before proceeding further, we establish a result that allows us to verify that all cubic graphs constructed in this section are prime. In particular, we show that the only cubic graphs that are not prime are exactly the circular ladder graphs which are defined as follows: For an $
\ell\geq 3$ the circular ladder graph on $k:=2\cdot \ell$ vertices is the graph with vertex-set $\{1,\ldots,k\}$ and edges 
$$\{\{i,i+1\},\{\ell+i,\ell+i+1\}\mid i=1,\ldots,\ell-1\}\cup \{\{i,\ell+i\}\mid i=1,\ldots,\ell\}\cup \{\{1,\ell\},\{\ell+1,k\}\} .$$
\begin{proposition}
\label{prop:cubicprime}
Let $\Gamma$ be a cubic graph. Then $\Gamma$ is prime if and only if $\Gamma$ is not isomorphic to a circular ladder graph.
\end{proposition}
\begin{proof}
It is easy to see that a circular ladder graph on $2\cdot \ell$ vertices can be written as the product of the graph $\Gamma_1$ consisting of a single edge and a graph $\Gamma_2$ forming a cycle graph on $\tfrac{\vert V(\Gamma)\vert}{2}$ vertices.  Hence, if $\Gamma$ is isomorphic to a circular ladder graph, then it is not prime.

Now, let us assume that $\Gamma$ is not prime. Thus, there exist finitely many non-trivial $\Gamma_1,\ldots,\Gamma_\ell$ such that $\Gamma=\prod_{i=1}^\ell \Gamma_i.$ Since $\Gamma$ is connected, the graphs $\Gamma_1,\ldots,\Gamma_\ell$ have to be connected. Furthermore, if $v:=(v_1,\ldots,v_\ell)\in V(\Gamma)$ is a vertex, then 
    \[
    \sum_{i=1}^\ell\deg(v_i)=\deg(v)=3
    \]
follows. Hence, we obtain exactly two cases, namely:
\begin{itemize}
    \item[(1)] $\ell=2$ and $\Gamma_i$ is $i$-regular for $i=1,2$ or
    \item[(2)] $\ell=3$ and $\Gamma_1,\Gamma_2,\Gamma_3$ are all $1$-regular.
\end{itemize}

If (1) is the case, then $\Gamma$ is isomorphic to the circular ladder graph on $\vert  V(\Gamma_1)\vert\cdot\vert  V(\Gamma_2)\vert=2\cdot \vert  V(\Gamma_2)\vert$ vertices. In the other case, $\Gamma$ is isomorphic to the circular ladder graph on $8$ vertices. Hence, we conclude the proof.
\end{proof}
\subsection{Regular graphs with automorphism groups of order at most 2}\label{section:atmost2}
In the case that $G$ is a group of order at most $2$, we provide examples of regular graphs of vertex degree at most $5$ satisfying the desired properties in the appendix  of this paper. Using these examples, we obtain the following result:

\begin{proposition}\label{prop:atmost2}
        For every finite group $G$ with $\vert G\vert\leq 2$ and every $d=3,4,5$ there exists a $d$-regular graph $\Gamma$ with $\Aut(\Gamma)\cong G$ that has an $\Aut(\Gamma)$-invariant $d$-edge-colouring. Note that in the case $d=3$, $\Gamma$ can be chosen to be prime.
\end{proposition}

We now aim to prove the general case ($d \geq 3$) by making use of graph products of regular graphs. For this, we first show that an edge-colouring of a graph obtained via a graph product can be constructed from edge-colourings of the factor graphs.

\begin{lemma}
\label{lemma:edgecolouring}
Let $\Gamma_1$ and $\Gamma_2$ be regular graphs with vertices of degree $d_1$ and $d_2$, respectively. If $\Gamma_1$ admits a $d_1$-edge-colouring and $\Gamma_2$ admits a $d_2$-edge-colouring, then $\Gamma_1 \times \Gamma_2$ admits a $d$-edge-colouring with $d := d_1 + d_2$. 
Furthermore, if $\Gamma_1$ and $\Gamma_2$ are relatively prime and both $d_i$-edge-colourings are automorphism group-invariant, then $\Gamma_1 \times \Gamma_2$ admits an $\mathrm{Aut}(\Gamma_1 \times \Gamma_2)$-invariant $d$-edge-colouring.
\end{lemma}
\begin{proof}
Let the edge-colourings of $\Gamma_1$ and $\Gamma_2$ be given by $\kappa_{1}\colon E(\Gamma_1)\to \{1,\ldots,d_1\}$ and $\kappa_{2}\colon E(\Gamma_2)\to \{1,\ldots,d_2\}$. Here, we use these colourings to construct a valid $d$-edge-colouring $\kappa\colon E(\Gamma_1\times\Gamma_2)\to\{1,\ldots,d\}$.
In particular, for an edge $e = \{(x_1,y_1),(x_2,y_2)\} \in E(\Gamma_1 \times \Gamma_2)$, we define $\kappa $ by
\[
\kappa(e)=
\begin{cases}
\kappa_{1}(\{x_1,x_2\}), & \text{if } y_1 = y_2, \\
d_1 + \kappa_{2}(\{y_1,y_2\}), & \text{if } x_1 = x_2.
\end{cases}
\]
It is straightforward to verify that this defines a valid edge-colouring of $\Gamma_1 \times \Gamma_2$.
Now suppose that $\Gamma_1$ and $\Gamma_2$ are relatively prime and that the colourings $\kappa_1$ and $\kappa_2$ are both automorphism-group-invariant. By \cite[Theorem 3.1]{graphmult} we know that 
\[
\mathrm{Aut}(\Gamma_1 \times \Gamma_2) \cong \mathrm{Aut}(\Gamma_1) \times \mathrm{Aut}(\Gamma_2).
\]
Hence, for all $e \in E(\Gamma_1 \times \Gamma_2)$ and all $\phi \in \mathrm{Aut}(\Gamma_1 \times \Gamma_2)$, we have $\kappa(e) = \kappa(\phi(e))$ which concludes the proof.
\end{proof}

 Before concluding the section, we introduce the \emph{$\ell$-fold truncation operator} of a regular graph.
\begin{definition}
Let $\Gamma$ be a $d$-regular graph, where $d\geq 3$. For $\ell\geq 1$ we define the $\ell$-fold truncation operator recursively by $\mathcal{T}_\ell(\Gamma,\beta):=\mathcal{T}(\Gamma,\beta)$ if $\ell=1$ and 
$\mathcal{T}_\ell(\Gamma,\beta):=\mathcal{T}(\mathcal{T}_{\ell-1}(\Gamma,\beta))$ if $\ell >1.$
\end{definition}

This leads directly to the following proposition.

\begin{proposition}
\label{prop:atmost22}
        For a group $G$ with $\vert G\vert \leq 2$ and every $d\geq 3$ there exists a $d$-regular graph $\Gamma$ with $\Aut(\Gamma)\cong G$ and an $\Aut(\Gamma)$-invariant $d$-edge-colouring of $\Gamma$.
\end{proposition}
\begin{proof}
Let $d' \in \{3,4,5\}$ be a natural number such that $d' \equiv d \pmod{3}$. Hence, there exists a natural number $\ell \in \mathbb{N}_0$ such that $d = d' + 3\ell$. By Proposition~\ref{prop:atmost2}, there exists a $d'$-regular graph $\Gamma'$ such that $\mathrm{Aut}(\Gamma') \cong G$ and $\Gamma'$ admits an $\mathrm{Aut}(\Gamma')$-invariant $d'$-edge-colouring. 
Furthermore, by the same proposition there exists a cubic prime graph $\hat{\Gamma}$ with $|\mathrm{Aut}(\hat{\Gamma})| = 1$ such that $\hat{\Gamma}$ admits an $\mathrm{Aut}(\hat{\Gamma})$-invariant 3-edge-colouring.
Thus, we define the graphs $\Gamma_i := \mathcal{T}_i(\hat{\Gamma})$ for $1 \leq i \leq \ell $. 
We know that for all $1 \leq i \leq \ell$ the graph $\Gamma_i$ is prime (see Proposition~\ref{prop:cubicprime}), admits an  $\mathrm{Aut}(\Gamma_i)$-invariant 3-edge-colouring (see Proposition~\ref{prop:truncproper}), and satisfies $|\mathrm{Aut}(\Gamma_i)| = 1$ (see Corollary~\ref{prop:autiso2}). Thus, the graph
\[
\Gamma = \Gamma' \times \prod_{i=1}^\ell \Gamma_i
\]
is a $d$-regular graph. Since the graphs $\Gamma_1,\ldots,\Gamma_\ell$ are prime and $\Gamma'$ contains vertices not contained in cycles of length $3$, the graph $\Gamma'$ cannot be written as the product $\Gamma'=Z\times \Gamma_i$ for a graph $Z$ and a $1\leq i\leq \ell$. Consequently, $\Gamma',\Gamma_1,\ldots,\Gamma_\ell$ are relatively prime. Hence, $\Gamma$ satisfies $\mathrm{Aut}(\Gamma) \cong G$ by \cite[Theorem 3.1]{graphmult}. Furthermore, by Lemma~\ref{lemma:edgecolouring}, this graph admits an $\mathrm{Aut}(\Gamma)$-invariant $d$-edge-colouring.
\end{proof}

\subsection{Cubic graphs with cyclic automorphism groups}
\label{subsection:cyclic}
In this section, we construct cubic prime graphs whose automorphism groups are cyclic and that additionally have automorphism group-invariant 3-edge-colourings.
Note that in \cite{Frucht}, Frucht presents a construction of cubic graphs whose automorphism groups are cyclic. Here, we further modify this construction to obtain cubic prime graphs $\Gamma_{C_\ell}$ with the desired properties for $\ell\geq 3$. Before proceeding further, we introduce the notion of a quadratic form. In \cite{Frucht}, these  polynomials or more precisely the corresponding monomials have been used to define the incidence structure of various graphs. 
\begin{definition}
    Let $\Gamma$ be a graph with vertices $V(\Gamma)=\{v_1,\ldots,v_k\}$ and $x_1,\ldots,x_k$ formal indeterminates such that $x_i$ corresponds to $v_i$ for all $1\leq i\leq k.$ Furthermore, for $1\leq i,j \leq k$ we define $a_{i,j}$ as $\frac{1}{2}$ if $\{v_i,v_j\}\in E$ and $0$ otherwise. Thus, we define the \emph{quadratic form} $Q_{\Gamma}$ of $\Gamma$ as
    \[
    Q_{\Gamma}:=\sum_{i,j=1}^k a_{i,j}x_ix_j.
    \]
\end{definition}
It can be observed that the quadratic form of a given graph can be related to the adjacency matrix of the graph as described below.
\begin{remark}\label{rem:indeterverts}
 Let $\Gamma$ be a graph with vertices $V(\Gamma)=\{v_1,\ldots,v_k\}.$ Furthermore, let $x=(x_1,\dots,x_k)$ be formal indeterminates such that $x_i$ corresponds to the vertex $v_i$ for all $1\leq i\leq k$, and $Q:=Q_{\Gamma}$ the quadratic form of $\Gamma$. If $A\in \{0,1\}^{k\times k}$ is the adjacency matrix of $\Gamma$, then 
 $$ Q=\frac{1}{2}xAx^t .$$
 Each monomial $x_ix_j=x_jx_i$ that is a summand of $Q$ corresponds to an edge $\{v_i,v_j\}\in E$ and thus the adjacency matrix $A$ can be recovered from $Q$. 
Additionally, we know that $\Aut(\Gamma)$ is isomorphic to the group consisting of the permutations of the columns and rows of $A$ that leave the adjacency matrix $A$ invariant. Thus, the automorphism group of $\Gamma$ is isomorphic to the group of permutations of $x=(x_1,\ldots,x_k)$
that preserve the quadratic form $Q$.
\end{remark}
With Remark~\ref{rem:indeterverts} it suffices to determine a quadratic form in order to define a graph.
\begin{remark}
    Let $\Gamma$ be a graph and $Q$ its quadratic form with corresponding indeterminates $x_1,\ldots,x_k$. Since $Q$ determines $\Gamma$ and also $\Aut(\Gamma)$, up to isomorphism, we identify $V(\Gamma)$ with the set of formal indeterminates. Hence, we write $V(\Gamma)=\{x_1,\ldots,x_k\}.$
\end{remark}

With the notion of a quadratic form in place, we are able to present our proposed cubic graph $\Gamma_{C_\ell}.$
Let therefore $\ell\geq3$ be a natural number and $\sigma\in S_\ell$ be an element of order $\ell.$ Thus, $G:=\langle \sigma\rangle\cong C_\ell$. We define the cubic graph $\Gamma_{C_\ell}$ via
$$V(\Gamma_{C_\ell})=\{x_{(g,i)}\mid g\in C_\ell,\; i=1,\ldots,8\}$$ and the quadratic form $Q_{C_\ell}$ given by
\begin{align*}
    & \sum_{g\in G} (x_{(g,1)}x_{(g,2)}+ x_{(g,1)}x_{(g,3)}+ x_{(g,2)}x_{(g,4)} +x_{(g,3)}x_{(g,4)}+ x_{(g,3)}x_{(g,6)} ) \\
    +&\sum_{g\in G} (x_{(g,4)}x_{(g,5)}+ x_{(g,5)}x_{(g,7)}+ x_{(g,5)}x_{(g,8)} +x_{(g,6)}x_{(g,7)}+ x_{(g,7)}x_{(g,8)} )\\
    +& \sum_{g\in G} (x_{(g,2)}x_{(g\cdot\sigma,1)}+ x_{(g,8)}x_{(g\cdot\sigma,6)}).
\end{align*}
It is easy to see that the graph $\Gamma_{C_\ell}$ is indeed cubic.
A component of this graph is illustrated in Figure~\ref{fig:cn}. 
Now, let us examine the automorphism group of $\Gamma_{C_\ell}$.

\begin{proposition}
\label{prop:cyclicauto}
For $\ell\geq 3$ the cubic graph $\Gamma_{C_\ell}$ is a prime graph that satisfies $\Aut(\Gamma_{C_\ell})\cong C_\ell$ and that has an $\Aut(\Gamma_{C_\ell})$-invariant $3$-edge-colouring.
\end{proposition}

\begin{proof}
Since the graph $\Gamma_{C_\ell}$ is not isomorphic to a circular ladder graph, it is prime by Proposition~\ref{prop:cubicprime}.
Now, let $\phi\in \Aut(\Gamma_{C_\ell})$ be an automorphism.  In order to prove the above statement, we have to show that there exists an $h\in G$ such that $\phi(x_{(g,i)})=x_{(h\cdot g,i)}$ for all $g\in G$ and $1\leq i \leq 8.$ Here, this is achieved by examining lengths of cycles in $\Gamma_{C_\ell}$.  First, we see that $\Gamma_{C_\ell}$ has exactly $\ell$ cycles of length 3, 
namely $(x_{(g,5)},x_{(g,7)},x_{(g,8)})$ and exactly $\ell$ cycles of length 4, namely $(x_{(g,1)},x_{(g,2)},x_{(g,4)},x_{(g,3)})$ with $g\in G.$ 
 Now, $\phi$ has to map a cycle onto a cycle of the same length and has to respect the incidences of $\Gamma_{C_\ell}$. Thus, if $g\in G$ is an arbitrary element, then there exists $h_g \in G$ with 
 \begin{align*}
&\phi((x_{(g,5)},x_{(g,7)},x_{(g,8)}))=     (x_{(h_g\cdot g,5)},x_{(h_g\cdot g,7)},x_{(h_g\cdot g,8)}),\\
&\phi((x_{(g,1)},x_{(g,2)},x_{(g,4)},x_{(g,3)}))=(x_{(h_g\cdot g,1)},x_{(h_g\cdot g,2)},x_{(h_g\cdot g,4)},x_{(h_g\cdot g,3)}).
 \end{align*}

 As a first step, we show $\phi(x_{(g,i)})=x_{(h_g\cdot g,i)}$ for all $i=1,\ldots,8$.
 Since the only edges in $E(\Gamma_{C_\ell})$ with one incident vertex being contained in a 3-cycle and the other vertex being contained in a 4-cycle are edges of the form $e=\{x_{(g',4)},x_{(g',5)}\},$ where $g'\in G$, we deduce 
 $\phi(x_{(g,4)})=x_{(h_g \cdot g,4)}$ and $\phi(x_{(g,5)})=x_{(h_g \cdot g,5)}.$ Furthermore, we observe that $x_{(g,7)}$ is contained in a cycle of length $5$, whereas there is no cycle of length 5 that contains $x_{(g,8)}$. Hence, $\phi(x_{(g,7)})=x_{(h_g \cdot g,7)}$ and $\phi(x_{(g,8)})=x_{(h_g \cdot g,8)}.$ 
 This directly implies $\phi(x_{(g,6)})=x_{(h_g \cdot g,6)}$. Finally, $\phi(x_{(g,3)})$ has to be incident to $\phi(x_{(g,4)})=x_{(h_g \cdot g,4)}$ and $\phi(x_{(g,6)})=x_{(h_g \cdot g,6)}$. Thus, $\phi(x_{(g,3)})=x_{(h_g \cdot g,3)}$. This leads to $\phi(x_{(g,1)})=x_{(h_g \cdot g,1)}$ and $\phi(x_{(g,2)})=x_{(h_g \cdot g,2)}$.
 So all together, we obtain $\phi(x_{(g,i)})=x_{(h_g \cdot g,i)}$ for all $i=1,\ldots,8.$

 It remains to show that $h_{g_1}=h_{g_2}$ for all $g_1,g_2\in G.$ We obtain this by observing that the cycle
$(x_{(\sigma,1)}, x_{(\sigma,2)},\ldots,x_{(\sigma^\ell,1)}, x_{(\sigma^\ell,2)})$
is mapped onto the same cycle 
$$(x_{(h_\sigma\cdot \sigma ,1)}, x_{( h_\sigma\cdot \sigma,2)},\ldots,x_{(h_{\sigma^\ell}\cdot \sigma^\ell,1)}, x_{(h_{\sigma^{\ell}}\cdot \sigma^{\ell},2)}).$$
This is only the case if $h_{\sigma}=\cdots= h_{\sigma^{\ell}}.$ Hence,
$$\Aut(\Gamma_{C_\ell})=\{\phi_h : V(\Gamma_{C_\ell})\to  V(\Gamma_{C_\ell}), x_{(g,i)}\mapsto x_{(h\cdot g,i)}\mid  h\in G\}.$$
This allows us to construct a 3-edge-colouring $\kappa$ of $\Gamma_{C_\ell}$ that is $\Aut(\Gamma_{C_\ell})$-invariant via
\[
\kappa(e)=\begin{cases}
1, & \text{if } e=\{x_{(g,2)},x_{(g\cdot\sigma,1)}\},\{x_{(g,8)},x_{(g\cdot \sigma,6)}\},\{x_{(g,5)},x_{(g,7)}\},\{x_{(g,3)},x_{(g,4)}\}, \\
2, & \text{if } e=\{x_{(g,1)},x_{(g,3)}\},\{x_{(g,2)},x_{(g,4)}\},\{x_{(g,5)},x_{(g,8)}\},\{x_{(g,6)},x_{(g,7)}\}, \\
3, & \text{if } e=\{x_{(g,1)},x_{(g,2)}\},\{x_{(g,3)},x_{(g,6)}\},\{x_{(g,7)},x_{(g,8)}\}, \{x_{(g,4)},x_{(g,5)}\}. \\

\end{cases}
\]
The edge-colouring $\kappa$ is shown in Figure~\ref{fig:cn}, where we illustrate the colours $1,2,3$ of $\kappa$ with the colours red, green and blue, respectively.
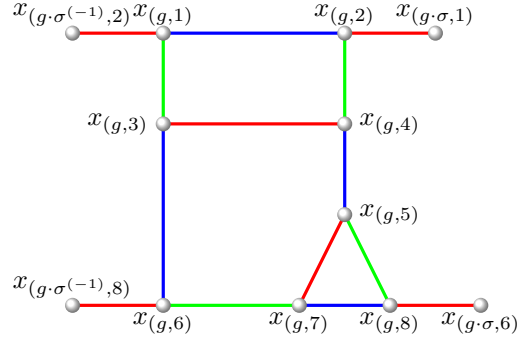
\begin{figure}[H]
    \centering
    \begin{tikzpicture}[vertexBall, faceStyle=nolabels, scale=1.2]

\coordinate (V1) at (0,0);
\coordinate (V21) at (2.5,0);
\coordinate (V22) at (1.5,0);
\coordinate (V3) at (0,2);
\coordinate (V4) at (2,2);
\coordinate (V5) at (0,3);
\coordinate (V6) at (2,3);
\coordinate (V7) at (2,1);

\coordinate (V8) at (3.5,0);
\coordinate (V9) at (-1,0);

\coordinate (V52) at (3,3);
\coordinate (V62) at (-1,3);

\draw[red, very thick] (V7) -- (V22);
\draw[red, very thick] (V62) -- (V5);

\draw[red, very thick] (V52) -- (V6);

\draw[blue, very thick] (V5) -- (V6);

\draw[green, very thick] (V6) -- (V4);
\draw[blue, very thick] (V4) -- (V7);

\draw[red, very thick] (V3) -- (V4);
\draw[green, very thick] (V3) -- (V5);
\draw[blue, very thick] (V3) -- (V1);

\draw[green, very thick] (V7) -- (V21);

\draw[red, very thick] (V9) -- (V1);
\draw[green, very thick] (V22) -- (V1);
\draw[blue, very thick] (V22) -- (V21);
\draw[red, very thick] (V8) -- (V21);

\vertexLabelR{V1}{left}{$ $}
\vertexLabelR{V21}{left}{$ $}
\vertexLabelR{V22}{left}{$ $}
\vertexLabelR{V52}{left}{$ $}
\vertexLabelR{V62}{left}{$ $}
\vertexLabelR{V3}{left}{$ $}
\vertexLabelR{V4}{left}{$ $}
\vertexLabelR{V5}{left}{$ $}
\vertexLabelR{V6}{left}{$ $}
\vertexLabelR{V7}{left}{$ $}
\vertexLabelR{V8}{left}{$ $}
\vertexLabelR{V9}{left}{$ $}

\node at (0., 3.2) {$x_{(g,1)}$};
\node at (2., 3.2) {$x_{(g,2)}$};
\node at (3., 3.2) {$x_{(g\cdot\sigma,1)}$};
\node at (-1., 3.2) {$x_{(g\cdot\sigma^{(-1)},2)}$};
\node at (-.5, 2.) {$x_{(g,3)}$};
\node at (2.5, 2.) {$x_{(g,4)}$};
\node at (-1, .2) {$x_{(g\cdot \sigma^{(-1)},8)}$};

\node at (0, -.2) {$x_{(g,6)}$};
\node at (1.5, -.2) {$x_{(g,7)}$};
\node at (2.5, -.2) {$x_{(g,8)}$};
\node at (3.5, -.2) {$x_{(g\cdot\sigma,6)}$};
\node at (2.5, 1) {$x_{(g,5)}$};

\end{tikzpicture}
    \caption{A subgraph of the cubic graph $\Gamma_{C_\ell}$ equipped with an $\Aut(\Gamma_{C_\ell})$-invariant 3-edge-colouring}
    \label{fig:cn}
\end{figure}
\end{proof}

\subsection{Cubic graphs with non-cyclic automorphism groups }
\label{section:correction}
Here, we briefly illustrate the construction of a cubic graph with prescribed automorphism group that has been established in \cite{SurfacesWithAuto}. Note that this graph construction is inspired by Frucht's work in \cite{Frucht}. For a more precise definition of the cubic graphs that are obtained from modifying Frucht's cubic graph construction, we refer the reader to \cite{SurfacesWithAuto}.

In order to describe the mentioned construction, let $G$ be a finite group that is generated by a set $S=\{g_1,\ldots,g_n\}$, where we assume that $S$ does not contain the identity of $G$. 
If we define $N_3(S)$ as $2n+6$ if $n>2$ and $6$ if $n=2,$ then the first two authors of this work construct a cubic graph $\Gamma_{G,S}$ in \cite{SurfacesWithAuto} that has vertices of the form $V(\Gamma_{G,S})=\{x_{(i,g)}\mid i=1,\ldots,N_3(S), g\in G \}.$ Hence, the graph $\Gamma_{G,S}$ has exactly $N_3(S)\cdot \vert G\vert$ vertices. In order to fully construct $\Gamma_{G,S}$ the authors are inspired by Frucht and make use of quadratic forms to define the corresponding graphs. For simplicity, we illustrate a component of the graph $\Gamma_{G,S}$ for $n=2$ in Figure~\ref{frucht2} and for $n\geq 3$ in Figure~\ref{frucht3}.

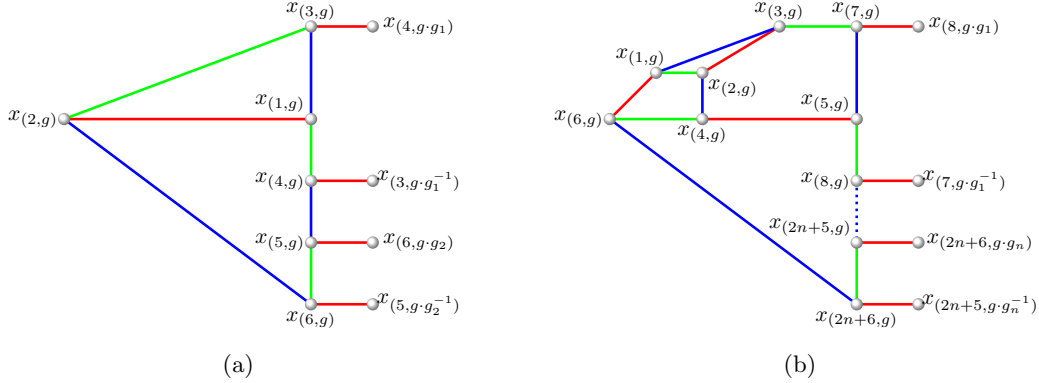
\begin{figure}[H]
\begin{minipage}{.49\textwidth} 
    \centering
    \resizebox{!}{4.5cm}{	\begin{tikzpicture}[vertexBall, edgeDouble=nolabels, faceStyle=nolabels, scale=0.5]

	\coordinate (V3) at (8 , 0);
	\coordinate (V4) at (0 , 0);
	\coordinate (V5) at (8 , 3);
	\coordinate (V6) at (8 , -2);
	\coordinate (V7) at (8 , -4);
	\coordinate (V8) at (8 , -6);
	\coordinate (V9) at (10 , 3);
	\coordinate (V10) at (10 , -2);
	\coordinate (V11) at (10, -4);
	\coordinate (V12) at (10 , -6);

	\draw[very thick,blue] (V4) -- (V8);
	\draw[very thick,green] (V8) -- (V7);
	\draw[very thick,red] (V4) -- (V3);

	\draw[very thick,blue] (V5) -- (V3);
	\draw[very thick,green] (V4) -- (V5);
	\draw[very thick,red] (V5) -- (V9);

	\draw[very thick,blue] (V6) -- (V7);
	\draw[very thick,green] (V3) -- (V6);
	\draw[very thick,red] (V7) -- (V11);
	\draw[very thick,red] (V8) -- (V12);
	\draw[very thick,red] (V6) -- (V10);

	\vertexLabelR[]{V3}{left}{$ $}
	\vertexLabelR[]{V4}{left}{$ $}
	\vertexLabelR[]{V5}{left}{$ $}
	\vertexLabelR[]{V6}{left}{$ $}
	\vertexLabelR[]{V7}{left}{$ $}
	\vertexLabelR[]{V8}{left}{$ $}
	\vertexLabelR[]{V9}{left}{$ $}
	\vertexLabelR[]{V10}{left}{$ $}
	\vertexLabelR[]{V11}{left}{$ $}
	\vertexLabelR[]{V12}{left}{$ $}

	\def\x{0.5}
	\node at (8-2*\x , 0+\x) {$x_{(1,g)}$};
	\node at (0-2*\x , 0) {$x_{(2,g)}$};
	\node at (8 , 3+\x) {$x_{(3,g)}$};
	\node at (8-2*\x , -2) {$x_{(4,g)}$};
	\node at (8-2*\x , -4) {$x_{(5,g)}$};
	\node at (8 , -6-\x) {$x_{(6,g)}$};
	\node at (10+3*\x , 3) {$x_{(4,g\cdot g_1)}$};
	\node at (10+3*\x , -2) {$x_{(3,g\cdot g_1^{-1})}$};
	\node at (10+3*\x, -4) {$x_{(6,g\cdot g_2)}$};
	\node at (10+3*\x , -6) {$x_{(5,g\cdot g_2^{-1})}$};

	\end{tikzpicture}}
    \subcaption{}
\label{frucht2}
\end{minipage}
\begin{minipage}{.49\textwidth} 
    \centering
    \resizebox{!}{4.5cm}{	\begin{tikzpicture}[vertexBall, edgeDouble=nolabels, faceStyle=nolabels, scale=0.5]

	\coordinate (V2) at (3 , 0);
	\coordinate (V3) at (8 , 0);
	\coordinate (V4) at (0 , 0);
	\coordinate (V5) at (8 , 3);
	\coordinate (V6) at (8 , -2);
	\coordinate (V7) at (8 , -4);
	\coordinate (V8) at (8 , -6);
	\coordinate (V9) at (10 , 3);
	\coordinate (V10) at (10 , -2);
	\coordinate (V11) at (10, -4);
	\coordinate (V12) at (10 , -6);

	\coordinate (V13) at (3 , 1.5);
	\coordinate (V14) at (1.5 , 1.5);
	\coordinate (V15) at (5.5 , 3);

	\draw[very thick,red] (V13) -- (V15);
	\draw[very thick,blue] (V14) -- (V15);
	\draw[very thick,green] (V13) -- (V14);	

	\draw[very thick,red] (V14) -- (V4);
	\draw[very thick,blue] (V2) -- (V13);
	\draw[very thick,green] (V2) -- (V4);

	\draw[very thick,blue] (V4) -- (V8);
	\draw[very thick,green] (V8) -- (V7);
	\draw[very thick,red] (V2) -- (V3);

	\draw[very thick,blue] (V5) -- (V3);
	\draw[very thick,green] (V15) -- (V5);
	\draw[very thick,red] (V5) -- (V9);

	\draw[very thick,blue, dotted] (V6) -- (V7);
	\draw[very thick,green] (V3) -- (V6);
	\draw[very thick,red] (V7) -- (V11);
	\draw[very thick,red] (V8) -- (V12);
	\draw[very thick,red] (V6) -- (V10);

	
	\vertexLabelR[]{V2}{left}{$ $}
	\vertexLabelR[]{V3}{left}{$ $}
	\vertexLabelR[]{V4}{left}{$ $}
	\vertexLabelR[]{V5}{left}{$ $}
	\vertexLabelR[]{V6}{left}{$ $}
	\vertexLabelR[]{V7}{left}{$ $}
	\vertexLabelR[]{V8}{left}{$ $}
	\vertexLabelR[]{V9}{left}{$ $}
	\vertexLabelR[]{V10}{left}{$ $}
	\vertexLabelR[]{V11}{left}{$ $}
	\vertexLabelR[]{V12}{left}{$ $}
	\vertexLabelR[]{V13}{left}{$ $}
	\vertexLabelR[]{V14}{left}{$ $}
	\vertexLabelR[]{V15}{left}{$ $}

	\def\x{0.5}
	\node at (1.5-\x , 1.5+\x) {$x_{(1,g)}$};
	\node at (3+2*\x , 1.5-\x) {$x_{(2,g)}$};
	\node at (5.5 , 3+\x) {$x_{(3,g)}$};

	\node at (3 , 0-\x) {$x_{(4,g)}$};
	\node at (8-2*\x , 0+\x) {$x_{(5,g)}$};
	\node at (0-2*\x , 0) {$x_{(6,g)}$};
	\node at (8 , 3+\x) {$x_{(7,g)}$};
	\node at (8-2*\x , -2) {$x_{(8,g)}$};
	\node at (8-3*\x , -4+\x) {$x_{(2n+5,g)}$};
	\node at (8 , -6-\x) {$x_{(2n+6,g)}$};
	\node at (10+3*\x , 3) {$x_{(8,g\cdot g_1)}$};
	\node at (10+3*\x , -2) {$x_{(7,g\cdot g_1^{-1})}$};
	\node at (10+4*\x, -4) {$x_{(2n+6,g\cdot g_n)}$};
	\node at (10+4*\x , -6) {$x_{(2n+5,g\cdot g_n^{-1})}$};

	\end{tikzpicture}}
    \subcaption{}
\label{frucht3}
\end{minipage}
\caption{Component of  $\Gamma_{G,S}$ for two generators (a) and for more than two generators (b)}
\label{fig:modfrucht}
\end{figure}
The graph $\Gamma_{G,S}$ satisfies $\Aut(\Gamma_{G,S})\cong G$, as established in \cite[Theorem 3.5]{SurfacesWithAuto}. Furthermore, in the proof of \cite[Theorem 6.1]{SurfacesWithAuto} it is shown that  $\Gamma_{G,S}$ has an $\Aut(\Gamma_{G,S})$-invariant 3-edge-colouring. Since the graph $\Gamma_{G,S}$ is clearly not isomorphic to any circular ladder graph, we know that $\Gamma_{G,S}$ is prime. Hence, we obtain the following reformulation of the result established by the first two authors of this work. 
\begin{proposition}\label{prop:fruchtresult}
    Let $G$ be a finite non-cyclic group that is generated by a set $S$ and $\Gamma_{G,S}$ the cubic graph satisfying $\Aut(\Gamma_{G,S})\cong G$ constructed in \cite{SurfacesWithAuto}. Then $\Gamma_{G,S}$ is a prime graph with $\Aut(\Gamma_{G,S})\cong G$ that has an $\Aut(\Gamma_{G,S})$-invariant $3$-edge-colouring.
\end{proposition}

\subsection{Regular graphs with prescribed automorphism groups}\label{section:regular}
Finally, in this section, we consider the general case. We begin by showing that for $ d = 3,4,5 $, there exist $ d $-regular graphs with the desired properties, namely a prescribed automorphism group and an automorphism group-invariant $d$-edge-colouring. We then use graph products to extend the result to all $ d \geq 3 $.

The following lemma will be useful in the proofs of Proposition~\ref{theorem:d=5} and Theorem~\ref{theorem:autom_strongEmbedding}. In particular, it shows that automorphism group-invariant  $ d $-edge-colourings of a given $ d $-regular graph give rise to corresponding strong embeddings that are invariant under the same automorphism group.

\begin{lemma} \label{lemma:colouring_CDC}
    Let $\Gamma$ be a $d$-regular graph. If $\Gamma$ has an $\Aut(\Gamma)$-invariant $d$-edge-colouring $\kappa:~E(\Gamma)\to \{1,\ldots,d\}$, then $\Gamma$ admits a facial $\Aut(\Gamma)$-invariant cycle double cover. 
\end{lemma}
\begin{proof}
 
Let $g\in S_d$ be a cycle of length $d$ and $g_i:=g^i$ with $1\leq i \leq d.$ 
Furthermore,  for $1\leq i \leq d$  let $\Gamma^i$ be the subgraph of $\Gamma$ that is induced by the edges $$\kappa^{-1}(1^{g_i})\cup \kappa^{-1}(1^{g_{i+1}}).$$ 
Since $\kappa$ is a proper edge-colouring of $\Gamma$, the graph $\Gamma^i$ is 2-regular. 
Hence, $\Gamma^i$ gives rise to a disjoint union of cycles $\gamma^i_1,\ldots,\gamma^i_{\ell_i}$ in $\Gamma$, where $\ell_i$ is a natural number. Since the edge-colouring is proper, every edge $e\in E(\Gamma)$ is contained in exactly two of the subgraphs $\Gamma^1,\ldots,\Gamma^d$ which means that $e$ is contained in a bi-coloured cycle with colours $(\kappa(e),\kappa(e)^g)$ and another bi-coloured cycle with colours $(\kappa(e),\kappa(e)^{g^{-1}}).$
Consequently, the union $$\mathcal{C}:=\bigcup_{i=1}^d \{\gamma^i_1,\ldots,\gamma^i_{\ell_i}\}$$ forms a cycle double cover of $\Gamma$.
Moreover, the cycle $g$ describes a local orientation at each vertex. Indeed, for each vertex $v\in V(\Gamma)$, the proper colouring $\kappa$ gives us a unique incident edge $e_i\in E(\Gamma)$ of colour $1^{g_i}$. We therefore declare the local rotation at $v$ to be $(e_1,e_2,\ldots,e_d)$.
Hence the facial walks are exactly the bi-coloured cycles in $\mathcal{C}$ constructed above.

As assumed, the edge-colouring $\kappa$ is $\Aut(\Gamma)$-invariant. This means that $\kappa(e)=\kappa(\phi(e))$ for every edge $e\in E(\Gamma)$ and every $\phi\in \Aut(\Gamma) $. Since a cycle $\gamma$ in the above cycle double covers consists of edges which are mapped onto the colours $1^{g_i},1^{g_{i+1}}$ under $\kappa$, we know that the same holds for $\phi(\gamma)$ for all $\phi\in\Aut(\Gamma)$. Hence, $\phi(\gamma)\in \mathcal{C}$ for all $\phi\in \Aut(\Gamma)$ which concludes the proof.
\end{proof}
We give the following remark.
\begin{remark}\label{rem:colouring}
In the above proof we have seen that any cycle $g\in S_d$ of length $d$ induces a facial cycle double cover of a $d$-regular graph admitting a $d$-edge-colouring. By examining the construction that is given in the proof of Lemma~\ref{lemma:colouring_CDC}, we see that $g$ and $g^{-1}$ construct the same cycle double cover. Even more, we know that two cycle double covers that are constructed by two cycles $g_1,g_2\in S_d$, are identical if and only if $g_1=g_2$ or $g_1=g_2^{-1}$. Thus, a $d$-colouring of a $d$-regular graph $\Gamma$ gives rise to exactly $\frac{(d-1)!}{2}$  facial cycle double covers that consist of bi-coloured cycles.
\end{remark}

Let us now focus on constructing a $4$-regular graph whose automorphism group realises a given finite group $G$.

\begin{proposition}\label{theorem:d=4}
        For every finite group $G$ there exists a $4$-regular graph $\Gamma^{(G,4)}$ that has an $\Aut(\Gamma^{(G,4)})$-invariant $4$-edge-colouring and satisfies $\Aut(\Gamma^{(G,4)})\cong G$.
\end{proposition}
\begin{proof}
Let $G$ be an arbitrary finite group generated by a set $S$. Since the case $\vert G\vert \leq 2$ has been proven in Proposition~\ref{prop:atmost22}, we can assume $\vert G\vert >2.$
By Proposition~\ref{prop:fruchtresult} and Proposition~\ref{prop:cyclicauto} we know that there exists a cubic prime graph $\Gamma$ with $\Aut(\Gamma)\cong G$ that has an $\Aut(\Gamma)$-invariant edge-colouring $\kappa:E(\Gamma)\to \{1,2,3\}.$  
We recall that $\Gamma$ has vertices $x_{(i,g)}$ with $g\in G$ and $i=1,\ldots,\hat{N}_3(S)$, where $\hat{N}_3(S)$ is equal to (1) $8$ if $n =1$, (2) $6$ if $n=2$, and (3) $2n+6$ if $n>2$. For the definition of the corresponding edges $E(\Gamma)$ we refer the reader to~\Cref{subsection:cyclic,section:correction}. 
We exploit the description of $\Gamma$ given in the aforementioned sections and modify this cubic graph to construct a $4$-regular graph with the desired properties.

As a first step, we construct the complete truncation $\Gamma^{(G,3)}:=\mathcal{T}(\Gamma)$ and label the vertices by making use of our edge 3-colouring $\kappa$ as follows:
For each vertex $x_{(i,g)}\in V(\Gamma)$, where $i=1,\ldots,\hat{N}_3(S)$ and $g\in G,$ we obtain three vertices $x_{(i,g,j)},$ where $j=1,2,3,$ in the complete truncation. We enforce that two vertices  $x_{(i_1,g_1,j_1)}$ and $x_{(i_2,g_2,j_2)}$ in $V(\Gamma^{(G,3)})$
are connected by an edge, if (1) $i_1=i_2$, $g_1=g_2$ and $j_1\neq j_2$, or (2) $e:=\{x_{(i_1,g_1)},x_{(i_2,g_2)}\}\in E(\Gamma)$ and $\kappa(e)=j_1=j_2$, see \Cref{fig:t3}. In this and the following figures, we draw the colours $1,2,3$ of the different edge-colourings with the colours red, green and blue, respectively.
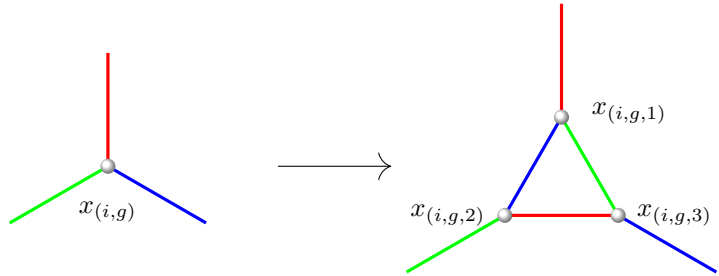
\begin{figure}[H]
    \centering
\begin{tikzpicture}[vertexBall, edgeDouble=nolabels, faceStyle=nolabels, scale=1.5]

\coordinate (V1) at (0+9 , 0);
\coordinate (V2) at (1+9 , 0);
\coordinate (V3) at (0.5+9 , 0.866025);
\coordinate (V4) at (1/2+9 , 1.866025);
\coordinate (V5) at (1.86603 +9, -0.5);
\coordinate (V6) at (-0.86603+9 , -0.5);

\draw[very thick,blue] (V1) -- node[edgeLabel] {$1$} (V3);
\draw[very thick,green] (V2) -- node[edgeLabel] {$2$} (V3);
\draw[very thick,red] (V1) -- node[edgeLabel] {$3$} (V2);
\draw[very thick,green] (V1) -- node[edgeLabel] {$4$} (V6);
\draw[very thick,red] (V3) -- node[edgeLabel] {$5$} (V4);
\draw[very thick,blue] (V2) -- node[edgeLabel] {$6$} (V5);

\vertexLabelR[]{V1}{left}{$ $}
\vertexLabelR[]{V2}{left}{$ $}
\vertexLabelR[]{V3}{left}{$ $}

\draw[-{To[scale=2]}] (7,0.433013) to (8,0.433013);

\coordinate (V7) at (0.5+5 , 0.433013);
\coordinate (V8) at (1/2+5 , 1.433013);
\coordinate (V9) at (1.36603+5 , -0.0669873);
\coordinate (V10) at (-0.36603+5 , -0.0669873);

\draw[very thick,red] (V7) -- (V8);
\draw[very thick,blue] (V7) -- (V9);
\draw[very thick,green] (V7) -- (V10);

\vertexLabelR[]{V7}{left}{$ $}

\node at (5.5, 0.033013) {$x_{(i,g)}$};

\node at (10.1, 0.9) {$x_{(i,g,1)}$};
\node at (10.5, 0.) {$x_{(i,g,3)}$};
\node at (8.5, 0.) {$x_{(i,g,2)}$};

\end{tikzpicture}
    \caption{The cubic graph $\Gamma$ with its 3-edge-colouring $\kappa$ (left) and the complete truncation $\Gamma^{(G,3)}$ with vertices labelled by exploiting $\kappa$ (right)}
    \label{fig:t3}
\end{figure}
It is easy to see that these incidences indeed give rise to the complete truncation $\Gamma^{(G,3)}$ of the graph $\Gamma$. 
Furthermore, by Proposition~\ref{prop:truncproper} we know that $\Gamma^{(G,3)}$  has an $\Aut(\Gamma^{(G,3)})$-invariant 3-edge-colouring $\kappa_3:E(\Gamma^{(G,3)})\to \{1,2,3\}.$
Next, we construct an auxiliary graph $Y$ that forms a cubic graph  with vertices of the form $y_{(i,g,j)}$, where $i=1,\ldots,\hat{N}_3(S), g\in G$ and $j=1,2,3$. The edges $E(Y)$ of this graph are given by $E(Y):=E_1\cup E_2 \cup E_3$ with 
\begin{align*}
&E_1:=\sum_{g\in G}\sum_{i=1}^{\hat{N}_3(S)/2}( y_{(2 i-1,g,1)}y_{(2i-1,g,2)}+y_{(2i-1,g,3)}y_{(2i,g,1)} + y_{(2i,g,2)}y_{(2i,g,3)}),\\
&E_2:=\sum_{g\in G}\sum_{i=1}^{\hat{N}_3(S)/2}( y_{(2 i,g,1)}y_{(2i,g,2)}+y_{(2i,g,3)}y_{(2i+1,g,1)} + y_{(2i-1,g,2)}y_{(2i-1,g,3)}),\\
&E_3:=\sum_{g\in G}\sum_{i=1}^{\hat{N}_3(S)/2}\sum_{j=1}^3 y_{(i,g,j)}y_{(\hat{N}_3(S)/2+i,g,j)},
\end{align*}
where subscripts are read modulo $\hat{N}_3(S).$ For every $g\in G$ we obtain a connected component of $Y$ containing $3\cdot \hat{N}_3(S)$ vertices as illustrated in Figure~\ref{fig:Y}.
Hence, the graph $Y$ is indeed a cubic graph satisfying $\vert V(Y)\vert = \vert V(\Gamma^{(G,3)})\vert $. Moreover, the partition $E(Y)=E_1\cup E_2\cup E_3$ yields a 3-edge-colouring $\kappa_Y$ of $Y$ by defining $\kappa_Y(e):=m$ for all $e\in E_m.$
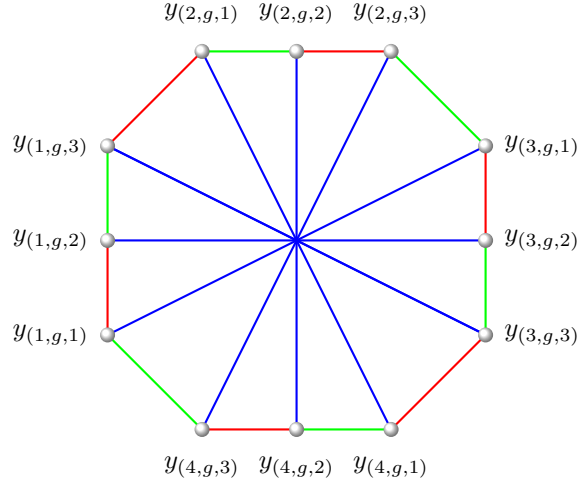
\begin{figure}[H]
    \centering
    \begin{tikzpicture}[vertexBall, edgeDouble=nolabels, faceStyle=nolabels, scale=2.5]

\coordinate (V11) at (-1,  0.5);
\coordinate (V2)  at (-1,  0.0);
\coordinate (V3)  at (-1, -0.5);

\coordinate (V4)  at (-0.5, -1);
\coordinate (V5)  at ( 0.0, -1);
\coordinate (V61) at ( 0.5, -1);

\coordinate (V62) at ( 1, -0.5);
\coordinate (V7)  at ( 1,  0.0);
\coordinate (V8)  at ( 1,  0.5);

\coordinate (V9)  at ( 0.5,  1);
\coordinate (V10) at ( 0.0,  1);
\coordinate (V12) at (-0.5,  1);

\coordinate (V11c) at (-1.2,  0.5);
\coordinate (V61c) at ( 0.5, -1.2);
\coordinate (V62c) at ( 1.2, -0.5);
\coordinate (V12c) at (-0.5,  1.2);


\draw[thick,green] (V11) -- (V2);
\draw[thick,red] (V3) -- (V2);
\draw[thick,green] (V3) -- (V4);
\draw[thick,red] (V5) -- (V4);
\draw[thick,green] (V61) -- (V5);

\draw[thick,green] (V62) -- (V7);
\draw[thick,red] (V8) -- (V7);
\draw[thick,green] (V8) -- (V9);
\draw[thick,red] (V10) -- (V9);
\draw[thick,green] (V10) -- (V12);

\draw[thick,red] (V11) -- (V12);
\draw[thick,red] (V61) -- (V62);

\draw[thick,blue] (V11) -- (V62);
\draw[thick,blue] (V2) -- (V7);
\draw[thick,blue] (V11) -- (V62);
\draw[thick,blue] (V3) -- (V8);

\draw[thick,blue] (V12) -- (V61);
\draw[thick,blue] (V10) -- (V5);
\draw[thick,blue] (V9) -- (V4);

\vertexLabelR{V11}{left}{$ $}
\vertexLabelR{V2}{left}{$ $}
\vertexLabelR{V3}{left}{$ $}
\vertexLabelR{V4}{left}{$ $}
\vertexLabelR{V5}{left}{$ $}

\vertexLabelR{V61}{left}{$ $}
\vertexLabelR{V62}{left}{$ $}

\vertexLabelR{V7}{left}{$ $}
\vertexLabelR{V8}{left}{$ $}
\vertexLabelR{V10}{left}{$ $}
\vertexLabelR{V12}{left}{$ $}
\vertexLabelR{V9}{left}{$ $}

\node[vertexBall,shift={(-0.75,0)}] at (V11) {$y_{(1,g,3)}$};
\node[vertexBall,,shift={(-0.75,0)}] at (V2) {$y_{(1,g,2)}$}; \node[vertexBall,shift={(-0.75,0)}] at (V3) {$y_{(1,g,1)}$};
\node[vertexBall,shift={(0,-0.5)}] at (V4) {$y_{(4,g,3)}$}; \node[vertexBall,shift={(0,-0.5)}] at (V5) {$y_{(4,g,2)}$};

 \node[vertexBall,shift={(0,-0.5)}] at (V61) {$y_{(4,g,1)}$};
  \node[vertexBall,shift={(0.75,0)}] at (V62) {$y_{(3,g,3)}$};
\node[vertexBall,shift={(0.75,0)}] at (V7) {$y_{(3,g,2)}$};
    \node[vertexBall,shift={(0.75,0)}] at (V8) {$y_{(3,g,1)}$};
  \node[vertexBall,shift={(0,0.5)}] at (V10) {$y_{(2,g,2)}$};
    \node[vertexBall,shift={(0,0.5)}] at (V12) {$y_{(2,g,1)}$};
      \node[vertexBall,shift={(0,0.5)}] at (V9) {$y_{(2,g,3)}$};

\end{tikzpicture}
    \caption{A connected component of the cubic graph $Y$. Note that for simplicity, we illustrate this graph having $3\cdot 4$ vertices. In our construction this graph contains $3\cdot \hat{N}_3(S)$ vertices.}
    \label{fig:Y}
\end{figure}

Now, we can use $\Gamma^{(G,3)}$ and $Y$ to construct a 4-regular graph $\Gamma^{(G,4)}$ with vertices given by $V(\Gamma^{(G,4)}):=V(\Gamma^{(G,3)})\cup V(Y)$ and edges given by
$E(\Gamma^{(G,4)}):=E(\Gamma^{(G,3)})\cup E(Y)\cup E_4,$ where $E_4$ is the set of edges that corresponds to the quadratic form $$\sum_{g\in G}\sum_{i=1}^{\hat{N}_3(S)}\sum_{j=1}^3x_{(i,g,j)}y_{(i,g,j)}.$$
We observe that the left-regular action of $G$ induces automorphisms of  $\Gamma^{(G,4)}.$ 
In particular, for each $h\in G$ we obtain an automorphism $\phi_h$ of $\Gamma^{(G,4)}$ via $\phi_{h}(x_{(i,g,j)}):=x_{(i,h\cdot g,j)}$ and $\phi_{h}(y_{(i,g,j)}):=y_{(i,h\cdot g,j)}.$ Hence, $G\hookrightarrow \Aut(\Gamma^{(G,4)}).$
It remains to show that every automorphism $\phi\in \Aut(\Gamma^{(G,4)})$ satisfies $\phi=\phi_h$ for a $h\in G.$
This follows from the following fact: Since $\phi$ is an automorphism of $\Gamma^{(G,4)}$, it has to map a cycle of length $3$ onto another cycle of the same length. 
Note that since $\Gamma^{(G,3)}$ forms the complete truncation of a cubic graph, every vertex $x_{(i,g,j)}\in V(\Gamma^{(G,4)})$ is contained in a cycle of length $3$. Moreover, we see that the vertices $y_{(i,g,j)}\in V(\Gamma^{(G,4)})$ are not contained in cycles of length 3. Hence, $V(\Gamma^{(G,3)})$ and $V(Y)$ considered as subsets of $V(\Gamma^{(G,4)})$ satisfy $V(\Gamma^{(G,3)})^{\Aut(\Gamma^{(G,4)})}= V(\Gamma^{(G,3)})$ and $V(Y)^{\Aut(\Gamma^{(G,4)})}= V(Y)$. Even more, we have 
\begin{align*}
   & E(\Gamma^{(G,3)})^{\Aut(\Gamma^{(G,4)})}= E(\Gamma^{(G,3)}),\\
   &E(Y)^{\Aut(\Gamma^{(G,4)})}= E(Y), \\&
   E_4^{\Aut(\Gamma^{(G,4)})}=E_4.
\end{align*}
Thus, the restriction $\phi'$ of an automorphism $\phi\in \Aut(\Gamma^{(G,4)})$ to $V(\Gamma^{(G,3)})\subseteq V(\Gamma^{(G,4)} )$ gives us an automorphism of $\Gamma^{(G,3)}.$
If we combine this statement with the fact that (1) $\Aut(\Gamma^{(G,3)})\cong G $ and (2) the edge-colouring $\kappa_3$ is $\Aut(\Gamma^{(G,3)})$-invariant,
we obtain that $\phi'(x_{(i,g,j)})= x_{(i,h\cdot g,j)}$ for some $h\in G$ and all $i=1,\ldots,\hat{N}_3(S),\,g\in G$ and $j=1,2,3$. Since incident vertices have to be mapped onto incident vertices, we obtain $\phi(y_{(i,g,j)})=y_{(i,h\cdot g, j)}$ for all $i=1,\ldots,\hat{N}_3(S),\,g\in G$ and $j=1,2,3$. Hence, $\phi =\phi_h$ and $\Aut(\Gamma^{(G,4)})$ is therefore isomorphic to $G$. With these properties, we see that the colouring $\kappa_4:E(\Gamma^{(G,4)})\to \{1,\ldots,4\}$ defined by 
\[
\kappa_4(e)=\begin{cases}
\kappa_3(e), & \text{if } e\in E(\Gamma^{(G,3)}), \\
\kappa_Y(e), & \text{if } e\in E(Y), \\
4, & \text{if } e\in E_4 \\
\end{cases}
\]
forms a proper 4-edge-colouring of $\Gamma^{(G,4)}$ that is $\Aut(\Gamma^{(G,4)})$-invariant.
\end{proof}

In Proposition~\ref{theorem:d=5},
we will use the above construction of the $4$-regular graph $\Gamma^{(G,4)}$ for a given group $G$
to obtain a $5$-regular graph with the desired properties. To simplify the proof, we begin with the following remark.

\begin{remark}\label{rem:ms}
In Proposition~\ref{theorem:d=4} we have constructed a $4$-regular graph $\Gamma^{(G,4)}$ with a given finite group $G$ as its automorphism group. We have additionally proven that $\Gamma^{(G,4)}$ has an $\Aut(\Gamma^{(G,4)})$-invariant $4$-edge-colouring.
Recall that the graph $\Gamma^{(G,4)}$ has vertices $x_{(i,g,j)}$ and $y_{(i,g,j)}$ with $g\in G$, $i=1,\ldots,\hat{N}_3(S)$ and $j=1,2,3$. Note that $\hat{N}_3(S)$ is equal to $(1)$ $8$ if $\vert S\vert =1$, $(2)$ $6$ if $\vert S\vert =2$, and $(3)$ $2n+6$ if $\vert S\vert >2$. For a definition of the edges $E(\Gamma^{(G,4)})$ we refer the reader to the aforementioned proof.
In the following, we will make use of this $4$-regular graph to show that there exists a $5$-regular graph $\Gamma^{(G,5)}$ with an automorphism group isomorphic to $G$ and a corresponding $\Aut(\Gamma^{(G,5)})$-invariant $5$-edge-colouring.
For this, we relabel the vertices of the graph to obtain a simplified proof. In particular, a vertex $x_{(i,g,j)}$ with $g\in G$, $i=1,\ldots,\hat{N}_3(S)$ and $j=1,2,3$ is relabelled as $w_{(6(i-1)+j,g)}$ and  a vertex $y_{(i,g,j)}$ with $g\in G$, $i=1,\ldots,\hat{N}_3(S)$ and $j=1,2,3$ is relabelled as $w_{(6(i-1)+3+j,g)}.$
Hence, after relabelling the graph $\Gamma^{(G,4)}$ has vertices of the form $w_{(i,g)}$ with $i=1,\ldots,N_4(S)$ and $g\in G$, where $N_4(S)$ is equal to $(1)$ $6\cdot 8$ if $\vert S\vert=1$, $(2)$ $6\cdot 6$ if $\vert S\vert=2$, and $(3)$ $6\cdot(2n+6)$ if $n>2$. We change the edges of $\Gamma^{(G,4)}$ accordingly.
\end{remark}
Now, we are able to show that for every finite group $G$ there exists a $5$-regular graph whose automorphism group realises $G$.
\begin{proposition}\label{theorem:d=5}
        For every finite group $G$ there exists a $5$-regular graph $\Gamma^{(G,5)}$ with $\Aut(\Gamma^{(G,5)})\cong G$ that has an $\Aut(\Gamma^{(G,5)})$-invariant $5$-edge-colouring.
\end{proposition}

\begin{proof}

Let $G$ be an arbitrary finite group generated by a set $S$. Again, because of Proposition~\ref{prop:atmost22}, we can assume $\vert G\vert >2.$
By Proposition~\ref{theorem:d=4}, $\Gamma^{(G,4)}$ is a 4-regular graph with $\Aut(\Gamma^{(G,4)})\cong G$ that has an $\Aut(\Gamma^{(G,4)})$-invariant 4-edge-colouring $\kappa:E(\Gamma^{(G,4)})\to \{1,\ldots,4\}.$ 
As a first step, we use the graph $\Gamma^{(G,4)}$ to construct another $4$-regular graph $\Gamma'$ with $\Aut(\Gamma') \cong G$ that admits an $\Aut(\Gamma')$-invariant $4$-colouring. We then use $\Gamma'$ to construct the graph $\Gamma^{(G,5)}$ with the desired properties, proceeding as follows:

With Remark~\ref{rem:ms}, we know that $\Gamma^{(G,4)}$ has vertices of the form $w_{(i,g)}$, where $g\in G$ and $i=1,\ldots, N_4(S)$. We refer the reader to the mentioned remark for a definition of $N_4(S)$.
Since $\kappa$ is $\Aut(\Gamma^{(G,4)})$-invariant, Lemma~\ref{lemma:colouring_CDC} implies that $\Gamma^{(G,4)}$ has a strong embedding $\beta$ satisfying $\Aut(\beta(\Gamma^{(G,4)}))\cong \Aut(\Gamma^{(G,4)})\cong G.$ In particular, we choose $\beta$ to be the strong embedding so that the corresponding facial cycle double cover consists of all $(1,2)$-, $(2,3)$-, $(3,4)$- and $(1,4)$-coloured cycles with respect to the colouring $\kappa.$
We use $\beta$ to construct the complete truncation $\Gamma:=\mathcal{T}(\Gamma^{(G,4)},\beta)$ and label the vertices by making use of the above colouring:
For each vertex $w_{(i,g)}\in V(\Gamma)$ with $i=1,\ldots,N_4(S)$ and $g\in G,$ we obtain four vertices $w_{(i,g,j)}$ for $j=1,\ldots,4$ in the complete truncation.
We enforce that two vertices  $w_{(i_1,g_1,j_1)}$ and $w_{(i_2,g_2,j_2)}$ in $\Gamma$
are connected by an edge, if 
\begin{enumerate}
    \item $i_1=i_2$, $g_1=g_2$ and $\{j_1,j_2\}\in \{\{1,2\},\{2,3\},\{3,4\},\{1,4\}\}$ or 
    \item $e:=\{w_{(i_1,g_1)},w_{(i_2,g_2)}\}\in E(\Gamma^{(G,4)})$ and $\kappa(e)=j_1=j_2$.
\end{enumerate}
We observe that these incidences indeed give rise to the complete truncation of the graph $\Gamma^{(G,4)}.$
The resulting graph $\Gamma$ is presented in Figure~\ref{fig:truncate4}. In this and the below figures, we illustrate the colours $1,2,3,4$ of the different edge-colourings with the colours red, green, blue and brown respectively.
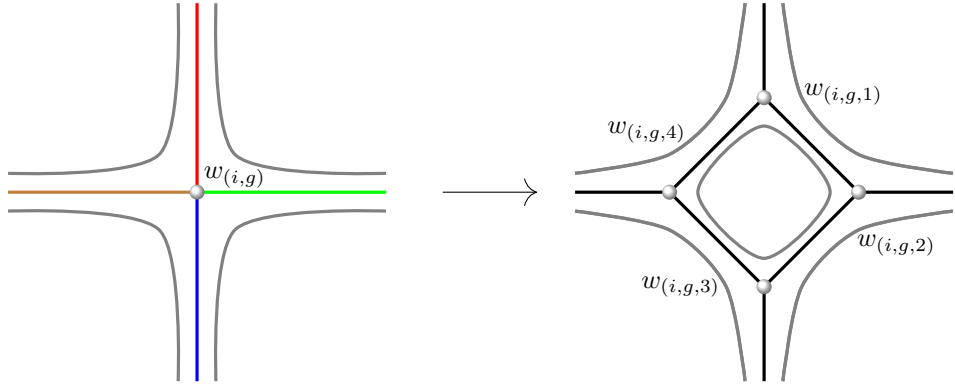
\begin{figure}[H]
    \centering
    \begin{tikzpicture}[vertexBall, edgeDouble=nolabels, scale=2.5]

\coordinate (V1) at (0 , 0);
\coordinate (V2) at (1 , 0);
\coordinate (V3) at (0 , 1);
\coordinate (V4) at (-1 , 0);
\coordinate (V5) at (0, -1);

\draw[-,very thick,green] (V1) to (V2);
\draw[-,very thick,red] (V1) to (V3);
\draw[-,very thick,brown] (V1) to (V4);
\draw[-,very thick,blue] (V1) to (V5);

\vertexLabelR[]{V1}{left}{$ $}

\draw[-{To[scale=2]}] (1.3,0) to (1.8,0);

\coordinate (V12) at (0.5+3 , 0);
\coordinate (V22) at (0 +3, 0.5);
\coordinate (V32) at (-0.5+3 , 0);
\coordinate (V42) at (3, -0.5);

\coordinate (V13) at (1 +3, 0);
\coordinate (V23) at (3, 1);
\coordinate (V33) at (-1 +3, 0);
\coordinate (V43) at (3, -1);

\coordinate (V14) at (0.5+6 , 0);

\coordinate (V24) at (0 +6, 0.5);
\coordinate (V34) at (-0.5+6 , 0);
\coordinate (V44) at (6, -0.5);

\coordinate (V15) at (1 +6, 0);
\coordinate (V25) at (6, 1);
\coordinate (V35) at (-1 +6, 0);
\coordinate (V45) at (6, -1);

\coordinate (V16) at (6 , 0);

\draw[-,very thick] (V12) to (V13);
\draw[-,very thick] (V22) to (V23);
\draw[-,very thick] (V32) to (V33);
\draw[-,very thick] (V42) to (V43);

\draw[-,very thick] (V12) to (V22);
\draw[-,very thick] (V22) to (V32);
\draw[-,very thick] (V32) to (V42);
\draw[-,very thick] (V42) to (V12);

 \draw[very thick, gray, smooth ]
  plot coordinates {
    (0.1,1)
    (0.2,0.2)
    (1,0.1)
  };

\draw[very thick, gray, smooth ]
  plot coordinates {
    (-0.1,1)
    (-0.2,0.2)
    (-1,0.1)
  };

  \draw[very thick, gray, smooth ]
  plot coordinates {
    (-0.1,-1)
    (-0.2,-0.2)
    (-1,-0.1)
  };

  \draw[very thick,gray, smooth ]
  plot coordinates {
    (0.1,-1)
    (0.2,-0.2)
    (1,-0.1)
  };

\vertexLabelR[]{V12}{left}{$ $}
\vertexLabelR[]{V22}{left}{$ $}
\vertexLabelR[]{V32}{left}{$ $}
\vertexLabelR[]{V42}{left}{$ $}

 \draw[very thick, gray, smooth ]
  plot coordinates {
    (0.1+3,1)
    (0.2+3,0.5)
        (0.5+3,0.1+0.1)
    (1+3,0.1)
  };

\draw[very thick, gray, smooth ]
  plot coordinates {
    (-0.1+3,1)
    (-0.2+3,0.5)
    (-0.5+3,0.1+0.1)
    (-1+3,0.1)
  };

  \draw[very thick, gray, smooth ]
  plot coordinates {
    (-0.1+3,-1)
        (-0.2+3,-0.5)
    (-0.5+3,-0.2)
    (-1+3,-0.1)
  };

    \draw[very thick, gray, smooth ]
  plot coordinates {
    (0.1+3,-1)
    (0.2+3,-0.5)
    (0.5+3,-0.2)
    (1+3,-0.1)
  };

\draw[very thick, gray, smooth cycle ]
  plot coordinates {
    (0.35+3,0)
    (3,-0.35)
    (-0.35+3,0)
    (3,0.35)
  };

 \draw[very thick, gray, smooth ]
  plot coordinates {
    (0.1+3,1)
    (0.2+3,0.5)
        (0.5+3,0.1+0.1)
    (1+3,0.1)
  };

\draw[very thick, gray, smooth ]
  plot coordinates {
    (-0.1+3,1)
    (-0.2+3,0.5)
    (-0.5+3,0.1+0.1)
    (-1+3,0.1)
  };

  \draw[very thick, gray, smooth ]
  plot coordinates {
    (-0.1+3,-1)
        (-0.2+3,-0.5)
    (-0.5+3,-0.2)
    (-1+3,-0.1)
  };

    \draw[very thick, gray, smooth ]
  plot coordinates {
    (0.1+3,-1)
    (0.2+3,-0.5)
    (0.5+3,-0.2)
    (1+3,-0.1)
  };

      \node[vertexBall,shift={(0.5,0.2)}] at (V1) {$w_{(i,g)}$};

    \node[vertexBall,shift={(-0.2,1.3)}] at (V12) {$w_{(i,g,1)}$};

\node[vertexBall,shift={(-0.3,0.8)}] at (V32) {$w_{(i,g,4)}$};

\node[vertexBall,shift={(-1.1,0.)}] at (V42) {$w_{(i,g,3)}$};

\node[vertexBall,shift={(0.5,-0.7)}] at (V12) {$w_{(i,g,2)}$};

\end{tikzpicture}
    \caption{The complete truncation $\Gamma$ of the 4-regular graph $\Gamma^{(G,4)}$}
    \label{fig:truncate4}
\end{figure}
By Proposition~\ref{prop:autiso}, we know that $\Gamma$ satisfies $\Aut(\Gamma)\cong \Aut(\Gamma^{(G,4)})\cong G.$ 
Note that $\Gamma$ forms a cubic graph. Hence, in the following, we further modify $\Gamma$ to construct a 4-regular graph $\Gamma'$ that has an $\Aut(\Gamma')$-invariant edge-colouring with $\Aut(\Gamma')\cong G$. For this, we define the vertices of $\Gamma'$ via $V(\Gamma'):=V(\Gamma)\cup \{w_{(i,g,5)}\mid i=1,\ldots,N_4(S),\,g\in G\}$
and edges $E(\Gamma'):=E(\Gamma)\cup E'$,
where $E'$ are edges corresponding to the below quadratic form:
\begin{align*}
E':=\sum_{g\in G}\sum_{i=1}^{N_4(S)}\sum_{j=1}^4w_{(i,g,j)}w_{(i,g,5)}.
\end{align*}
A component of this graph is illustrated in Figure~\ref{fig:gamma5}.
Note that every automorphism $\phi\in \Aut(\Gamma)$ can be translated into an automorphism $\psi\in \Aut(\Gamma')$. Hence, $G\cong \Aut(\Gamma)\hookrightarrow \Aut(\Gamma').$
 Furthermore, a vertex $w_{(i,g,5)}$ is contained in four cycles of length 3, namely 
 \begin{align*}
 &(w_{(i,g,1)},w_{(i,g,2)},w_{(i,g,5)}), (w_{(i,g,2)},w_{(i,g,3)},w_{(i,g,5)}), \\&(w_{(i,g,3)},w_{(i,g,4)},w_{(i,g,5)}),
(w_{(i,g,1)},w_{(i,g,4)},w_{(i,g,5)}).
 \end{align*}
 
Since the complete truncation $\Gamma$ of $\Gamma^{(G,4)}$ contains no cycles of length 3, the vertices $w_{(i,g,j)}$ for $1\leq j\leq 4$ are contained in exactly two cycles of length $3$. Thus, an automorphism $\psi\in \Aut(\Gamma')$ has to permute the vertices $w_{(i,g,5)}$. 
Thus, $\psi$ can be translated into an automorphism $\phi\in \Aut(\Gamma)$ and we obtain $G\cong \Aut(\Gamma').$ More precisely, for every $h\in G$ we obtain exactly one $\phi_h\in \Aut(\Gamma')$ if we set $\phi_h(w_{(i,g,j)}):=w_{(i,h\cdot g,j )}$ for all $i=1,\ldots,N_4(S),\,g\in G$ and $j=1,\ldots,5$.
Hence, we can define an $\Aut(\Gamma')$-invariant 4-edge-colouring of $\Gamma'$ via
$\kappa':E(\Gamma')\to \{1,2,3,4\}$ defined by 
\[
\kappa'(e')=\begin{cases}
j, & \text{if } e'=\{w_{(i,g,j)},w_{(i',g',j)}\}, \\
1, & \text{if } e'=\{w_{(i,g,2)},w_{(i,g,3)}\},\{w_{(i,g,4)},w_{(i,g,5)}\}, \\
2, & \text{if } e'=\{w_{(i,g,3)},w_{(i,g,4)}\},\{w_{(i,g,1)},w_{(i,g,5)}\}, \\
3, & \text{if } e'=\{w_{(i,g,1)},w_{(i,g,4)}\}, \{w_{(i,g,2)},w_{(i,g,5)}\}, \\
4, & \text{if } e'=\{w_{(i,g,1)},w_{(i,g,2)}\},\{w_{(i,g,3)},w_{(i,g,5)}\}.\\
\end{cases}
\]
The section of graph $\Gamma'$ with its $4$-edge colouring $\kappa'$ is illustrated in Figure~\ref{fig:gamma5}.
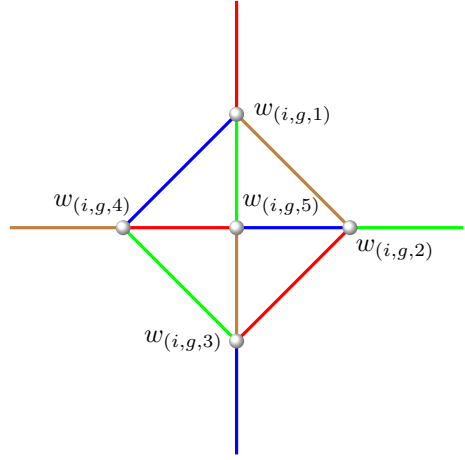
\begin{figure}[H]
    \centering
    \begin{tikzpicture}[vertexBall, edgeDouble=nolabels, scale=3]

\coordinate (V1) at (0 , 0);
\coordinate (V2) at (1 , 0);
\coordinate (V3) at (0 , 1);
\coordinate (V4) at (-1 , 0);
\coordinate (V5) at (0, -1);

\coordinate (V12) at (0.5+3 , 0);
\coordinate (V22) at (0 +3, 0.5);
\coordinate (V32) at (-0.5+3 , 0);
\coordinate (V42) at (3, -0.5);

\coordinate (V13) at (1 +3, 0);
\coordinate (V23) at (3, 1);
\coordinate (V33) at (-1 +3, 0);
\coordinate (V43) at (3, -1);

\coordinate (V14) at (0.5+6 , 0);

\coordinate (V24) at (0 +6, 0.5);
\coordinate (V34) at (-0.5+6 , 0);
\coordinate (V44) at (6, -0.5);

\coordinate (V15) at (1 +6, 0);
\coordinate (V25) at (6, 1);
\coordinate (V35) at (-1 +6, 0);
\coordinate (V45) at (6, -1);

\coordinate (V16) at (6 , 0);

\draw[-,very thick,green] (V14) to (V15);
\draw[-,very thick,red] (V24) to (V25);
\draw[-,very thick,brown] (V34) to (V35);
\draw[-,very thick,blue] (V44) to (V45);

\draw[-,very thick,blue] (V16) to (V14);
\draw[-,very thick,green] (V16) to (V24);
\draw[-,very thick,red] (V16) to (V34);
\draw[-,very thick,brown] (V16) to (V44);

\draw[-,very thick,brown] (V14) to (V24);
\draw[-,very thick,red] (V14) to (V44);
\draw[-,very thick,blue] (V24) to (V34);
\draw[-,very thick,green] (V34) to (V44);

  \vertexLabelR[]{V14}{left}{$ $}
\vertexLabelR[]{V24}{left}{$ $}
\vertexLabelR[]{V34}{left}{$ $}
\vertexLabelR[]{V44}{left}{$ $}

  \vertexLabelR[]{V16}{left}{$ $}

\node[vertexBall,shift={(-0.75,1.5)}] at (V14) {$w_{(i,g,1)}$};

\node[vertexBall,shift={(-0.4,0.3)}] at (V34) {$w_{(i,g,4)}$};
\node[vertexBall,shift={(-0.7,0.)}] at (V44) {$w_{(i,g,3)}$};

\node[vertexBall,shift={(0.6,-0.3)}] at (V14) {$w_{(i,g,2)}$};

\node[vertexBall,shift={(-0.9,0.3)}] at (V14) {$w_{(i,g,5)}$};

\end{tikzpicture}
    \caption{A 4-edge-colouring of the 4-regular graph $\Gamma'$}
    \label{fig:gamma5}
\end{figure}
Now, we construct an auxiliary graph $Z$ that forms a 4-regular graph with vertices of the form $z_{(i,g,j)}$, where $i=1,\ldots,N_4(S),\, g\in G$ and $j=1,\ldots,5$. The edges $E(Z)$ of this graph are given by $E(Z):=E_1\cup E_2 \cup E_3\cup E_4,$ where

\begin{align*}
E_1:=&\sum_{g\in G}\sum_{i=1}^{N_4(S)/2}( z_{(2 i-1,g,1)}z_{(2i-1,g,2)}+z_{(2i-1,g,3)}z_{(2i-1,g,4)} + z_{(2i-1,g,5)}z_{(2i,g,1)})\\
+&\sum_{g\in G}\sum_{i=1}^{N_4(S)/2}( z_{(2 i,g,2)}z_{(2i,g,3)}+z_{(2i,g,4)}z_{(2i,g,5)}),\\
E_2:=&\sum_{g\in G}\sum_{i=1}^{N_4(S)/2}\sum_{j=1}^5z_{(2i-1,g,j)}z_{(2i,g,j)},\\ 
E_3:=&\sum_{g\in G}\sum_{i=1}^{N_4(S)/2}( z_{(2 i,g,1)}z_{(2i,g,2)}+z_{(2i,g,3)}z_{(2i,g,4)} + z_{(2i,g,5)}z_{(2i+1,g,1)})\\
+&\sum_{g\in G}\sum_{i=1}^{N_4(S)/2}( z_{(2 i-1,g,2)}z_{(2i-1,g,3)}+z_{(2i-1,g,4)}z_{(2i-1,g,5)}),\\
E_4:=&\sum_{g\in G}\sum_{i=1}^{N_4(S)/2}\sum_{j=1}^5z_{(2i,g,j)}z_{(2i+1,g,j)}.  
\end{align*}

For every $g\in G$ we obtain a connected component of $Z$ containing $5\cdot N_4(S)$ vertices as illustrated in Figure~\ref{fig:Z}.
\begin{figure}[H]
    \centering
    \begin{tikzpicture}[vertexBall, edgeDouble=nolabels, faceStyle=nolabels, scale=3.5]


\coordinate (R1) at (1,-2/3);
\coordinate (R2) at (1,-1/3);
\coordinate (R3) at (1,0);
\coordinate (R4) at (1,1/3);
\coordinate (R5) at (1,2/3);

\coordinate (T1) at (-2/3,1);
\coordinate (T2) at (-1/3,1);
\coordinate (T3) at (0,1);
\coordinate (T4) at (1/3,1);
\coordinate (T5) at (2/3,1);

\coordinate (L1) at (-1,2/3);
\coordinate (L2) at (-1,1/3);
\coordinate (L3) at (-1,0);
\coordinate (L4) at (-1,-1/3);
\coordinate (L5) at (-1,-2/3);

\coordinate (B1) at (2/3,-1);
\coordinate (B2) at (1/3,-1);
\coordinate (B3) at (0,-1);
\coordinate (B4) at (-1/3,-1);
\coordinate (B5) at (-2/3,-1);


\draw[thick,blue] (R1)--(R2);
\draw[thick,red]  (R2)--(R3);
\draw[thick,blue] (R3)--(R4);
\draw[thick,red]  (R4)--(R5);

\draw[thick,blue] (R5)--(T5);

\draw[thick,red]  (T5)--(T4);
\draw[thick,blue] (T4)--(T3);
\draw[thick,red]  (T3)--(T2);
\draw[thick,blue] (T2)--(T1);

\draw[thick,red]  (T1)--(L1);

\draw[thick,blue] (L1)--(L2);
\draw[thick,red]  (L2)--(L3);
\draw[thick,blue] (L3)--(L4);
\draw[thick,red]  (L4)--(L5);

\draw[thick,blue] (L5)--(B5);

\draw[thick,red]  (B5)--(B4);
\draw[thick,blue] (B4)--(B3);
\draw[thick,red]  (B3)--(B2);
\draw[thick,blue] (B2)--(B1);

\draw[thick,red]  (B1)--(R1);


\draw[thick,brown] (R5)--(T1);
\draw[thick,green] (R5)--(B1);
\draw[thick,brown] (L5)--(B1);
\draw[thick,green] (T1)--(L5);

\draw[thick,brown] (R4)--(T2);
\draw[thick,green] (R4)--(B2);
\draw[thick,brown] (L4)--(B2);
\draw[thick,green] (T2)--(L4);

\draw[thick,brown] (R3)--(T3);
\draw[thick,green] (R3)--(B3);
\draw[thick,brown] (L3)--(B3);
\draw[thick,green] (T3)--(L3);

\draw[thick,brown] (R2)--(T4);
\draw[thick,green] (R2)--(B4);
\draw[thick,brown] (L2)--(B4);
\draw[thick,green] (T4)--(L2);

\draw[thick,brown] (R1)--(T5);
\draw[thick,green] (R1)--(B5);
\draw[thick,brown] (L1)--(B5);
\draw[thick,green] (T5)--(L1);

\vertexLabelR{R1}{left}{$ $}
\vertexLabelR{R2}{left}{$ $}
\vertexLabelR{R3}{left}{$ $}
\vertexLabelR{R4}{left}{$ $}
\vertexLabelR{R5}{left}{$ $}
\vertexLabelR{B1}{left}{$ $}
\vertexLabelR{B2}{left}{$ $}
\vertexLabelR{B3}{left}{$ $}
\vertexLabelR{B4}{left}{$ $}
\vertexLabelR{B5}{left}{$ $}
\vertexLabelR{L1}{left}{$ $}
\vertexLabelR{L2}{left}{$ $}
\vertexLabelR{L3}{left}{$ $}
\vertexLabelR{L4}{left}{$ $}
\vertexLabelR{L5}{left}{$ $}
\vertexLabelR{T1}{left}{$ $}
\vertexLabelR{T2}{left}{$ $}
\vertexLabelR{T3}{left}{$ $}
\vertexLabelR{T4}{left}{$ $}
\vertexLabelR{T5}{left}{$ $}
\node[vertexBall,shift={(0.75,0)}] at (R1) {$z_{(3,g,5)}$};
\node[vertexBall,,shift={(0.75,0)}] at (R2) {$z_{(3,g,4)}$}; \node[vertexBall,shift={(0.75,0)}] at (R3) {$z_{(3,g,3)}$};
\node[vertexBall,shift={(0.75,0)}] at (R4) {$z_{(3,g,2)}$}; \node[vertexBall,shift={(0.75,0)}] at (R5) {$z_{(3,g,1)}$};

\node[vertexBall,shift={(0,0.5)}] at (T1) {$z_{(2,g,1)}$}; \node[vertexBall,shift={(0,0.5)}] at (T2) {$z_{(2,g,2)}$}; \node[vertexBall,shift={(0,0.5)}] at (T3) {$z_{(2,g,3)}$};
\node[vertexBall,shift={(0,0.5)}] at (T4) {$z_{(2,g,4)}$}; \node[vertexBall,shift={(0,0.55)}] at (T5) {$z_{(2,g,5)}$};

\node[vertexBall,shift={(-0.75,0)}] at (L1) {$z_{(1,g,5)}$}; \node[vertexBall,shift={(-0.75,0)}] at (L2) {$z_{(1,g,4)}$}; \node[vertexBall,shift={(-0.75,0)}] at (L3) {$z_{(1,g,3)}$};
\node[vertexBall,shift={(-0.75,0)}] at (L4) {$z_{(1,g,2)}$}; \node[vertexBall,shift={(-0.75,0)}] at (L5) {$z_{(1,g,1)}$};

\node[vertexBall,shift={(0,-0.5)}] at (B1) {$z_{(4,g,1)}$}; \node[vertexBall,shift={(0,-0.5)}] at (B2) {$z_{(4,g,2)}$}; \node[vertexBall,shift={(0,-0.5)}] at (B3) {$z_{(4,g,3)}$};
\node[vertexBall,shift={(0,-0.5)}] at (B4) {$z_{(4,g,4)}$}; \node[vertexBall,shift={(0,-0.5)}] at (B5) {$z_{(4,g,5)}$};

\end{tikzpicture}
    \caption{A connected component of the 4-regular graph $Z$. Note that for simplicity, we illustrate this graph having $4\cdot 5$ vertices. In our construction this graph contains $N_4(S)\cdot 5$ vertices.}
    \label{fig:Z}
\end{figure}
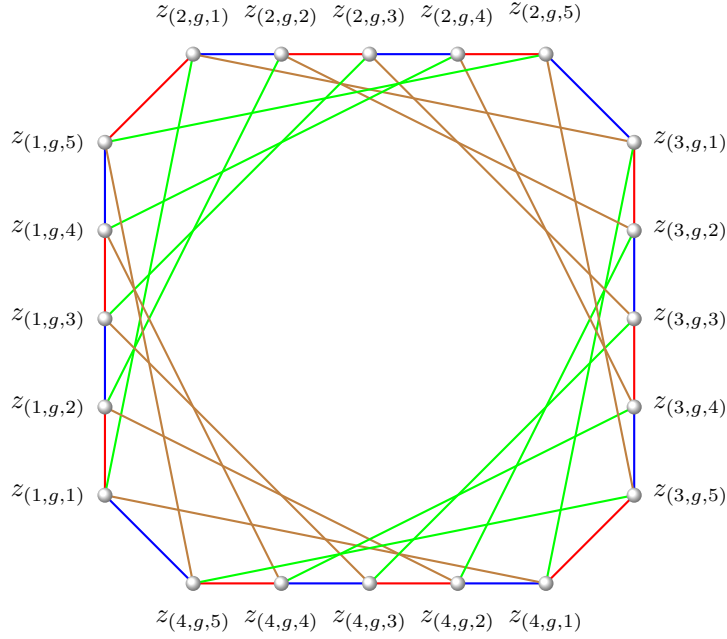
 Moreover, the partition $E(Z)=E_1\cup E_2\cup E_3\cup E_4$ yields a proper 4-edge-colouring $\kappa_Z$ of $Z$ by defining $\kappa_Z(e):=m$ for all $e\in E_{m}.$

The desired 5-regular graph $\Gamma^{(G,5)}$ is then constructed via vertices $V(\Gamma^{(G,5)})=V(\Gamma')\cup V(Z)$ and edges $E(\Gamma^{(G,5)})$ given by
$E(\Gamma')\cup E(Z)\cup E_5,$ where $E_5$ is the set of edges that corresponds to the quadratic form
\begin{align*}
\sum_{g\in G}\sum_{i=1}^{N_4(S)}\sum_{j=1}^5w_{(i,g,j)}z_{(i,g,j)}.
\end{align*}
Again, for each $h\in G$ we obtain an automorphism $\psi_h$ of $\Gamma^{(G,5)}$ via $\psi_h(w_{(i,g,j)}):=w_{(i,h\cdot g,j)}$. It remains to show that every automorphism $\psi\in \Aut(\Gamma^{(G,5)})$ satisfies $\psi=\psi_h$ for some $h\in G.$
This follows from the following fact: Since $\psi$ is an automorphism of $\Gamma^{(G,5)}$, it has to map a cycle of length $3$ onto another cycle of the same length. 
Note that every vertex $w_{(i,g,j)}\in V(\Gamma^{(G,5)})$ is contained in a cycle of length $3$ and the $z_{(i,g,j)}\in V(\Gamma^{(G,5)})$ are not contained in cycles of length 3. Hence, $V(\Gamma')$ and $V(Z)$ considered as subsets of $V(\Gamma^{(G,5)})$ satisfy $V(\Gamma')^{\Aut(\Gamma^{(G,5)})}= V(\Gamma')$ and $V(Z)^{\Aut(\Gamma^{(G,5)})}= V(Z)$. Even more, we have 
\begin{align*}
   & E(\Gamma')^{\Aut(\Gamma^{(G,5)})}= E(\Gamma'),\\
   &E(Z)^{\Aut(\Gamma^{(G,5)})}= E(Z), \\&
   E_5^{\Aut(\Gamma^{(G,5)})}=E_5.
\end{align*}
Thus, the restriction $\psi'$ of the automorphism $\psi$ to $V(\Gamma')\subseteq V(\Gamma^{(G,5)} )$ gives us an automorphism of $\Gamma'.$
Combining this statement with the fact that (1) $\Aut(\Gamma')\cong G $ and (2) the colouring $\kappa'$ is $\Aut(\Gamma')$-invariant leads to $\psi'(w_{(i,g,j)})= w_{(i,h\cdot g,j)}$ for some $h\in G$ and all $i=1,\ldots,N_4(S),\,g\in G$ and $j=1,\ldots,5$. Since incident vertices have to be mapped onto incident vertices, we obtain $\psi(z_{(i,g,j)})=z_{(i,h\cdot g, j)}$ for all $i=1,\ldots,N_4(S),\,g\in G$ and $j=1,\ldots,5$. Hence, $\psi =\psi_h$ and $\Aut(\Gamma^{(G,5)})$ is therefore isomorphic to $G$. 
Finally, we see that the colouring $\kappa_5:E(\Gamma^{(G,5)})\to \{1,\ldots,5\}$ defined by 
\[
\kappa_5(e')=\begin{cases}
\kappa'(e'), & \text{if } e'\in E(\Gamma')\\
\kappa_Z(e'), & \text{if } e'\in E(Z)\\
5, & \text{if } e'=\{w_{(i,g,j)},z_{(i,g,j)}\}\\
\end{cases}
\]
forms a 5-edge-colouring of $\Gamma^{(G,5)}$ that is $\Aut(\Gamma^{(G,5)})$-invariant. 
\end{proof}
\begin{remark}\label{rem:proof}
    Let $G$ be a finite group and $d=3,4,5$. In this section, we have established that there exists a $d$-regular graph $\Gamma$ satisfying $\Aut(\Gamma)\cong G$ that has an $\Aut(\Gamma)$-invariant $d$-edge-colouring. From our constructions in this section and the modification in \cite{SurfacesWithAuto}, we observe that $\Gamma$ can be chosen such that it contains at least one vertex that is not contained in a cycle of length $3$. 
\end{remark}
With the above proposition in place, we are now in a position to present our first main result.
\begin{theorem} \label{theorem:dregular}
    For every finite group $G$ and every $d\geq 3$ there exists a $d$-regular graph $\Gamma$  and a $d$-edge-colouring $\kappa$ of $\Gamma$ such that $\Aut(\Gamma)\cong G$ and $\kappa$ is $\Aut(\Gamma)$-invariant.
\end{theorem}

\begin{proof}
Since the case $\vert G \vert \leq 2$ is covered by Proposition~\ref{prop:atmost22}, we can assume $\vert G\vert >2$.
Let therefore $d' \in \{3,4,5\}$ be a natural number such that $d' \equiv d \pmod{3}$. Hence, there exists a natural number $\ell\in \mathbb{N}_0$ such that $d = d' + 3\ell$. By our results in Section~\ref{sec:regular}, there exists a $d'$-regular graph $\Gamma'$ such that $\mathrm{Aut}(\Gamma') \cong G$ and $\Gamma'$ admits an $\mathrm{Aut}(\Gamma')$-invariant $d'$-edge-colouring. 
Furthermore, by Proposition~\ref{prop:atmost22} there exists a cubic prime graph $\hat{\Gamma}$ such that $|\mathrm{Aut}(\hat{\Gamma})| = 1$ and $\hat{\Gamma}$ admits an $\mathrm{Aut}(\hat{\Gamma})$-invariant 3-edge-colouring.

Thus, we define $\Gamma_i := \mathcal{T}_i(\hat{\Gamma})$ for $1 \leq i \leq \ell$. By Proposition~\ref{prop:truncproper}, we know that for all $1 \leq i \leq \ell$, the graph $\Gamma_i$ admits an $\mathrm{Aut}(\Gamma_i)$-invariant 3-edge-colouring, and satisfies $|\mathrm{Aut}(\Gamma_i)| = 1$. Thus, the graph
\[
\Gamma = \Gamma' \times \prod_{i=1}^\ell \Gamma_i
\]
is a $d$-regular graph.
Since the graphs $\Gamma_1,\ldots,\Gamma_\ell$ are prime and additionally $\Gamma'$ contains vertices that do not lie on cycles of length $3$ (see Remark~\ref{rem:proof}), $\Gamma'$ cannot be written as a product $\Gamma'=Z\times \Gamma_i$ for a graph $Z$ and a $1\leq i\leq \ell$. Hence, $\Gamma',\Gamma_1,\ldots,\Gamma_\ell$ are relatively prime and so $\Gamma$ satisfies $\mathrm{Aut}(\Gamma) \cong G$ by \cite[Theorem 3.1]{graphmult}. Furthermore, by Lemma~\ref{lemma:edgecolouring}, this graph admits an $\mathrm{Aut}(\Gamma)$-invariant $d$-edge-colouring. This concludes the proof.
\end{proof}

This statement together with Lemma~\ref{lemma:colouring_CDC} leads to the following result.

\begin{theorem} \label{theorem:autom_strongEmbedding}
        For every finite group $G$ and every $d\geq 3$ there exists a $d$-regular graph $\Gamma$ and a corresponding strong embedding $\beta$ of $\Gamma$ such that $$\Aut(\Gamma)\cong \Aut(\beta(\Gamma))\cong G.$$
\end{theorem}
\begin{proof}
Let $d\geq 3$ be an arbitrary natural number.
By Theorem~\ref{theorem:dregular}, we know that there exists a $d$-regular graph $\Gamma$ and a corresponding $\Aut(\Gamma)$-invariant edge-colouring $\kappa$ of $\Gamma$ such that $\Aut(\Gamma)\cong G.$ With Lemma~\ref{lemma:colouring_CDC}, we can use the edge-colouring $\kappa$ to construct a facial $\Aut(\Gamma)$-invariant cycle double cover $\mathcal{C}$ of $\Gamma$. This cycle double cover $\mathcal{C}$ corresponds to a strong embedding $\beta$ of $\Gamma$ satisfying $\Aut(\Gamma)\cong \Aut(\beta(\Gamma))\cong G.$ This concludes the proof.
\end{proof}

\section{Regular graphs with prescribed automorphism group of unbounded strong genus } \label{sec:genus}

The aim of this section is to show that for every $d\geq 3$ there exists a sequence of $d$-regular graphs with a prescribed automorphism group that admit strong embeddings with the same automorphism group and whose genera tend to infinity.
Analogously to the previous proofs, we begin with the cases $d=3,4,5$.

\begin{proposition}\label{prop:infinite_3}
    For every finite group $G$ there exists a sequence $(\Gamma_\ell)_{\ell \in \mathbb{N}}$ of cubic prime graphs with corresponding strong embeddings $(\beta_\ell)_{\ell\in \mathbb{N}}$ such that the following conditions are satisfied:
    \begin{enumerate}
        \item $\Aut(\Gamma_\ell)\cong\Aut(\beta_\ell(\Gamma_\ell))\cong G$ for all $\ell\in \mathbb{N}$ and
        \item $\lim_{\ell\to \infty} g(\beta_\ell(\Gamma_\ell)) =\infty$.
    \end{enumerate}
\end{proposition}
\begin{proof}
    Let $G$ be an arbitrary finite group. 
    By \Cref{prop:cyclicauto,prop:fruchtresult,prop:atmost2}, there exists a cubic prime graph $\Gamma$ with an $\Aut(\Gamma)$-invariant 3-colouring $\kappa:E(\Gamma)\rightarrow \{1,2,3\}$ and $\Aut(\Gamma)\cong G$.  According to Remark~\ref{rem:colouring}, the cycle $(1,2,3)\in S_3$ induces a facial $\Aut(\Gamma)$-invariant cycle double cover $\mathcal{C}$ of $\Gamma$ that is constructed by exploiting the  3-edge-colouring $\kappa$, see also Lemma~\ref{lemma:colouring_CDC}. This facial cycle double cover $\mathcal{C}$ defines a strong embedding $\beta$ of $\Gamma$.
    Our desired result follows, if we can show that $\Gamma$ can be exploited to construct another cubic prime graph $\Gamma'$ together with a strong embedding $\beta'$ satisfying $\Aut(\Gamma')\cong \Aut(\beta'(\Gamma'))\cong G$ and $\chi(\beta (\Gamma))>\chi(\beta' (\Gamma')).$ In particular, if we are able to achieve this construction, we can recursively build a sequence of cubic graphs with corresponding strong embeddings that satisfy the properties of the above statement. Here, we obtain this construction by defining $\Gamma'$ as $\Gamma':=\mathcal{T}(\Gamma,\beta)$, the complete truncation of $\Gamma$.
    
    By Proposition~\ref{prop:cubicprime} the graph $\Gamma'$ is prime and $\Aut(\Gamma')\cong\Aut(\Gamma)\cong G$ is satisfied by Corollary~\ref{prop:autiso2}.
    In order to conclude the result, we have to construct an $\Aut(\Gamma')$-invariant strong embedding $\beta'$ of $\Gamma'$ with a smaller Euler characteristic than $\beta(\Gamma)$. For this, the $\Aut(\Gamma)$-invariant 3-colouring $\kappa$ of $\Gamma$ is translated to a $\Aut(\Gamma')$-invariant 3-colouring $\kappa'$ of $\Gamma'$ as described in Proposition~\ref{prop:truncproper}.
    Again, we use the colouring $\kappa$ to label the vertices of $\Gamma':$ For each vertex $v_{i}\in V(\Gamma)$, we obtain three vertices $v_{(i,j)},$ where $j=1,2,3,$ in the complete truncation $\Gamma'$. Two vertices  $v_{(i_1,j_1)}$ and $v_{(i_2,j_2)}$ in $V(\Gamma')$ are connected by an edge, if (1) $i_1=i_2$ and $j_1\neq j_2$, or (2) $e:=\{v_{i_1},v_{i_2}\}\in E(\Gamma)$ and $\kappa(e)=j_1=j_2$.
    Consider a bi-coloured cycle $\gamma=(v_1,v_2,\dots,v_m)$ with colours $a,b\in\{1,2,3\}$ and even length in $\Gamma$. Suppose that the edge $\{v_1,v_2\}\in E(\Gamma)$ is coloured in $a$ and let $c$ be the unique integer in $\{1,2,3\}\setminus\{a,b\}$.
    Then the cycle $\gamma$ translates to the bi-coloured cycle $$(v_{(1,b)},v_{(1,c)},v_{(1,a)},v_{(2,a)},v_{(2,c)},v_{(2,b)},v_{(3,b)},\dots,v_{(m,b)})$$
    with colours $a$ and $b$ in $\Gamma'$, as illustrated in \Cref{fig:3reg_limit}.
    Translating all bi-coloured cycles of $\mathcal{C}$ in $\Gamma$ analogously to bi-coloured cycles in $\Gamma'$ results in a facial cycle double cover $\mathcal{C}'$ of $\Gamma'$ obtained by the $\Aut(\Gamma')$-invariant 3-edge-colouring $\kappa'.$ Hence, $\kappa'$ induces an $\Aut(\Gamma')$-invariant strong embedding $\beta'$ of $\Gamma'$.

    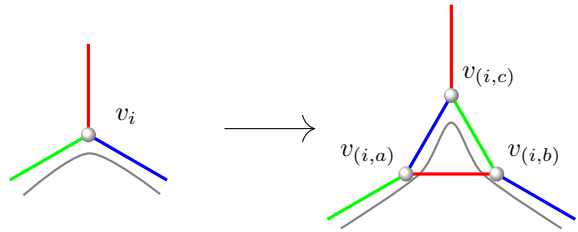
\begin{figure}[H]
        \centering
        \begin{tikzpicture}[vertexBall, scale=1.2]

\coordinate (V7) at (0.5, 0.433013);
\coordinate (V8) at (1/2, 1.433013);
\coordinate (V9) at (1.36603, -0.0669873);
\coordinate (V10) at (-0.36603, -0.0669873);

\draw[thick, gray, smooth ]
plot coordinates {
    (-0.86603+4.15 , -0.6)
    (4.15,0)
    (4.5,0.866025-0.3)
    (4.85,0)
    (1.86603 +3.85, -0.6)
  };      

\draw[-{To[scale=2]}] (2,0.5) to (3,0.5);

\draw[-,red,very thick] (V7) to (V8);
\draw[-,blue,very thick] (V7) to (V9);
\draw[-,green,very thick] (V7) to (V10);

\vertexLabelR[]{V7}{left}{$ $}

\node[vertexBall,shift={(0.5,0.25)}] at (V7) {$v_{i}$};

\begin{scope}[shift={(0.4,1.1)},rotate=240]
    \draw[thick,gray,  smooth ]
    plot coordinates {
        (1.36603 + 0.1, -0.0669873 + 0.2)
        (0.5 + 0.2, 0.433013 + 0.1)
        (1/2 + 0.2, 1.433013)
      };      
\end{scope}

\coordinate (V1) at (0+4, 0);
\coordinate (V2) at (1+4 , 0);
\coordinate (V3) at (0.5+4 , 0.866025);
\coordinate (V4) at (1/2+4 , 1.866025);
\coordinate (V5) at (1.86603 +4, -0.5);
\coordinate (V6) at (-0.86603+4 , -0.5);

\draw[-,blue,very thick] (V1) to (V3);
\draw[-,green,very thick] (V2) to (V3);
\draw[-,red,very thick] (V1) to (V2);
\draw[-,green,very thick] (V1) to (V6);
\draw[-,red,very thick] (V3) to (V4);
\draw[-,blue,very thick] (V2) to (V5);

\vertexLabelR[]{V1}{left}{$ $}
\vertexLabelR[]{V2}{left}{$ $}
\vertexLabelR[]{V3}{left}{$ $}

\node[vertexBall,shift={(-0.5,0.25)}] at (V1) {$v_{(i,a)}$};

\node[vertexBall,shift={(0.5,0.25)}] at (V2) {$v_{(i,b)}$};

\node[vertexBall,shift={(0.5,0.25)}] at (V3) {$v_{(i,c)}$};

\end{tikzpicture}
        \caption{Translating a bi-coloured cycle of $\Gamma$ to a bi-coloured cycle of $\mathcal{T}(\Gamma)$.}
        \label{fig:3reg_limit}
    \end{figure}
    
    It remains to show that the genus of $\beta'(\Gamma')$ is greater than the genus of $\beta(\Gamma)$. The Euler characteristic of $\beta'(\Gamma')$ can be computed based on the Euler characteristic of $\beta(\Gamma)$ as follows
    $$\chi(\beta'(\Gamma'))=3\vert V(\Gamma)\vert - (\vert E(\Gamma)\vert + 3\vert V(\Gamma)\vert) + \vert F(\beta(\Gamma))\vert =\chi(\beta(\Gamma))-\vert V(\Gamma)\vert<\chi(\beta(\Gamma)).$$
    Thus, we obtain in total $\chi(\beta (\Gamma))>\chi(\beta' (\Gamma'))$ and the result follows.
\end{proof}
Now, we turn to the case $d=4,5$. For this, we introduce the following remark: 
\begin{remark}\label{rem:rotedges}
Let $\Gamma$ be a $d$-regular graph and $\beta$ a strong embedding of $\Gamma.$  
In the following, we assume that the vertices of $\Gamma$ are given by the set $V(\Gamma)=\{1,\ldots,k\},$ where $k:=\vert V(\Gamma)\vert .$ 
We fix a notation for every edge $e=\{i,j\}\in E(\Gamma)$ as follows:
We know that for every vertex $i\in V(\Gamma)$ the strong embedding $\beta$ induces a local orientation $(e^i_1,\ldots,e^i_d)$ of the edges that are incident to $i$. Hence, for each vertex $i\in V(\Gamma)$, we choose a fixed orientation and consider this orientation as an ordered list, i.e., $[e^i_1,\ldots,e^i_d]$. Hence, for the edge $e$ that appears at the $m$-th position in the local rotation at vertex $i$, we introduce the notation $e=e[i,m]$. With this notation, we know that considering an edge $e[i,m]=\{i,j\}$ of a vertex $i$, as an edge that is contained in the local rotation at vertex $j$ gives us $e=e[i,m]=e[j,r]$ for a suitable $r\in \{1,\ldots,d\}$.
\end{remark}

\begin{proposition}\label{prop:infinite_45}
    For every finite group $G$ and $d=4,5$ there exists a sequence $(\Gamma_\ell)_{\ell \in \mathbb{N}}$ of $d$-regular graphs with corresponding strong embeddings $(\beta_\ell)_{\ell\in \mathbb{N}}$ such that the following conditions are satisfied:
    \begin{enumerate}
        \item $\Aut(\Gamma_\ell)\cong\Aut(\beta_\ell(\Gamma_\ell))\cong G$ for all $\ell\in \mathbb{N}$ and
        \item 
        $\lim_{\ell\to \infty} g(\beta_\ell(\Gamma_\ell)) =\infty$.
    \end{enumerate}
\end{proposition}
\begin{proof}

Let $G$ be an arbitrary finite group. 
By Theorem~\ref{theorem:autom_strongEmbedding}, there exists a $d$-regular graph $\Gamma$ and a strong embedding $\beta$ of $\Gamma$ such that $\Aut(\Gamma)\cong\Aut(\beta(\Gamma))\cong G.$ We assume that the vertices of $\Gamma$ are given by the set $V(\Gamma)=\{1,\ldots,k\},$ where $k:=\vert V(\Gamma)\vert .$ 
Our desired result follows, if we can show that $\Gamma$ can be exploited to construct another $d$-regular graph $\Gamma'$ together with a strong embedding $\beta'$ satisfying $\Aut(\Gamma')\cong \Aut(\beta'(\Gamma'))\cong G$ and $\chi(\beta (\Gamma))>\chi(\beta' (\Gamma')).$
Again, if we obtain the above construction, we can recursively build a sequence of $d$-regular graphs with corresponding strong embeddings that satisfy the properties of the stated proposition.
We achieve this construction by modifying $\Gamma$ in three steps and thus obtaining two auxiliary graphs $\Gamma_1$ and $\Gamma_2$ along the way. 

 We first compute the complete truncation $\Gamma_1=\mathcal{T}(\Gamma,\beta)$ of $\Gamma.$ Here, we choose the vertices of $\Gamma_1$ to be of the form $x_{(i,e)}$ where $i=1,\ldots,k$ and $e\in E(\Gamma)$ is an edge with $i\in e,$ see Figure~\ref{five}.
 
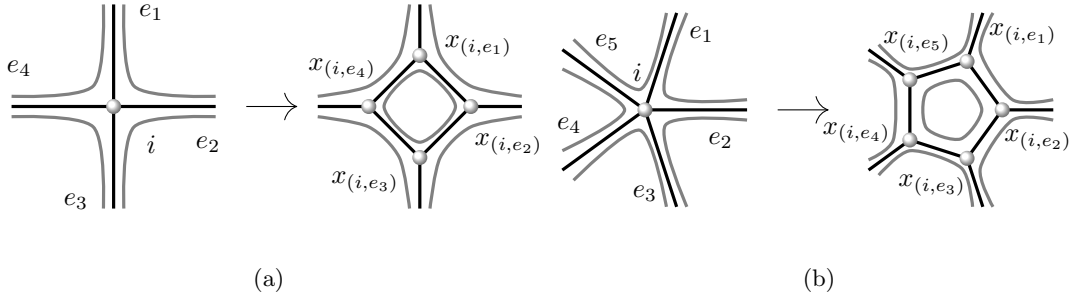
\begin{figure}[H]
    \centering
    \begin{subfigure}{.48\textwidth}
        \begin{tikzpicture}[vertexBall, edgeDouble=nolabels, scale=1.35]

\coordinate (V1) at (0 , 0);
\coordinate (V2) at (1 , 0);
\coordinate (V3) at (0 , 1);
\coordinate (V4) at (-1 , 0);
\coordinate (V5) at (0, -1);

\draw[-,very thick] (V1) to (V2);
\draw[-,very thick] (V1) to (V3);
\draw[-,very thick] (V1) to (V4);
\draw[-,very thick] (V1) to (V5);

\vertexLabelR[]{V1}{left}{$ $}

\draw[-{To[scale=2]}] (1.3,0) to (1.8,0);

\coordinate (V12) at (0.5+3 , 0);
\coordinate (V22) at (0 +3, 0.5);
\coordinate (V32) at (-0.5+3 , 0);
\coordinate (V42) at (3, -0.5);

\coordinate (V13) at (1 +3, 0);
\coordinate (V23) at (3, 1);
\coordinate (V33) at (-1 +3, 0);
\coordinate (V43) at (3, -1);

\draw[-,very thick] (V12) to (V13);
\draw[-,very thick] (V22) to (V23);
\draw[-,very thick] (V32) to (V33);
\draw[-,very thick] (V42) to (V43);

\draw[-,very thick] (V12) to (V22);
\draw[-,very thick] (V22) to (V32);
\draw[-,very thick] (V32) to (V42);
\draw[-,very thick] (V42) to (V12);

 \draw[very thick, gray, smooth ]
  plot coordinates {
    (0.1,1)
    (0.2,0.2)
    (1,0.1)
  };

\draw[very thick, gray, smooth ]
  plot coordinates {
    (-0.1,1)
    (-0.2,0.2)
    (-1,0.1)
  };

  \draw[very thick, gray, smooth ]
  plot coordinates {
    (-0.1,-1)
    (-0.2,-0.2)
    (-1,-0.1)
  };

  \draw[very thick, gray, smooth ]
  plot coordinates {
    (0.1,-1)
    (0.2,-0.2)
    (1,-0.1)
  };

\vertexLabelR[]{V12}{left}{$ $}
\vertexLabelR[]{V22}{left}{$ $}
\vertexLabelR[]{V32}{left}{$ $}
\vertexLabelR[]{V42}{left}{$ $}

\draw[very thick, gray, smooth cycle ]
  plot coordinates {
    (0.35+3,0)
    (3,-0.35)
    (-0.35+3,0)
    (3,0.35)
  };

 \draw[very thick, gray, smooth ]
  plot coordinates {
    (0.1+3,1)
    (0.2+3,0.5)
        (0.5+3,0.1+0.1)
    (1+3,0.1)
  };

\draw[very thick, gray, smooth ]
  plot coordinates {
    (-0.1+3,1)
    (-0.2+3,0.5)
    (-0.5+3,0.1+0.1)
    (-1+3,0.1)
  };

  \draw[very thick, gray, smooth ]
  plot coordinates {
    (-0.1+3,-1)
        (-0.2+3,-0.5)
    (-0.5+3,-0.2)
    (-1+3,-0.1)
  };

    \draw[very thick, gray, smooth ]
  plot coordinates {
    (0.1+3,-1)
    (0.2+3,-0.5)
    (0.5+3,-0.2)
    (1+3,-0.1)
  };

  \node[vertexBall,shift={(0.5,-0.5)}] at (V1) {$i$};

    \node[vertexBall,shift={(0.5,1.25)}] at (V1) {$e_1$};

\node[vertexBall,shift={(-0.5,-1.25)}] at (V1) {$e_3$};

\node[vertexBall,shift={(-1.25,0.5)}] at (V1) {$e_4$};

\node[vertexBall,shift={(1.25,-0.5)}] at (V1) {$e_2$};

\node[vertexBall,shift={(-1.4,-1)}] at (V12) {$x_{(i,e_3)}$};

\node[vertexBall,shift={(.1,.8)}] at (V12) {$x_{(i,e_1)}$};

\node[vertexBall,shift={(-1.7,.5)}] at (V12) {$x_{(i,e_4)}$};
\node[vertexBall,shift={(.5,-.5)}] at (V12) {$x_{(i,e_2)}$};
\end{tikzpicture}
        \subcaption{}
    \end{subfigure}
    \begin{subfigure}{.48\textwidth}
        \raisebox{15pt}{\begin{tikzpicture}[vertexBall, edgeDouble=nolabels, faceStyle=nolabels, scale=1.35]

\coordinate (V1) at (0 , 0);
\coordinate (V2) at (1 , 0);
\coordinate (V3) at (0.31, 0.95);
\coordinate (V4) at (-0.81, 0.59);
\coordinate (V5) at (-0.81, -0.59);
\coordinate (V6) at (0.31, -0.95);

\draw[-,very thick] (V1) to (V2);
\draw[-,very thick] (V1) to (V3);
\draw[-,very thick] (V1) to (V4);
\draw[-,very thick] (V1) to (V5);
\draw[-,very thick] (V1) to (V6);

\vertexLabelR[]{V1}{left}{$ $}

\draw[-{To[scale=2]}] (1.3,0) to (1.8,0);

\coordinate (V12) at (0.5*1 +3, 0);
\coordinate (V22) at (0.5*0.31+3, 0.5*0.95);
\coordinate (V32) at (0.5*-0.81+3, 0.5*0.59);
\coordinate (V42) at (0.5*-0.81+3, 0.5*-0.59);
\coordinate (V52) at (0.5*0.31+3, 0.5*-0.95);

\coordinate (V13) at (1 +3, 0);
\coordinate (V23) at (0.31+3, 0.95);
\coordinate (V33) at (-0.81+3, 0.59);
\coordinate (V43) at (-0.81+3, -0.59);
\coordinate (V53) at (0.31+3, -0.95);

\draw[-,very thick] (V12) to (V13);
\draw[-,very thick] (V22) to (V23);
\draw[-,very thick] (V32) to (V33);
\draw[-,very thick] (V42) to (V43);
\draw[-,very thick] (V52) to (V53);

\draw[-,very thick] (V12) to (V22);
\draw[-,very thick] (V22) to (V32);
\draw[-,very thick] (V32) to (V42);
\draw[-,very thick] (V42) to (V52);
\draw[-,very thick] (V52) to (V12);

 \draw[very thick, gray, smooth ]
  plot coordinates {
 (0.31+0.1, 0.95)
    (0.2,0.15)
    (1,0.1)
  };

   \draw[very thick, gray, smooth ]
  plot coordinates {
 (0.31+0.1, -0.95)
    (0.2,-0.15)
    (1,-0.1)
  };

     \draw[very thick, gray, smooth ]
  plot coordinates {
 (-0.81, -0.4)
    (-0.2,0)
    (-0.81, 0.4)
  };

       \draw[very thick, gray, smooth ]
  plot coordinates {
 (-0.7, 0.69)
    (-0.05,0.2)
 (0.31-0.1, 0.95)
  };

 \draw[very thick, gray, smooth ]
  plot coordinates {
 (-0.7, -0.69)
    (-0.05,-0.2)
 (0.31-0.1, -0.95)
  };

\vertexLabelR[]{V12}{left}{$ $}
\vertexLabelR[]{V22}{left}{$ $}
\vertexLabelR[]{V32}{left}{$ $}
\vertexLabelR[]{V42}{left}{$ $}
\vertexLabelR[]{V52}{left}{$ $}

    \draw[very thick, gray, smooth cycle ]
  plot coordinates {
(0.3*1 +3, 0)
(0.3*0.31+3, 0.3*0.95)
 (0.3*-0.81+3, 0.3*0.59)
(0.3*-0.81+3, 0.3*-0.59)
 (0.3*0.31+3, 0.3*-0.95)
  };

 \draw[very thick, gray , smooth ]
  plot coordinates {
      (1 +3, 0.1)
      (0.6*1 +3, 0.1)
(0.6*0.31+3+.1, 0.6*0.95)
(0.31+3+0.1, 0.95)
  };

   \draw[very thick,gray, smooth ]
  plot coordinates {
      (1 +3, -0.1)
      (0.6*1 +3, -0.1)
(0.6*0.31+3+.1, -0.6*0.95)
(0.31+3+0.1, -0.95)
  };

   \draw[very thick, gray, smooth ]
  plot coordinates {
 (-0.81+3, 0.49)
 (-0.81*0.7+3, 0.49*0.6) 
(-0.81*0.7+3, -0.49*0.6)
(-0.81+3, -0.49)
  };

     \draw[very thick, gray, smooth ]
  plot coordinates {
(0.31+3-0.1, 0.95)
(0.1+3, 0.6*0.95)
 (-0.4+3, 0.4)
 (-0.81+3+0.1, 0.49+0.15)
  };

     \draw[very thick, gray, smooth ]
  plot coordinates {
(0.31+3-0.1, -0.95)
(0.1+3, -0.6*0.95)
 (-0.4+3, -0.4)
 (-0.81+3+0.1, -0.49-0.15)
  };

  \node[vertexBall,shift={(-.1,.5)}] at (V1) {$i$};

    \node[vertexBall,shift={(0.75,1)}] at (V1) {$e_1$};

    \node[vertexBall,shift={(1,-.4)}] at (V1) {$e_2$};

        \node[vertexBall,shift={(0.,-1.1)}] at (V1) {$e_3$};
    \node[vertexBall,shift={(-1,-.2)}] at (V1) {$e_4$};
        \node[vertexBall,shift={(-0.5,.9)}] at (V1) {$e_5$};

    \node[vertexBall,shift={(0.75+4.25,1)}] at (V1) {$x_{(i,e_1)}$};

    \node[vertexBall,shift={(1+4.2,-.4)}] at (V1) {$x_{(i,e_2)}$};

        \node[vertexBall,shift={(0.+3.8,-1.)}] at (V1) {$x_{(i,e_3)}$};
    \node[vertexBall,shift={(-1+3.8,-.3)}] at (V1) {$x_{(i,e_4)}$};
        \node[vertexBall,shift={(-0.5+4.1,.9)}] at (V1) {$x_{(i,e_5)}$};

\end{tikzpicture}}
        \subcaption{}
    \end{subfigure}
    \caption{Truncating a vertex of degree 4 (a) and a vertex of degree 5 (b) in $\Gamma$}
    \label{five}
\end{figure}

By Proposition~\ref{prop:embedding_trunc}, $\Gamma_1$ has a strong embedding $\beta_1.$
If we follow the proof of the mentioned proposition and choose $\beta_1$ to be the embedding that arises from $\beta$ as indicated in the above figure, we know that $\chi(\beta (\Gamma))=\chi(\beta_1 (\Gamma_1))$. In this case, Proposition~\ref{prop:autiso} implies $\Aut(\Gamma_1)\cong \Aut(\beta(\Gamma))\cong G.$
Now, we further modify this graph $\Gamma_1$ to obtain another cubic graph $\Gamma_2.$ The idea behind this second modification is that every edge $\{x_{(i,e)},x_{(j,e)}\}\in E(\Gamma_1),$ where $e=\{i,j\}\in E(\Gamma),$ is split into a cycle of length $4$ in order to construct $\Gamma_2$. This is achieved as follows:

We define the vertices of $\Gamma_2$ to be given by $x_{(i,e,t)},$ where $i=1,\ldots,k,$  $i\in e\in E(\Gamma)$ and $t=1,2$. Furthermore, we construct the edges of $\Gamma_2$ by the following quadratic forms $Q_1$ and $Q_2$.
For the definition of $Q_1$, we define the cycle $g=(1,\ldots,d)\in S_d$ and use this permutation to translate the cycles of length $d$ in $\Gamma_1$ that correspond to vertices in $\Gamma$ into cycles of length $2\cdot d$ in the desired graph $\Gamma_2$. This is realised by 
\[
Q_1:=\sum_{i=1}^k \sum_{m=1}^{d}x_{(i,e[i,m],1)}x_{(i,e[i,m],2)}+x_{(i,e[i,m],2)}x_{(i,e[i,m^g],1)}
\]
and shown in \Cref{fig:circle}.
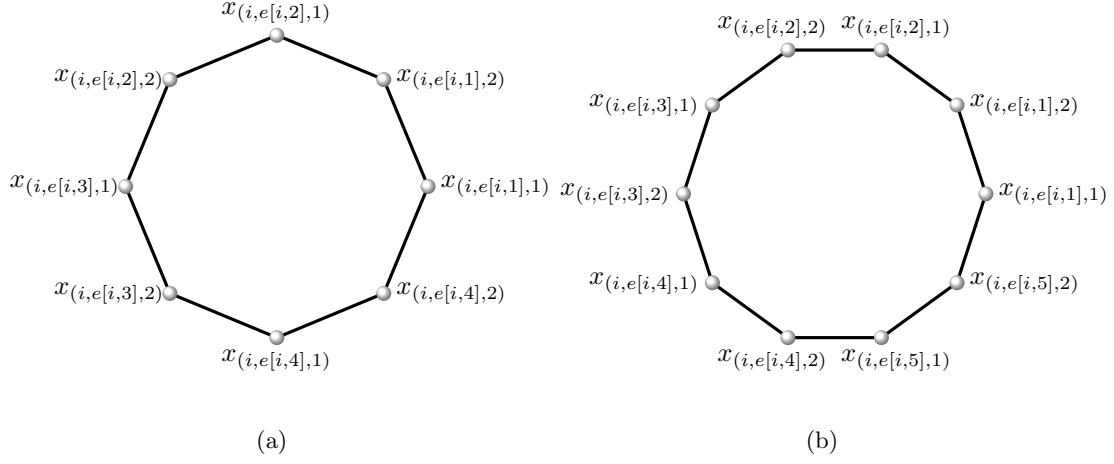
\begin{figure}[H]
    \centering
    \begin{subfigure}{.48\textwidth}
    \centering
    \begin{tikzpicture}[vertexBall, edgeDouble=nolabels, faceStyle=nolabels, scale=1]

\coordinate (V1) at (0:2);
\coordinate (V2) at (45:2);
\coordinate (V3) at (90:2);
\coordinate (V4) at (135:2);
\coordinate (V5) at (180:2);
\coordinate (V6) at (225:2);
\coordinate (V7) at (270:2);
\coordinate (V8) at (315:2);

\draw[-,very thick] (V1) -- (V2);
\draw[-,very thick] (V2) -- (V3);
\draw[-,very thick] (V3) -- (V4);
\draw[-,very thick] (V4) -- (V5);
\draw[-,very thick] (V5) -- (V6);
\draw[-,very thick] (V6) -- (V7);
\draw[-,very thick] (V7) -- (V8);
\draw[-,very thick] (V8) -- (V1);

\vertexLabelR[]{V1}{left}{$ $}
\vertexLabelR[]{V2}{left}{$ $}
\vertexLabelR[]{V3}{left}{$ $}
\vertexLabelR[]{V4}{left}{$ $}
\vertexLabelR[]{V5}{left}{$ $}
\vertexLabelR[]{V6}{left}{$ $}
\vertexLabelR[]{V7}{left}{$ $}
\vertexLabelR[]{V8}{left}{$ $}

\node[vertexBall,shift={(0.9,.)}] at (V1) {$x_{(i,e[i,1],1)}$};

\node[vertexBall,shift={(0.9,.)}] at (V2) {$x_{(i,e[i,1],2)}$};

\node[vertexBall,shift={(-0.,.3)}] at (V3) {$x_{(i,e[i,2],1)}$};

\node[vertexBall,shift={(-0.8,.)}] at (V4) {$x_{(i,e[i,2],2)}$};

\node[vertexBall,shift={(-0.8,.)}] at (V5) {$x_{(i,e[i,3],1)}$};

\node[vertexBall,shift={(-0.8,.)}] at (V6) {$x_{(i,e[i,3],2)}$};

\node[vertexBall,shift={(-0.,-.3)}] at (V7) {$x_{(i,e[i,4],1)}$};

\node[vertexBall,shift={(0.9,.)}] at (V8) {$x_{(i,e[i,4],2)}$};
\end{tikzpicture}
    \subcaption{}
    \end{subfigure}
    \begin{subfigure}{.48\textwidth}
    \centering
    \begin{tikzpicture}[vertexBall, edgeDouble=nolabels, faceStyle=nolabels, scale=1]

\coordinate (V1)  at (0:2);
\coordinate (V2)  at (36:2);
\coordinate (V3)  at (72:2);
\coordinate (V4)  at (108:2);
\coordinate (V5)  at (144:2);
\coordinate (V6)  at (180:2);
\coordinate (V7)  at (216:2);
\coordinate (V8)  at (252:2);
\coordinate (V9)  at (288:2);
\coordinate (V10) at (324:2);

\draw[-,very thick] (V1) -- (V2);
\draw[-,very thick] (V2) -- (V3);
\draw[-,very thick] (V3) -- (V4);
\draw[-,very thick] (V4) -- (V5);
\draw[-,very thick] (V5) -- (V6);
\draw[-,very thick] (V6) -- (V7);
\draw[-,very thick] (V7) -- (V8);
\draw[-,very thick] (V8) -- (V9);
\draw[-,very thick] (V9) -- (V10);
\draw[-,very thick] (V10) -- (V1);

\vertexLabelR[]{V1}{left}{$ $}
\vertexLabelR[]{V2}{left}{$ $}
\vertexLabelR[]{V3}{left}{$ $}
\vertexLabelR[]{V4}{left}{$ $}
\vertexLabelR[]{V5}{left}{$ $}
\vertexLabelR[]{V6}{left}{$ $}
\vertexLabelR[]{V7}{left}{$ $}
\vertexLabelR[]{V8}{left}{$ $}
\vertexLabelR[]{V9}{left}{$ $}
\vertexLabelR[]{V10}{left}{$ $}

\node[vertexBall,shift={(0.9,.)}] at (V1) {$x_{(i,e[i,1],1)}$};

\node[vertexBall,shift={(0.9,.)}] at (V2) {$x_{(i,e[i,1],2)}$};

\node[vertexBall,shift={(.2,.3)}] at (V3) {$x_{(i,e[i,2],1)}$};

\node[vertexBall,shift={(-.2,.3)}] at (V4) {$x_{(i,e[i,2],2)}$};

\node[vertexBall,shift={(-0.9,.)}] at (V5) {$x_{(i,e[i,3],1)}$};

\node[vertexBall,shift={(-0.9,.)}] at (V6) {$x_{(i,e[i,3],2)}$};

\node[vertexBall,shift={(-0.9,.)}] at (V7) {$x_{(i,e[i,4],1)}$};

\node[vertexBall,shift={(-.2,-.3)}] at (V8) {$x_{(i,e[i,4],2)}$};

\node[vertexBall,shift={(.2,-.3)}] at (V9) {$x_{(i,e[i,5],1)}$};

\node[vertexBall,shift={(0.9,.)}] at (V10) {$x_{(i,e[i,5],2)}$};
\end{tikzpicture}
    \subcaption{}
    \end{subfigure}
    \caption{Cycles of length $2\cdot d$ in $\Gamma_2$ that correspond to cycles of length $d$ in $\Gamma_1$ for $d=4$ (a) and $d=5$ (b)}
    \label{fig:circle}
\end{figure}

Next, we construct the quadratic form $Q_2$ which helps us to obtain cycles of length 4 that contain the four vertices $x_{(i,e,1)},x_{(i,e,2)},x_{(j,e,1)},x_{(j,e,2)}$ for every edge $e=\{i,j\}\in E(\Gamma)$. Note that in order to get a graph with the desired automorphism group and a desired strong embedding, we have to be careful in this modification step and have to take the strong embedding $\beta_1$ of $\Gamma_1$ into consideration. Thus, let $\mathcal{C}_1$ be the facial cycle double cover of $\Gamma_1$ that corresponds to the strong embedding $\beta_1$. 
Furthermore, for every edge $e=\{i,j\}\in E(\Gamma)$, let $\gamma_1,\gamma_2 \in \mathcal{C}_1$ be the cycles that contain the edge $\{x_{(i,e)},x_{(j,e)}\}$ in $E(\Gamma_1)$. Recall from Remark~\ref{rem:rotedges} that there exist suitable $m, r\in\{1,\dots,d\}$ such that $e=e[i,m]=e[j,r]$. Since $\gamma_1$ and $\gamma_2$ contain the above edge, we know that there exists an $h\in\{g,g^{-1}\}$ such that $\gamma_1$ contains the subsequence $$(x_{(i,e[i,m^{g^{-1}}])},x_{(i,e[i,m])}, x_{(j,e[j,r])},x_{(j,e[j,r^{h^{(-1)}}])}  )$$ and $\gamma_2$ contains the subsequence $$(x_{(i,e[i,m^{g}])},x_{(i,e[i,m])},x_{(j,e[j,{r}])},x_{(j,e[j,{r^{h}])}}  ).$$ In the following, we assume $h=g$ for simplicity. The reasoning for $h=g^{-1}$ follows by giving similar arguments. Hence, we obtain the remaining edges of $\Gamma_2$ via
\[
Q_2:=  \sum_{\{i,j\} \in E(\Gamma)}\sum_{t=1}^2 x_{(i,e[i,m],t)}x_{(j,e[j,r],t)}
\]

Hence, for each edge $e=\{i,j\}\in E(\Gamma)$ the sequence $(x_{(i,e,1)},x_{(i,e,2)},x_{(j,e,2)},x_{(j,e,1)})$ forms a cycle of length $4$ in $\Gamma_2$, see Figure~\ref{fig:edgecut5}.

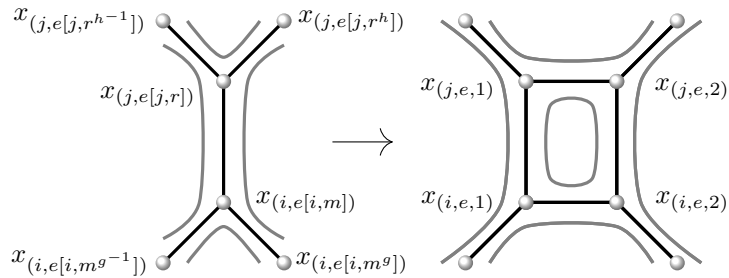
\begin{figure}[H]
    \centering
    \begin{tikzpicture}[vertexBall, edgeDouble=nolabels, faceStyle=nolabels, scale=1.6]

\coordinate (V1) at (0+4-1 , 0);
\coordinate (V2) at (0+4-1 , 1);
\coordinate (V3) at (0.75+4-1 , 1);
\coordinate (V4) at (0.75+4-1 , 0);

\coordinate (V5) at (-0.5+4-1 , -0.5);
\coordinate (V6) at (0.5+0.75+4-1 , -0.5);

\coordinate (V7) at (-0.5+4-1 , 1.5);
\coordinate (V8) at (0.5+0.75+4-1 , 1.5);

\coordinate (V9) at (0.375+4-1 , 0.5);

\draw[-,very thick] (V1) to (V2);
\draw[-,very thick] (V2) to (V3);
\draw[-,very thick] (V3) to (V4);
\draw[-,very thick] (V1) to (V4);

\draw[-,very thick] (V1) to (V5);
\draw[-,very thick] (V4) to (V6);
\draw[-,very thick] (V2) to (V7);
\draw[-,very thick] (V3) to (V8);

\vertexLabelR[]{V1}{left}{$ $}
\vertexLabelR[]{V2}{left}{$ $}
\vertexLabelR[]{V3}{left}{$ $}
\vertexLabelR[]{V4}{left}{$ $}

\coordinate (V7) at (0.5 , 0);
\coordinate (V8) at (1/2 , 1);
\coordinate (V9) at (1. , -0.5);
\coordinate (V10) at (0 , -0.5);
\coordinate (V11) at (1 , 1.5 );
\coordinate (V12) at (0 , 1.5);

\coordinate (V1) at (0+4-1 , 0);
\coordinate (V2) at (0+4-1 , 1);
\coordinate (V3) at (0.75+4-1 , 1);
\coordinate (V4) at (0.75+4-1 , 0);

\coordinate (V5) at (-0.5+4-1 , -0.5);
\coordinate (V6) at (0.5+0.75+4-1 , -0.5);

\coordinate (V7) at (-0.5+4-1 , 1.5);
\coordinate (V8) at (0.5+0.75+4-1 , 1.5);

\coordinate (V9) at (0.375+4-1 , 0.5);

\draw[-,very thick] (V1) to (V2);
\draw[-,very thick] (V2) to (V3);
\draw[-,very thick] (V3) to (V4);
\draw[-,very thick] (V1) to (V4);

\draw[-,very thick] (V1) to (V5);
\draw[-,very thick] (V4) to (V6);
\draw[-,very thick] (V2) to (V7);
\draw[-,very thick] (V3) to (V8);

\vertexLabelR[]{V1}{left}{$ $}
\vertexLabelR[]{V2}{left}{$ $}
\vertexLabelR[]{V3}{left}{$ $}
\vertexLabelR[]{V4}{left}{$ $}

\vertexLabelR[]{V5}{left}{$ $}
\vertexLabelR[]{V6}{left}{$ $}
\vertexLabelR[]{V7}{left}{$ $}
\vertexLabelR[]{V8}{left}{$ $}
\draw[-{To[scale=2]}] (1.4,0.5) to (1.9,0.5);

\coordinate (V7) at (0.5 , 0);
\coordinate (V8) at (1/2 , 1);
\coordinate (V9) at (1. , -0.5);
\coordinate (V10) at (0 , -0.5);
\coordinate (V11) at (1 , 1.5 );
\coordinate (V12) at (0 , 1.5);

\draw[-,very thick] (V7) to (V8);
\draw[-,very thick] (V7) to (V9);
\draw[-,very thick] (V7) to (V10);
\draw[-,very thick] (V8) to (V11);
\draw[-,very thick] (V8) to (V12);

 \draw[very thick, gray, smooth ]
  plot coordinates {
    (1,1.3)
    (0.7,1)
    (0.7,0)
    (1,-0.3)
  };

   \draw[very thick, gray, smooth ]
  plot coordinates {
    (0,1.3)
    (0.3,1)
    (0.3,0)
    (0,-0.3)
  };

     \draw[very thick, gray, smooth ]
  plot coordinates {
    (0.2,1.5)
    (0.5,1.2)
    (0.8,1.5)
  };

     \draw[very thick, gray, smooth ]
  plot coordinates {
    (0.2,-.5)
    (0.5,-.2)
    (0.8,-.5)
  };
  
\vertexLabelR[]{V7}{left}{$ $}
\vertexLabelR[]{V8}{left}{$ $}
\vertexLabelR[]{V9}{left}{$ $}
\vertexLabelR[]{V10}{left}{$ $}
\vertexLabelR[]{V11}{left}{$ $}
\vertexLabelR[]{V12}{left}{$ $}

    \node[vertexBall,shift={(1.1,0)}] at (V7) {$x_{(i,e[i,m])}$};

    \node[vertexBall,shift={(-1,-0.2)}] at (V8) {$x_{(j,e[j,r])}$};

    \node[vertexBall,shift={(-1.9,0.8)}] at (V8) {$x_{(j,e[j,r^{h^{-1}}])}$};

    \node[vertexBall,shift={(-1.9,-2.4)}] at (V8) {$x_{(i,e[i,m^{g^{-1}}])}$};

    \node[vertexBall,shift={(1.7,0.8)}] at (V8) {$x_{(j,e[j,r^{h}])}$};

    \node[vertexBall,shift={(1.7,-2.4)}] at (V8) {$x_{(i,e[i,m^{g}])}$};

\draw[very thick, gray, smooth cycle ]
  plot coordinates {
    (4.2-1,0.2)
    (4.2-1,0.8)
    (4.55-1,0.8)
    (4.55-1,0.2)
  };

  \draw[very thick, gray, smooth ]
  plot coordinates {
    (-0.5-0.2+4-1 , -0.5)
    (-0.2+4-1,0)
    (-0.2+4-1,1)
    (-0.5-0.2 +4-1, 1.5)
  };

  \draw[very thick, gray, smooth  ]
  plot coordinates {
    (0.75+0.5+0.2 +4-1, -0.5)
    (0.75+.2+4-1,0)
    (0.75+.2+4-1,1)
    (0.5+0.75+0.2+4-1 , 1.5)
  };

  \draw[very thick, gray, smooth ]
  plot coordinates {
    (0.75+0.5-0.2 +4-1, 1.5)
    (0.75+4-1,1.2)
    (0+4-1,1.2)
    (-0.5+0.2+4-1 , 1.5)
  };

  \draw[very thick, gray, smooth ]
  plot coordinates {
    (0.75+0.5-0.2+4-1 , -0.5)
    (0.75+4-1,-.2)
    (0+4-1,-.2)
    (-0.5+0.2 +4-1, -.5)
  };  \draw[very thick, gray, smooth ]
  plot coordinates {
    (-0.5-0.2+4-1 , -0.5)
    (-0.2+4-1,0)
    (-0.2+4-1,1)
    (-0.5-0.2 +4-1, 1.5)
  };

  \draw[very thick, gray, smooth ]
  plot coordinates {
    (0.75+0.5+0.2 +4-1, -0.5)
    (0.75+.2+4-1,0)
    (0.75+.2+4-1,1)
    (0.5+0.75+0.2+4-1 , 1.5)
  };

  \draw[very thick, gray, smooth ]
  plot coordinates {
    (0.75+0.5-0.2 +4-1, 1.5)
    (0.75+4-1,1.2)
    (0+4-1,1.2)
    (-0.5+0.2+4-1 , 1.5)
  };

  \draw[very thick, gray, smooth ]
  plot coordinates {
    (0.75+0.5-0.2+4-1 , -0.5)
    (0.75+4-1,-.2)
    (0+4-1,-.2)
    (-0.5+0.2 +4-1, -.5)
  };

\begin{scope}[shift={(2,0)}]

\coordinate (V1) at (0+4 , 0);
\coordinate (V2) at (0+4 , 1);
\coordinate (V3) at (0.75+4 , 1);
\coordinate (V4) at (0.75+4 , 0);

\coordinate (V5) at (-0.5+4 , -0.5);
\coordinate (V6) at (0.5+0.75+4 , -0.5);

\coordinate (V7) at (-0.5+4 , 1.5);
\coordinate (V8) at (0.5+0.75+4 , 1.5);

\coordinate (V9) at (0.375+4 , 0.5);

    \node[vertexBall,shift={(-3.8,0.)}] at (V4) {$x_{(i,e,2)}$};

    \node[vertexBall,shift={(-3.8,1.5)}] at (V4) {$x_{(j,e,2)}$};

    \node[vertexBall,shift={(-6.9,0.)}] at (V4) {$x_{(i,e,1)}$};

    \node[vertexBall,shift={(-6.9,1.5)}] at (V4) {$x_{(j,e,1)}$};

\end{scope}
\end{tikzpicture}
\caption{Using an edge $e=e[i,m]=e[j,r]\in E(\Gamma)$ to translate an edge $\{x_{(i,e)},x_{(j,e)}\}\in E(\Gamma_1)$ into a $4$-cycle in the graph $\Gamma_2$  }
    \label{fig:edgecut5}
\end{figure}
If we construct a strong embedding $\beta_2$ of $\Gamma_2$ by modifying the embedding $\beta_1$ of $\Gamma_1$ as illustrated in \Cref{fig:edgecut5}, we obtain $\chi(\beta_2 (\Gamma_2))=\chi(\beta_1 (\Gamma_1))$.
Since the set $$E_1=\{\{x_{(i,e)},x_{(j,e)}\}\mid e=\{i,j\}\in E(\Gamma)\}\subseteq E(\Gamma_1)$$ satisfies $E_1^{\Aut(\Gamma_1)}=E_1,$ 
it follows that every automorphism of $\Gamma_1$ can be extended to an automorphism of $\Gamma_2$.
Hence, $\Aut(\Gamma_2)$ contains a subgroup $H$ that is isomorphic to $\Aut(\Gamma_1)\cong G.$
Finally, we construct the desired graph $\Gamma'$ for $d=4$ and $d=5$ by inserting a subgraph into the described cycles of length $4$, see Figures~\ref{fig:e4} and \ref{fig:5limit}. The goal is to further add edges to the graph $\Gamma_2$ such that (1) the resulting graph $\Gamma'$ is $d$-regular, (2) the introduced subgraph has a $C_2\times C_2$ symmetry and (3) $\Gamma'$ is constructed such that every automorphism $\Gamma'$ can be translated into an automorphism of the graph $\Gamma_2$ that lies in $H\leq \Aut(\Gamma_2).$
More precisely, we want to add further edges to the graph $\Gamma_2$ such that if $(x_{(i,e,1)},x_{(i,e,2)},x_{(j,e,2)},x_{(j,e,1)}),$ is a cycle of length $4$ in $\Gamma'$, then every automorphism $\psi'\in \Aut(\Gamma')$ satisfies $\psi'(\{x_{(i,e,1)},x_{(i,e,2)}\})=\{x_{(i',e',1)},x_{(i',e',2)}\}$ and $\psi'(\{x_{(j,e,1)},x_{(j,e,2)}\})=\{x_{(j',e',1)},x_{(j',e',2)}\},$ where $(x_{(i',e',1)},x_{(i',e',2)},x_{(j',e',2)},x_{(j',e',1)})$ is another cycle of length $4$ in $\Gamma'$ resulting from an edge $e'=\{i',j'\}\in E(\Gamma).$
Here, we achieve this by conducting a case distinction with respect to $d=4,5.$ 
\paragraph{The case $d=4$:}
We define the vertices of $\Gamma'$ to be given by $$V(\Gamma_2)\cup \{y_{(e,a)}\mid e \in E(\Gamma) \text{ and } a=1,\ldots,8\}.$$
Furthermore, the edges of $\Gamma'$ shall be given by $E(\Gamma_2)\cup E$, where $E$ are the edges corresponding to the quadratic form 
\begin{align*}
\sum_{e:=\{i,j\}\in E(\Gamma), i<j}&\big(x_{(i,e,1)}y_{(e,1)}+  x_{(i,e,2)}y_{(e,2)}  +y_{(e,1)}y_{(e,2)}
+y_{(e,2)}y_{(e,3)}\\
+&y_{(e,1)}y_{(e,3)}+y_{(e,1)}y_{(e,4)}+y_{(e,3)}y_{(e,4)}+y_{(e,2)}y_{(e,5)}\\
+&y_{(e,3)}y_{(e,5)}+y_{(e,5)}y_{(e,6)}
+y_{(e,4)}y_{(e,6)}+y_{(e,4)}y_{(e,7)}\\
+&y_{(e,5)}y_{(e,8)}+y_{(e,6)}y_{(e,7)}
+y_{(e,6)}y_{(e,8)}+y_{(e,7)}y_{(e,8)}\\
+&x_{(j,e,1)}y_{(e,7)}+x_{(j,e,2)}y_{(e, 8)}\big).
\end{align*}
This subgraph is illustrated in Figure~\ref{subfig:e4}.
As discussed above, the subgraph induced by $\{y_{(e,a)} \mid a=1,\ldots,8\}$ admits a $C_2 \times C_2$ symmetry. It therefore follows that every $\psi_2 \in H $ can be extended to an automorphism of the graph $\Gamma'.$ Hence, $G\cong H \hookrightarrow \Aut(\Gamma').$ Thus, it remains to show that the automorphism group is indeed isomorphic to $G$.
For this, we observe that each vertex $x_{(i,e,m)}$ is not contained in any cycle of length $3$, whereas each vertex $y_{(e,a)}$ lies in at least two such cycles. This distinction implies that $V(\Gamma_2)$ considered as a subset of $V(\Gamma')$ satisfies:
\[
V(\Gamma_2)^{\Aut(\Gamma')} = V(\Gamma_2).
\]
Now, let $e = \{i,j\} \in E(\Gamma)$ with $i < j$ be an edge. Then:
\begin{enumerate}
    \item $x_{(i,e,1)},x_{(i,e,2)}\in V(\Gamma_2)$ are contained in exactly two $4$-cycles, namely $(x_{(i,e,1)},x_{(i,e,2)},y_{(e,1)},y_{(e,2)})$ and $(x_{(i,e,1)},x_{(i,e,2)},x_{(j,e,2)},x_{(j,e,1)})$,
    \item $x_{(j,e,1)},x_{(j,e,2)}\in V(\Gamma_2)$ are contained in exactly two 4-cycles, namely $(x_{(j,e,1)},x_{(j,e,2)},y_{(e,8)},y_{(e,7)})$ and $(x_{(i,e,1)},x_{(i,e,2)},x_{(j,e,2)},x_{(j,e,1)})$.
\end{enumerate}
and no other pair of vertices in $\{x_{(i,e,1)}, x_{(i,e,2)}, x_{(j,e,1)}, x_{(j,e,2)}\}$ is contained in exactly two cycles of length $4.$
Thus, it follows that every edge $ \{x_{(i,e,1)}, x_{(i,e,2)}\} \in E(\Gamma')$ is mapped under any automorphism $\psi'$ of $\Gamma'$ to an edge of the form $\{x_{(i',e',1)}, x_{(i',e',2)}\}$. Hence, each automorphism $\psi'$ induces an automorphism $\psi_2 \in H$.
Therefore,
$\Aut(\Gamma') \cong H \cong G.$

If we modify the strong embedding $\beta_2$ as illustrated in Figure~\ref{subfig:4reg_limit_cycles} to get a strong embedding $\beta'$ of $\Gamma',$ we conclude $\Aut(\Gamma')\cong\Aut(\beta'(\Gamma'))\cong G.$
As illustrated in the described figure, if $\mathcal{C}_2$ is the cycle double cover corresponding to $\beta_2$, then the facial cycle double cover $\mathcal{C}'$ corresponding to $\beta'$ has $\vert \mathcal{C}_2\vert$ cycles that correspond to cycles in $\Gamma_2.$ Furthermore, each edge $e\in E(\Gamma)$ contributes $8$ additional vertices, $18$ additional edges and $8$ cycles to $\Gamma'$ compared to $\Gamma_2$.
\begin{align*}
    \chi(\beta'(\Gamma'))&=(\vert V(\Gamma_2)\vert + 8\cdot\vert E(\Gamma)\vert) - (\vert E(\Gamma_2)\vert + 18\cdot\vert E(\Gamma)\vert) + (\vert F(\beta_2(\Gamma_2))\vert +8\cdot\vert E(\Gamma)\vert)\\
    &=\chi(\beta_2(\Gamma_2))-2\cdot\vert E(\Gamma)\vert =\chi(\beta_1(\Gamma_1))-2\cdot\vert E(\Gamma)\vert =\chi(\beta(\Gamma))-2\cdot\vert E(\Gamma)\vert.
\end{align*}
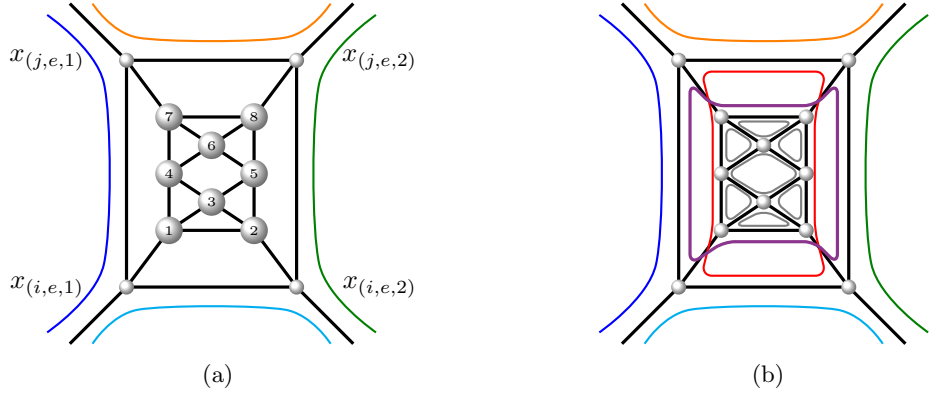
\begin{figure}[H]
    \centering
    \begin{subfigure}{.48\textwidth}
        \centering
        \begin{tikzpicture}[vertexBall, edgeDouble=nolabels, faceStyle=nolabels, scale=3]

\coordinate (V1) at (0+4-1 , 0);
\coordinate (V2) at (0+4-1 , 1);
\coordinate (V3) at (0.75+4-1 , 1);
\coordinate (V4) at (0.75+4-1 , 0);

\coordinate (V5) at (-0.5+4-1 , -0.5);
\coordinate (V6) at (0.5+0.75+4-1 , -0.5);

\coordinate (V7) at (-0.5+4-1 , 1.5);
\coordinate (V8) at (0.5+0.75+4-1 , 1.5);

\coordinate (V9) at (0.375+4-1 , 0.5);

\coordinate (V7) at (0.5 , 0);
\coordinate (V8) at (1/2 , 1);
\coordinate (V9) at (1. , -0.5);
\coordinate (V10) at (0 , -0.5);
\coordinate (V11) at (1 , 1.5 );
\coordinate (V12) at (0 , 1.5);

\begin{scope}[shift={(2,0)}]

\coordinate (V1) at (0+4 , 0);
\coordinate (V2) at (0+4 , 1);
\coordinate (V3) at (0.75+4 , 1);
\coordinate (V4) at (0.75+4 , 0);

\coordinate (V5) at (-0.25+4 , -0.25);
\coordinate (V6) at (0.25+0.75+4 , -0.25);

\coordinate (V7) at (-0.25+4 , 1.25);
\coordinate (V8) at (0.25+0.75+4 , 1.25);

\coordinate (V9) at (0.375+4 , 0.5);

\draw[-,very thick] (V1) to (V2);
\draw[-,very thick] (V2) to (V3);
\draw[-,very thick] (V3) to (V4);
\draw[-,very thick] (V1) to (V4);

\draw[-,very thick] (V1) to (V5);
\draw[-,very thick] (V4) to (V6);
\draw[-,very thick] (V2) to (V7);
\draw[-,very thick] (V3) to (V8);

\draw[thick, blue, smooth ]
  plot coordinates {
    (-0.15-0.2+4 , -0.2)
    (-0.1+4,0.1)
    (-0.1+4,0.9)
    (-0.15-0.2+4 , 1.2)
  };
  \draw[thick, green!50!black, smooth ]
  plot coordinates {
    (0.15+0.75+0.2+4, -0.2)
    (0.75+.1+4,0.1)
    (0.75+.1+4,0.9)
    (0.15+0.75+0.2+4, 1.2)
  };

 \draw[thick, orange, smooth ]
  plot coordinates {
    (0.75+0.35-0.2+4 , 1.25)
    (0.7+4,1.1)
    (0.05+4,1.1)
    (-0.35+0.2 +4, 1.25)
  };

  \draw[thick, cyan, smooth ]
  plot coordinates {
    (0.75+0.35-0.2+4, -0.25)
    (0.7+4,-.1)
    (0.05+4,-.1)
    (-0.35+0.2 +4 , -.25)
  };

\coordinate (V12) at (0 , 0);
\coordinate (V22) at (0 , 1);
\coordinate (V32) at (0.75 , 1);
\coordinate (V42) at (0.75 , 0);

\coordinate (V52) at (-0.5 , -0.5);
\coordinate (V62) at (0.5+0.75 , -0.5);

\coordinate (V72) at (-0.5 , 1.5);
\coordinate (V82) at (0.5+0.75 , 1.5);

\coordinate (V92) at (0.375 , 0.5);

\coordinate (B1) at (0.375+4 , 0.25);
\coordinate (B1l) at (0.75*0.25+4 , 0.25);

\coordinate (B1r) at (0.75*0.75+4 , 0.25);
\coordinate (B12) at (0.375+4 , 0.75);

\coordinate (B12l) at (0.75*0.25+4 , 0.75);

\coordinate (B12r) at (0.75*0.75+4 , 0.75);

\coordinate (Bmr) at (0.75*0.75+4 , 0.5);

\coordinate (Bm1) at (0.75*0.5+4 , 0.625);

\coordinate (Bm2) at (0.75*0.5+4 , 0.375);

\coordinate (Bml) at (0.75*0.25+4, 0.5);

\draw[-,very thick] (Bm2) to (Bml);
\draw[-,very thick] (Bm2) to (Bmr);
\draw[-,very thick] (B1r) to (B1l);

\draw[-,very thick] (B12r) to (B12l);

\draw[-,very thick] (V1) to (B1l);

\draw[-,very thick] (V4) to (B1r);

\draw[-,very thick] (Bm2) to (B1r);
\draw[-,very thick] (Bm2) to (B1l);

\draw[-,very thick] (Bm1) to (B12r);
\draw[-,very thick] (Bm1) to (B12l);

\draw[-,very thick] (Bm1) to (Bml);
\draw[-,very thick] (Bm1) to (Bmr);

\draw[-,very thick] (V2) to (B12l);

\draw[-,very thick] (V3) to (B12r);

\draw[-,very thick] (Bmr) to (B12r);
\draw[-,very thick] (Bmr) to (B1r);

\draw[-,very thick] (Bml) to (B12l);
\draw[-,very thick] (Bml) to (B1l);

\vertexLabelR[]{B1l}{left}{$1$}
\vertexLabelR[]{B1r}{left}{$2$}
\vertexLabelR[]{B12l}{left}{$7$}
\vertexLabelR[]{B12r}{left}{$8$}
\vertexLabelR[]{Bmr}{left}{$5$}
\vertexLabelR[]{Bm1}{left}{$6$}
\vertexLabelR[]{Bm2}{left}{$3$}
\vertexLabelR[]{Bml}{left}{$4$}

\vertexLabelR[]{V1}{left}{$ $}
\vertexLabelR[]{V2}{left}{$ $}
\vertexLabelR[]{V3}{left}{$ $}
\vertexLabelR[]{V4}{left}{$ $}

\node[vertexBall,shift={(1.1,0.)}] at (V4) {$x_{(i,e,2)}$};

\node[vertexBall,shift={(1.1,3)}] at (V4) {$x_{(j,e,2)}$};

\node[vertexBall,shift={(-3.3,0.)}] at (V4) {$x_{(i,e,1)}$};

\node[vertexBall,shift={(-3.3,3)}] at (V4) {$x_{(j,e,1)}$};

\end{scope}
\end{tikzpicture} 
        \subcaption{}
        \label{subfig:e4}
    \end{subfigure}
    \begin{subfigure}{.48\textwidth}
        \centering
        \begin{tikzpicture}[vertexBall, edgeDouble=nolabels, faceStyle=nolabels, scale=3]

\coordinate (V1) at (0+4 , 0);
\coordinate (V2) at (0+4 , 1);
\coordinate (V3) at (0.75+4 , 1);
\coordinate (V4) at (0.75+4 , 0);

\coordinate (V5) at (-0.25+4 , -0.25);
\coordinate (V6) at (0.25+0.75+4 , -0.25);

\coordinate (V7) at (-0.25+4 , 1.25);
\coordinate (V8) at (0.25+0.75+4 , 1.25);

\coordinate (V9) at (0.375+4 , 0.5);

\draw[-,very thick] (V1) to (V2);
\draw[-,very thick] (V2) to (V3);
\draw[-,very thick] (V3) to (V4);
\draw[-,very thick] (V1) to (V4);

\draw[-,very thick] (V1) to (V5);
\draw[-,very thick] (V4) to (V6);
\draw[-,very thick] (V2) to (V7);
\draw[-,very thick] (V3) to (V8);

\draw[thick, blue, smooth ]
  plot coordinates {
    (-0.15-0.2+4 , -0.2)
    (-0.1+4,0.1)
    (-0.1+4,0.9)
    (-0.15-0.2+4 , 1.2)
  };
  \draw[thick, green!50!black, smooth ]
  plot coordinates {
    (0.15+0.75+0.2+4, -0.2)
    (0.75+.1+4,0.1)
    (0.75+.1+4,0.9)
    (0.15+0.75+0.2+4, 1.2)
  };

 \draw[thick, orange, smooth ]
  plot coordinates {
    (0.75+0.35-0.2+4 , 1.25)
    (0.7+4,1.1)
    (0.05+4,1.1)
    (-0.35+0.2 +4, 1.25)
  };

  \draw[thick, cyan, smooth ]
  plot coordinates {
    (0.75+0.35-0.2+4, -0.25)
    (0.7+4,-.1)
    (0.05+4,-.1)
    (-0.35+0.2 +4 , -.25)
  };

\coordinate (V12) at (0 , 0);
\coordinate (V22) at (0 , 1);
\coordinate (V32) at (0.75 , 1);
\coordinate (V42) at (0.75 , 0);

\coordinate (V52) at (-0.5 , -0.5);
\coordinate (V62) at (0.5+0.75 , -0.5);

\coordinate (V72) at (-0.5 , 1.5);
\coordinate (V82) at (0.5+0.75 , 1.5);

\coordinate (V92) at (0.375 , 0.5);

\coordinate (B1) at (0.375+4 , 0.25);
\coordinate (B1l) at (0.75*0.25+4 , 0.25);

\coordinate (B1r) at (0.75*0.75+4 , 0.25);
\coordinate (B12) at (0.375+4 , 0.75);

\coordinate (B12l) at (0.75*0.25+4 , 0.75);

\coordinate (B12r) at (0.75*0.75+4 , 0.75);

\coordinate (Bmr) at (0.75*0.75+4 , 0.5);

\coordinate (Bm1) at (0.75*0.5+4 , 0.625);

\coordinate (Bm2) at (0.75*0.5+4 , 0.375);

\coordinate (Bml) at (0.75*0.25+4, 0.5);

\draw[-,very thick] (Bm2) to (Bml);
\draw[-,very thick] (Bm2) to (Bmr);
\draw[-,very thick] (B1r) to (B1l);

\draw[-,very thick] (B12r) to (B12l);

\draw[-,very thick] (V1) to (B1l);

\draw[-,very thick] (V4) to (B1r);

\draw[-,very thick] (Bm2) to (B1r);
\draw[-,very thick] (Bm2) to (B1l);

\draw[-,very thick] (Bm1) to (B12r);
\draw[-,very thick] (Bm1) to (B12l);

\draw[-,very thick] (Bm1) to (Bml);
\draw[-,very thick] (Bm1) to (Bmr);

\draw[-,very thick] (V2) to (B12l);

\draw[-,very thick] (V3) to (B12r);

\draw[-,very thick] (Bmr) to (B12r);
\draw[-,very thick] (Bmr) to (B1r);

\draw[-,very thick] (Bml) to (B12l);
\draw[-,very thick] (Bml) to (B1l);

\vertexLabelR[]{B1l}{left}{$ $}
\vertexLabelR[]{B1r}{left}{$ $}
\vertexLabelR[]{B12l}{left}{$ $}
\vertexLabelR[]{B12r}{left}{$ $}
\vertexLabelR[]{Bmr}{left}{$ $}
\vertexLabelR[]{Bm1}{left}{$ $}
\vertexLabelR[]{Bm2}{left}{$ $}
\vertexLabelR[]{Bml}{left}{$ $}

\vertexLabelR[]{V1}{left}{$ $}
\vertexLabelR[]{V2}{left}{$ $}
\vertexLabelR[]{V3}{left}{$ $}
\vertexLabelR[]{V4}{left}{$ $}

\draw[red, thick,rounded corners] (4.1, 0.05) -- (4.65, 0.05) -- (4.6, 0.25) -- (4.6,0.75) -- (4.65,0.95) -- (4.1, 0.95) -- (4.15,0.75) -- (4.15,0.25) -- cycle;

\draw[Fuchsia, very thick,rounded corners] (4.05, 0.1) -- (4.2, 0.2) -- (4.6, 0.2) -- (4.7,0.1) -- (4.7,0.9) -- (4.6, 0.8) -- (4.15,0.8) -- (4.05,0.9) -- cycle;

\draw[gray, thick,rounded corners] (4.24, 0.27) -- (4.375, 0.34) -- (4.51,0.27) -- cycle;

\draw[gray, thick,rounded corners] (4.21, 0.28) -- (4.21, 0.47) -- (4.34,0.38) -- cycle;

\draw[gray, thick,rounded corners] (4.54, 0.28) -- (4.54, 0.47) -- (4.42,0.38) -- cycle;

\draw[gray, thick,rounded corners] (4.375, 0.4) -- (4.21, 0.5) -- (4.375, 0.6) -- (4.54, 0.5) -- cycle;

\draw[gray, thick,rounded corners] (4.21, 0.535) -- (4.21, 0.715) -- (4.34, 0.63) -- cycle;

\draw[gray, thick,rounded corners] (4.54, 0.535) -- (4.54, 0.715) -- (4.42, 0.63) -- cycle;

\draw[gray, thick,rounded corners] (4.24, 0.73) -- (4.375, 0.65) -- (4.51,0.73) -- cycle;
\end{tikzpicture} 
        \subcaption{}
        \label{subfig:4reg_limit_cycles}
    \end{subfigure}
    \caption{(a) The 4-regular graph $\Gamma'$, where each vertex $y_{(e,a)}$ is indicated by the label $a$ for better illustration. (b) The constructed $\Aut(\Gamma')$-invariant facial cycle double cover of $\Gamma'$.}
    \label{fig:e4}
\end{figure}

\paragraph{The case $d=5$:}
We define the vertices of $\Gamma'$ to be given by $$V(\Gamma_2)\cup \{y_{(e,a)}\mid e \in E(\Gamma)\text{ and } a=1,\ldots,10\}.$$ 
Additionally, we define the edges of $\Gamma'$ as $E(\Gamma'):=E(\Gamma_2)\cup E,$ where $E$ is the set of edges in $\Gamma'$ that corresponds to the quadratic form:

\begin{align*}
\sum_{e:=\{i,j\}\in E(\Gamma), i<j}&\big(x_{(i,e,1)}y_{(e,1)}+  x_{(i,e,2)}y_{(e,3)}+x_{(i,e,1)}y_{(e,2)}+  x_{(i,e,2)}y_{(e,2)} +y_{(e,1)}y_{(e,2)}
+y_{(e,2)}y_{(e,3)}\\
+&y_{(e,2)}y_{(e,4)}+y_{(e,1)}y_{(e,4)}+y_{(e,3)}y_{(e,4)}+y_{(e,1)}y_{(e,5)}\\
+&y_{(e,3)}y_{(e,6)}+y_{(e,1)}y_{(e,8)}
+y_{(e,4)}y_{(e,5)}+y_{(e,4)}y_{(e,6)}\\
+&y_{(e,3)}y_{(e,10)}+y_{(e,5)}y_{(e,6)}
+y_{(e,5)}y_{(e,7)}+y_{(e,6)}y_{(e,7)}\\
+&y_{(e,5)}y_{(e,8)}+y_{(e,7)}y_{(e,8)}+y_{(e,7)}y_{(e,9)}+y_{(e,7)}y_{(e,10)}\\
+&y_{(e,6)}y_{(e,10)}+y_{(e,8)}y_{(e,9)}+y_{(e,9)}y_{(e,10)}\\
+&x_{(j,e,1)}y_{(e,8)}+x_{(j,e,2)}y_{(e, 10)}+x_{(j,e,1)}y_{(e,9)}+x_{(j,e,2)}y_{(e, 9)}\big).
\end{align*}
This subgraph is illustrated in Figure~\ref{fig:5limit}.
Again, the subgraph induced by $\{y_{(e,a)} \mid a=1,\ldots,10\}$ admits a $C_2 \times C_2$ symmetry. It therefore follows that every $\psi_2 \in H $ can be extended to an automorphism of the graph $\Gamma'.$ Hence, $G \hookrightarrow \Aut(\Gamma').$ 
Next, observe that each vertex $x_{(i,e,t)}$ is contained in exactly two cycles of length $3$, whereas each vertex $y_{(e,a)}$ lies in at least four such cycles. This distinction implies that 
\[
V(\Gamma_2)^{\Aut(\Gamma')} = V(\Gamma_2).
\]
Now, let $e = \{i,j\} \in E(\Gamma)$ with $i < j$ be an edge. Then:
\begin{enumerate}
    \item the vertices $x_{(i,e,1)}$ and $x_{(i,e,2)}$ share a common neighbour, namely $y_{(e,2)}$, and
    \item the vertices $x_{(j,e,1)}$ and $x_{(j,e,2)}$ share a common neighbour, namely $y_{(e,9)}$,
\end{enumerate}
and no other pair of vertices in $\{x_{(i,e,1)}, x_{(i,e,2)}, x_{(j,e,1)}, x_{(j,e,2)}\}$ shares a common neighbour.
It follows that every edge $\{x_{(i,e,1)}, x_{(i,e,2)}\} \in E(\Gamma')$ is mapped under any automorphism $\psi'$ of $\Gamma'$ to an edge of the form $\{x_{(i',e',1)}, x_{(i',e',2)}\}$ as in the above case. We therefore obtain $\Aut(\Gamma') \cong \{\psi_\phi \mid \phi \in \Aut(\Gamma)\} \cong G.$

If we modify the strong embedding $\beta_2$ as illustrated in \Cref{subfig:5reg_limit_cycles} to get a strong embedding $\beta'$ of $\Gamma',$ we conclude $\Aut(\Gamma')\cong\Aut(\beta'(\Gamma'))\cong G.$
Furthermore, a closer look at the Euler characteristic gives us the following counts: If $\mathcal{C}_2$ is the cycle double cover corresponding to $\beta_2$, then the facial cycle double cover $\mathcal{C}'$ corresponding to $\beta'$ has $\vert \mathcal{C}_2\vert$ cycles that correspond to cycles in $\Gamma_2.$
Furthermore, each edge $e\in E(\Gamma)$ contributes $10$ additional vertices, $29$ additional edges and $17$ cycles to $\Gamma'$ compared to $\Gamma_2$.
\begin{align*}
    \chi(\beta'(\Gamma'))&=(\vert V(\Gamma_2)\vert + 10\cdot\vert E(\Gamma)\vert) - (\vert E(\Gamma_2)\vert + 29\cdot\vert E(\Gamma)\vert) + (\vert F(\beta_2(\Gamma_2))\vert+17\cdot\vert E(\Gamma)\vert)\\
    &=\chi(\beta_2(\Gamma_2))-2\cdot\vert E(\Gamma)\vert =\chi(\beta_1(\Gamma_1))-2\cdot\vert E(\Gamma)\vert =\chi(\beta(\Gamma))-2\cdot\vert E(\Gamma)\vert.
\end{align*}
Thus, the result follows.
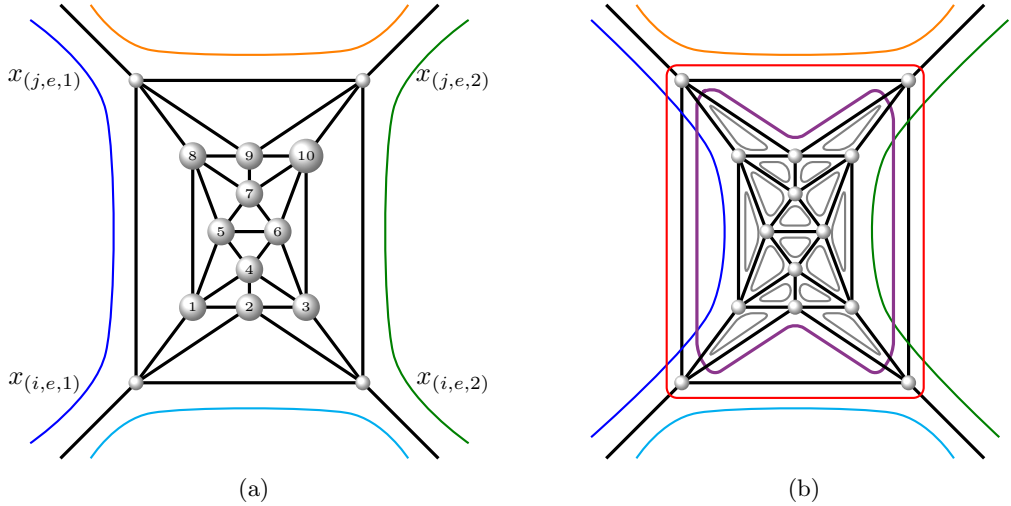
\begin{figure}[H]
    \centering
    \begin{subfigure}{.48\textwidth}
        \centering
        \begin{tikzpicture}[vertexBall, edgeDouble=nolabels, faceStyle=nolabels, scale=4]

\coordinate (V1) at (0+4 , 0);
\coordinate (V2) at (0+4 , 1);
\coordinate (V3) at (0.75+4 , 1);
\coordinate (V4) at (0.75+4 , 0);

\coordinate (V5) at (-0.25+4 , -0.25);
\coordinate (V6) at (0.25+0.75+4 , -0.25);

\coordinate (V7) at (-0.25+4 , 1.25);
\coordinate (V8) at (0.25+0.75+4 , 1.25);

\coordinate (V9) at (0.375+4 , 0.5);

\draw[-,very thick] (V1) to (V2);
\draw[-,very thick] (V2) to (V3);
\draw[-,very thick] (V3) to (V4);
\draw[-,very thick] (V1) to (V4);

\draw[-,very thick] (V1) to (V5);
\draw[-,very thick] (V4) to (V6);
\draw[-,very thick] (V2) to (V7);
\draw[-,very thick] (V3) to (V8);

\draw[thick, blue, smooth ]
  plot coordinates {
    (-0.15-0.2+4 , -0.2)
    (-0.1+4,0.1)
    (-0.1+4,0.9)
    (-0.15-0.2+4 , 1.2)
  };
  \draw[thick, green!50!black, smooth ]
  plot coordinates {
    (0.15+0.75+0.2+4, -0.2)
    (0.75+.1+4,0.1)
    (0.75+.1+4,0.9)
    (0.15+0.75+0.2+4, 1.2)
  };

 \draw[thick, orange, smooth ]
  plot coordinates {
    (0.75+0.35-0.2+4 , 1.25)
    (0.7+4,1.1)
    (0.05+4,1.1)
    (-0.35+0.2 +4, 1.25)
  };

  \draw[thick, cyan, smooth ]
  plot coordinates {
    (0.75+0.35-0.2+4, -0.25)
    (0.7+4,-.1)
    (0.05+4,-.1)
    (-0.35+0.2 +4 , -.25)
  };

\coordinate (V12) at (0 , 0);
\coordinate (V22) at (0 , 1);
\coordinate (V32) at (0.75 , 1);
\coordinate (V42) at (0.75 , 0);

\coordinate (V52) at (-0.5 , -0.5);
\coordinate (V62) at (0.5+0.75 , -0.5);

\coordinate (V72) at (-0.5 , 1.5);
\coordinate (V82) at (0.5+0.75 , 1.5);

\coordinate (V92) at (0.375 , 0.5);

\coordinate (B1) at (0.375+4 , 0.25);
\coordinate (B1l) at (0.75*0.25+4 , 0.25);

\coordinate (B1r) at (0.75*0.75+4 , 0.25);
\coordinate (B12) at (0.375+4 , 0.75);

\coordinate (B12l) at (0.75*0.25+4 , 0.75);

\coordinate (B12r) at (0.75*0.75+4 , 0.75);

\coordinate (Bmr) at (0.75*0.625+4 , 0.5);

\coordinate (Bm1) at (0.75*0.5+4 , 0.625);

\coordinate (Bm2) at (0.75*0.5+4 , 0.375);

\coordinate (Bml) at (0.75*0.375+4 , 0.5);

\draw[-,very thick] (Bm1) to (Bml);
\draw[-,very thick] (Bm1) to (Bmr);
\draw[-,very thick] (Bm2) to (Bml);
\draw[-,very thick] (Bm2) to (Bmr);
\draw[-,very thick] (Bmr) to (Bml);
\draw[-,very thick] (B1r) to (B1l);
\draw[-,very thick] (B12r) to (B12l);

\draw[-,very thick] (B12l) to (B1l);
\draw[-,very thick] (B12r) to (B1r);

\draw[-,very thick] (V1) to (B1l);
\draw[-,very thick] (V1) to (B1);

\draw[-,very thick] (V4) to (B1);
\draw[-,very thick] (V4) to (B1r);

\draw[-,very thick] (Bm2) to (B1);

\draw[-,very thick] (Bm1) to (B12);

\draw[-,very thick] (V2) to (B12);
\draw[-,very thick] (V2) to (B12l);

\draw[-,very thick] (V3) to (B12);
\draw[-,very thick] (V3) to (B12r);

\draw[-,very thick] (Bmr) to (B12r);
\draw[-,very thick] (Bmr) to (B1r);

\draw[-,very thick] (Bml) to (B12l);
\draw[-,very thick] (Bml) to (B1l);

\draw[-,very thick] (Bm2) to (B1l);
\draw[-,very thick] (Bm2) to (B1r);
\draw[-,very thick] (Bm1) to (B12l);
\draw[-,very thick] (Bm1) to (B12r);

\vertexLabelR[]{B1}{left}{$ 2$}
\vertexLabelR[]{B12}{left}{$9$}
\vertexLabelR[]{B1l}{left}{$1$}
\vertexLabelR[]{B1r}{left}{$3$}
\vertexLabelR[]{B12l}{left}{$8$}
\vertexLabelR[]{B12r}{left}{$10$}
\vertexLabelR[]{Bmr}{left}{$6$}
\vertexLabelR[]{Bm1}{left}{$7$}
\vertexLabelR[]{Bm2}{left}{$4$}
\vertexLabelR[]{Bml}{left}{$5$}

\vertexLabelR[]{V1}{left}{$ $}
\vertexLabelR[]{V2}{left}{$ $}
\vertexLabelR[]{V3}{left}{$ $}
\vertexLabelR[]{V4}{left}{$ $}

\node[vertexBall,shift={(1.2,0)}] at (V4) {$x_{(i,e,2)}$};

\node[vertexBall,shift={(1.2,0)}] at (V3) {$x_{(j,e,2)}$};

\node[vertexBall,shift={(-1.2,0)}] at (V2) {$x_{(j,e,1)}$};

\node[vertexBall,shift={(-1.2,0)}] at (V1) {$x_{(i,e,1)}$};
\end{tikzpicture}  
        \subcaption{}
    \end{subfigure}
    \begin{subfigure}{.48\textwidth}
        \centering
        \begin{tikzpicture}[vertexBall, edgeDouble=nolabels, faceStyle=nolabels, scale=4]

\coordinate (V1) at (0+4 , 0);
\coordinate (V2) at (0+4 , 1);
\coordinate (V3) at (0.75+4 , 1);
\coordinate (V4) at (0.75+4 , 0);

\coordinate (V5) at (-0.25+4 , -0.25);
\coordinate (V6) at (0.25+0.75+4 , -0.25);

\coordinate (V7) at (-0.25+4 , 1.25);
\coordinate (V8) at (0.25+0.75+4 , 1.25);

\coordinate (V9) at (0.375+4 , 0.5);

\draw[-,very thick] (V1) to (V2);
\draw[-,very thick] (V2) to (V3);
\draw[-,very thick] (V3) to (V4);
\draw[-,very thick] (V1) to (V4);

\draw[-,very thick] (V1) to (V5);
\draw[-,very thick] (V4) to (V6);
\draw[-,very thick] (V2) to (V7);
\draw[-,very thick] (V3) to (V8);

\draw[thick, blue, smooth ]
  plot coordinates {
    (-0.1-0.2+4, -0.18)
    (-0.15+4.25,0.25)
    (-0.15+4.25,0.75)
    (-0.1-0.2+4, 1.2)
  };
  \draw[thick, green!50!black, smooth ]
  plot coordinates {
    (0.75+0.1+0.2+4 , -0.18)
    (0.75+.17+3.75,0.25)
    (0.75+.17+3.75,0.75)
    (0.25+0.63+0.2+4 , 1.2)
  };

 \draw[thick, orange, smooth ]
  plot coordinates {
    (0.75+0.35-0.2+4 , 1.25)
    (0.7+4,1.1)
    (0.05+4,1.1)
    (-0.35+0.2 +4, 1.25)
  };

  \draw[thick, cyan, smooth ]
  plot coordinates {
    (0.75+0.35-0.2+4, -0.25)
    (0.7+4,-.1)
    (0.05+4,-.1)
    (-0.35+0.2 +4 , -.25)
  };

\draw[red, thick,rounded corners] (3.95, -0.05) -- (4.8, -0.05) -- (4.8, 1.05) -- (3.95, 1.05) -- cycle;

\draw[Fuchsia, very thick,rounded corners] (4.05, 0.1) -- (4.1,0.02) -- (4.375, 0.2) -- (4.65,0.02) -- (4.7, 0.1) -- (4.7, 0.9) -- (4.65,0.98) -- (4.375,0.8) -- (4.1,0.98) -- (4.05,0.95) -- cycle;

\draw[gray, thick,rounded corners] (4.24, 0.27) -- (4.36, 0.34) -- (4.36,0.27) -- cycle;

\draw[gray, thick,rounded corners] (4.39, 0.27) -- (4.39, 0.34) -- (4.51,0.27) -- cycle;

\draw[gray, thick,rounded corners] (4.22, 0.29) -- (4.29, 0.47) -- (4.35,0.38) -- cycle;

\draw[gray, thick,rounded corners] (4.53, 0.29) -- (4.46, 0.47) -- (4.4,0.38) -- cycle;

\draw[gray, thick,rounded corners] (4.375, 0.4) -- (4.31, 0.48) -- (4.44, 0.48) -- cycle;

\draw[gray, thick,rounded corners] (4.31, 0.52) -- (4.375, 0.6) -- (4.44, 0.52) -- cycle;

\draw[gray, thick,rounded corners] (4.29, 0.535) -- (4.22, 0.71) -- (4.35, 0.62) -- cycle;

\draw[gray, thick,rounded corners] (4.46, 0.535) -- (4.53, 0.71) -- (4.4, 0.62) -- cycle;

\draw[gray, thick,rounded corners] (4.4, 0.73) -- (4.39, 0.65) -- (4.51,0.73) -- cycle;

\draw[gray, thick,rounded corners] (4.24, 0.73) -- (4.36, 0.65) -- (4.35,0.73) -- cycle;

\draw[gray, thick,rounded corners] (4.21, 0.35) -- (4.21, 0.65) -- (4.26,0.49) -- cycle;

\draw[gray, thick,rounded corners] (4.54, 0.35) -- (4.54, 0.65) -- (4.49,0.49) -- cycle;

\draw[gray, thick,rounded corners] (4.2, 0.23) -- (4.08, 0.08) -- (4.3,0.23) -- cycle;

\draw[gray, thick,rounded corners] (4.45, 0.23) -- (4.67, 0.08) -- (4.55,0.23) -- cycle;

\draw[gray, thick,rounded corners] (4.2, 0.77) -- (4.08, 0.93) -- (4.3,0.77) -- cycle;

\draw[gray, thick,rounded corners] (4.45, 0.77) -- (4.67, 0.93) -- (4.55,0.77) -- cycle;

\coordinate (V12) at (0 , 0);
\coordinate (V22) at (0 , 1);
\coordinate (V32) at (0.75 , 1);
\coordinate (V42) at (0.75 , 0);

\coordinate (V52) at (-0.5 , -0.5);
\coordinate (V62) at (0.5+0.75 , -0.5);

\coordinate (V72) at (-0.5 , 1.5);
\coordinate (V82) at (0.5+0.75 , 1.5);

\coordinate (V92) at (0.375 , 0.5);

\coordinate (B1) at (0.375+4 , 0.25);
\coordinate (B1l) at (0.75*0.25+4 , 0.25);

\coordinate (B1r) at (0.75*0.75+4 , 0.25);
\coordinate (B12) at (0.375+4 , 0.75);

\coordinate (B12l) at (0.75*0.25+4 , 0.75);

\coordinate (B12r) at (0.75*0.75+4 , 0.75);

\coordinate (Bmr) at (0.75*0.625+4 , 0.5);

\coordinate (Bm1) at (0.75*0.5+4 , 0.625);

\coordinate (Bm2) at (0.75*0.5+4 , 0.375);

\coordinate (Bml) at (0.75*0.375+4 , 0.5);

\draw[-,very thick] (Bm1) to (Bml);
\draw[-,very thick] (Bm1) to (Bmr);
\draw[-,very thick] (Bm2) to (Bml);
\draw[-,very thick] (Bm2) to (Bmr);
\draw[-,very thick] (Bmr) to (Bml);
\draw[-,very thick] (B1r) to (B1l);
\draw[-,very thick] (B12r) to (B12l);

\draw[-,very thick] (B12l) to (B1l);
\draw[-,very thick] (B12r) to (B1r);

\draw[-,very thick] (V1) to (B1l);
\draw[-,very thick] (V1) to (B1);

\draw[-,very thick] (V4) to (B1);
\draw[-,very thick] (V4) to (B1r);

\draw[-,very thick] (Bm2) to (B1);

\draw[-,very thick] (Bm1) to (B12);

\draw[-,very thick] (V2) to (B12);
\draw[-,very thick] (V2) to (B12l);

\draw[-,very thick] (V3) to (B12);
\draw[-,very thick] (V3) to (B12r);

\draw[-,very thick] (Bmr) to (B12r);
\draw[-,very thick] (Bmr) to (B1r);

\draw[-,very thick] (Bml) to (B12l);
\draw[-,very thick] (Bml) to (B1l);

\draw[-,very thick] (Bm2) to (B1l);
\draw[-,very thick] (Bm2) to (B1r);
\draw[-,very thick] (Bm1) to (B12l);
\draw[-,very thick] (Bm1) to (B12r);

\vertexLabelR[]{B1}{left}{$ $}
\vertexLabelR[]{B12}{left}{$ $}
\vertexLabelR[]{B1l}{left}{$ $}
\vertexLabelR[]{B1r}{left}{$ $}
\vertexLabelR[]{B12l}{left}{$ $}
\vertexLabelR[]{B12r}{left}{$ $}
\vertexLabelR[]{Bmr}{left}{$ $}
\vertexLabelR[]{Bm1}{left}{$ $}
\vertexLabelR[]{Bm2}{left}{$ $}
\vertexLabelR[]{Bml}{left}{$ $}

\vertexLabelR[]{V1}{left}{$ $}
\vertexLabelR[]{V2}{left}{$ $}
\vertexLabelR[]{V3}{left}{$ $}
\vertexLabelR[]{V4}{left}{$ $}
\end{tikzpicture}  
        \subcaption{}
        \label{subfig:5reg_limit_cycles}
    \end{subfigure}
    \caption{(a) The 5-regular graph $\Gamma'$, where each vertex $y_{(e,a)}$ is indicated by the label $a$ for better illustration. (b) The constructed $\Aut(\Gamma')$-invariant facial cycle double cover of $\Gamma'$.}
    \label{fig:5limit}
\end{figure}
\end{proof}

To extend the desired sequence of $d$-regular graphs for $d\geq 6$, we need the following lemma.

\begin{lemma}\label{lemma:chi_bound_colour_embedding}
Let $\Gamma$ be a $d$-regular graph with a $d$-edge-colouring $\kappa$. Further, let $\beta$ be the strong embedding of $\Gamma$ whose facial cycles are the
bicoloured cycles with respect to $\kappa$. Then
$$
\chi(\beta(\Gamma))\leq \left(1-\frac{d}{4}\right)|V(\Gamma)|.
$$
\end{lemma}

\begin{proof}
Let $F(\beta(\Gamma))$ denote the set of facial cycles of the embedding
$\beta$. Since the facial cycles are bi-coloured and the edge-colouring is
proper, each facial cycle has length at least $4$. Hence, by counting edge-face
incidences, we obtain
$$
4|F(\beta(\Gamma))|\leq 2|E(\Gamma)|.
$$
Thus,
$$
|F(\beta(\Gamma))|\leq \frac{1}{2}|E(\Gamma)|.
$$
Since $\Gamma$ is $d$-regular, we have
$$
|E(\Gamma)|=\frac{d}{2}|V(\Gamma)|.
$$
Therefore,
$$
\begin{aligned}
\chi(\beta(\Gamma))
&=|V(\Gamma)|-|E(\Gamma)|+|F(\beta(\Gamma))| \\
&\leq |V(\Gamma)|-|E(\Gamma)|+\frac{1}{2}|E(\Gamma)| \\
&=|V(\Gamma)|-\frac{1}{2}|E(\Gamma)| \\
&=|V(\Gamma)|-\frac{d}{4}|V(\Gamma)| \\
&=\left(1-\frac{d}{4}\right)|V(\Gamma)|.
\end{aligned}
$$
This proves the claim.
\end{proof}

\begin{theorem}\label{thm:infinite}
    For every finite group $G$ and every $d\geq 3$ there exists a sequence of $d$-regular graphs $(\Gamma_\ell)_{\ell \in \mathbb{N}}$ with corresponding strong embeddings $(\beta_\ell)_{\ell\in \mathbb{N}}$ such that 
    \begin{enumerate}
        \item $\Aut(\Gamma_\ell)\cong\Aut(\beta_\ell(\Gamma_\ell))\cong G$ for all $\ell\in \mathbb{N},$
        \item $\lim_{\ell\to \infty} g(\beta_\ell(\Gamma_\ell)) =\infty.$
    \end{enumerate}
\end{theorem}

\begin{proof}
We prove this statement by utilising graph products and the $\ell$-fold truncation operator of cubic graphs. Since the above statement has been proven for $d=3,4,5$ in Propositions~\ref{prop:infinite_3} and \ref{prop:infinite_45}, we can assume that $d\geq6$. 

By Theorem~\ref{theorem:dregular} we know that there exists a $(d{-}3)$-regular graph $\Gamma'$ with $\Aut(\Gamma')$-invariant $(d-3)$-edge-colouring $\kappa'$  such that $\Aut(\Gamma')\cong G$. By the same argument, there exists a cubic graph $\Gamma_1'$  with $\Aut(\Gamma_1')$-invariant $3$-edge-colouring $\kappa_1'$ satisfying $\vert \Aut(\Gamma_1')\vert =1$.
Since the graph $\Gamma_1'$ has trivial automorphism group, $\Gamma_1'$ is prime by Proposition~\ref{prop:cubicprime}. This implies that $\Gamma_{\ell+1}':=\mathcal{T}_{\ell}(\Gamma_1')$ is prime for all $\ell\geq 1.$ Furthermore, with Proposition~\ref{prop:truncproper}, the colouring $\kappa_1'$ induces an $\Aut(\Gamma_\ell')$-invariant $3$-edge-colouring $\kappa_\ell'$ of $\Gamma_\ell'$. Note that applying truncation to $\Gamma_1'$ again results in a cubic prime graph with trivial automorphism. Hence, we can choose $\Gamma_1$ to be a graph such that $\vert V(\Gamma_\ell')\vert > \vert V(\Gamma')\vert$ for all $\ell \in \mathbb{N}$. Since $\Gamma'$ cannot be a prime factor of $\Gamma'_{\ell}$, we deduce that $\Gamma'$ and $\Gamma_\ell'$ are relatively prime for all $\ell\in \mathbb{N}$.
Thus, $\Gamma_\ell:=\Gamma'\times \Gamma_\ell'$ is a $d$-regular graph that has an $\Aut(\Gamma_\ell)$-invariant $d$-edge-colouring $\kappa_\ell$ and satisfies  $\Aut(\Gamma_\ell)\cong\Aut(\Gamma')\cong G$, see Lemma~\ref{lemma:edgecolouring}. 
By Lemma~\ref{lemma:colouring_CDC}, for every $\ell\in \mathbb{N}$ the bi-colouring cycles of $\Gamma_\ell$ with respect to $\kappa_\ell$ give rise to a strong embedding $\beta_\ell$ of $\Gamma_\ell$ satisfying $$\Aut(\Gamma_\ell)\cong \Aut(\beta_\ell(\Gamma_\ell))\cong G.$$ By our construction we know that $\vert V(\Gamma_{\ell})\vert<\vert V(\Gamma_{\ell+1})\vert$ for all $\ell \in \mathbb{N}$. Thus, Lemma~\ref{lemma:chi_bound_colour_embedding} implies
\[
\lim_{\ell\to \infty} g(\beta_\ell(\Gamma_\ell)) =\infty.
\]
\end{proof}

\section*{Appendix}
Here, we present examples of regular graphs whose automorphism groups have order at most $2$, together with corresponding automorphism group-invariant edge-colourings. These graphs help us to establish Proposition~\ref{prop:atmost2}. More precisely, for $(d=3,4,5)$ we give a $d$-regular graph with trivial automorphism group and a $d$-regular graph with automorphism group of order two.
Since every colour class in a proper $d$-edge-colouring forms a perfect matching, we denote the colour classes by $M_i$, where $i = 1, \ldots, d$. We furthermore indicate whether the presented graphs are prime. The data in this section was verified using Magma~\cite{Bosma} and GAP~\cite{GAP4}; the corresponding implementations are available in~\cite{AkpanyaKregularSymmetricGraphs}. We also implemented the construction of $d$-regular graphs with prescribed automorphism group in both systems and tested it for $3 \leq d \leq 9$ on groups from the SmallGroups Library~\cite{SmallGroupPaper,MR1935567} of order at most $120$. In every tested case, the constructed graph had automorphism group isomorphic to the prescribed group $G$.

\subsection*{Examples}

\begin{table}[H]
\centering
\begin{tabular}{|l|l|}
\hline
$V(\Gamma_1)$ 
& $ \{1,\ldots,14\} $ \\
\hline
$E(\Gamma_1)$ 
& $\{\{1,3\}, \{1,10\}, \{1,12\}, \{2,3\}, \{2,9\}, \{2,13\}, \{3,11\},$ \\
& \,$\{4,5\}, \{4,6\}, \{4,8\}, \{5,7\}, \{5,9\}, \{6,7\}, \{6,10\},\{7,14\},$\\
&$\{8,9\}, \{8,11\}, \{10,11\}, \{12,13\}, \{12,14\}, \{13,14\}\}$ \\
\hline
$M_1$& $\{\{ 1, 3 \}, \{ 2, 9 \}, \{ 4, 5 \}, \{ 6, 10 \}, \{ 7, 14 \}, \{ 8, 11 \}, \{ 12, 13 \}\}$\\
\hline
$M_2$& $\{\{ 1, 10 \}, \{ 2, 13 \}, \{ 3, 11 \}, \{ 4, 8 \}, \{ 5, 9 \}, \{ 6, 7 \}, \{ 12, 14 \}\}$\\
\hline
$M_3$& $\{\{ 1, 12 \}, \{ 2, 3 \}, \{ 4, 6 \}, \{ 5, 7 \}, \{ 8, 9 \}, \{ 10, 11 \}, \{ 13, 14 \}\}$\\
\hline
$\Aut(\Gamma_1)$ & $\{()\}$\\
\hline
prime & This graph is prime by Proposition~\ref{prop:cubicprime}.\\
\hline
\end{tabular}
\caption{A cubic prime graph $\Gamma_1$ with trivial automorphism group }
\label{tab:3reg}
\end{table}

\begin{table}[H]
\centering
\begin{tabular}{|l|l|}
\hline
$V(\Gamma_2)$ 
& $ \{1,\ldots,16\} $ \\
\hline
$E(\Gamma_2)$ 
& $\{\{ 1, 3 \}, \{ 1, 9 \}, \{ 1, 11 \}, \{ 2, 3 \}, \{ 2, 8 \}, \{ 2, 12 \},\{ 3, 10 \}, \{ 4, 5 \},\{ 4, 6 \},\{ 4, 7 \},$\\
&$ \{ 5, 8 \}, \{ 5, 15 \}, \{ 6, 9 \}, \{ 6, 14 \}, \{ 7, 8 \}, \{ 7, 10 \},\{ 9, 10 \}, \{ 11, 12 \},$\\
    & $\{ 11, 13 \}, \{ 12, 13 \}, \{ 13, 16 \}, \{ 14, 15 \}, \{ 14, 16 \}, \{ 15, 16 \}\}$ \\
    \hline 
    $M_1$ &$\{\{ 1, 11 \}, \{ 2, 3 \}, \{ 4, 6 \}, \{ 5, 15 \}, \{ 7, 8 \}, \{ 9, 10 \}, \{ 12, 13 \}, \{ 14, 16 \}\}$\\
    \hline 
  $M_2$ &$\{\{ 1, 3 \}, \{ 2, 8 \}, \{ 4, 5 \}, \{ 6, 9 \}, \{ 7, 10 \}, \{ 11, 12 \}, \{ 13, 16 \}, \{ 14, 15 \}\}$ \\
  \hline 
  $M_3$ &$\{\{ 1, 9 \}, \{ 2, 12 \}, \{ 3, 10 \}, \{ 4, 7 \}, \{ 5, 8 \}, \{ 6, 14 \}, \{ 11, 13 \}, \{ 15, 16  \}\}$\\
  \hline 
$\Aut(\Gamma_2)$ & $\langle(1,5)(2,6)(3,4)(7,10)(8,9)(11,15)(12,14)(13,16)\rangle$\\
\hline 
prime & This graph is prime by Proposition~\ref{prop:cubicprime}.\\
\hline
\end{tabular}
\caption{A cubic prime graph $\Gamma_2$ with  automorphism group of order 2}
\end{table}

\begin{table}[H]
\centering
\begin{tabular}{|l|l|}
\hline
$V(\Gamma_3)$ 
& $ \{1,\ldots,32\} $ \\
\hline
$E(\Gamma_3)$ 
&$\{\{ 1, 2 \},\{ 1, 3 \},\{ 1, 12 \},\{ 1, 17 \},\{ 2, 4 \},\{ 2, 15 \},\{ 2, 18 \},\{ 3, 4\},\{ 3, 14 \},\{ 3, 19 \},$\\
&$\{ 4, 16 \},\{ 4, 20 \},\{ 5, 6 \},\{ 5, 7 \},\{ 5, 9 \},\{ 5, 21 \},\{  6, 8 \},\{ 6, 10 \},\{ 6, 22 \},\{ 7, 8 \},$\\
  &$\{ 7, 11 \},\{ 7, 23 \},\{ 8, 12 \},\{ 8, 24 \},\{ 9, 10 \},\{ 9, 13 \},\{ 9, 25 \},\{ 10, 15 \},\{ 10, 26 \},$\\
  &$\{ 11, 12 \},\{ 11, 14 \},\{11, 27 \},\{ 12, 28 \},\{ 13, 14 \},\{ 13, 16 \},\{ 13, 29 \},\{ 14, 30 \},$\\
  &$\{ 15, 16 \},\{ 15, 31 \},\{ 16, 32 \},\{ 17, 18 \},\{ 17, 25 \},\{ 17, 32 \},\{ 18, 19 \},\{ 18, 26 \},$\\
  &$\{ 19, 20 \},\{ 19, 27 \},\{ 20, 21 \},\{ 20, 28 \},\{ 21, 22 \},\{ 21, 29 \},\{ 22, 23 \},\{ 22, 30 \},$\\
  &$\{ 23, 24 \},\{ 23, 31 \},\{ 24, 25 \},\{ 24, 32 \},\{ 25, 26 \},\{ 26, 27 \},\{ 27, 28 \},\{ 28, 29 \},$\\
  &$\{ 29, 30 \},\{ 30, 31 \},\{ 31, 32 \}\}$
\\
\hline
$M_1$ & $\{\{ 1, 2 \},\{ 5, 6 \},\{ 7, 11 \},\{ 8, 12 \},\{ 9, 13 \},\{ 10, 15 \},\{ 3, 14 \},\{ 4, 16 \},\{17,18\},$\\
&$\{19,20\},\{21,22\},\{23,24\},\{25,26\},\{27,28\},\{29,30\},\{31,32\}\}$\\
\hline
$M_2$& $\{\{ 1, 3 \},\{ 2, 4 \},\{ 7, 8 \},\{ 5, 9 \},\{ 6, 10 \},\{ 11, 12 \},\{ 13, 14 \},\{ 15, 16 \},\{17,32\},$\\
&$\{18,19\},\{20,21\},\{22,23\},\{24,25\},\{26,27\},\{28,29\},\{30,31\}\}$\\
\hline
$M_3$& $\{\{ 3, 4 \},\{ 5, 7 \},\{ 6, 8 \},\{ 9, 10 \},\{ 11, 14 \},\{ 1, 12 \},\{ 2, 15 \},\{ 13, 16 \},\{17,25\},$\\
&$\{18,26\},\{19,27\},\{20,28\},\{21,29\},\{22,30\},\{23,31\},\{24,32\}\}$\\
\hline
$M_4$& $\{\{1,17\}, \{2,18\}, \{3,19\}, \{4,20\}, \{5,21\}, \{6,7\}, \{8,24\}, \{9,25\},\{10,26\},$\\
&$ \{11,13\}, \{12,15\}, \{14,16\}, \{22,27\}, \{23,30\}, \{28,31\}, \{29,32\}\}$\\
\hline
$\Aut(\Gamma_3)$ & $\{()\}$\\
\hline
Prime & This graph is prime by Corollary~\ref{cor:prime}, since the vertex $1$ is contained in\\
&exactly two cycles of length 4, namely $(1,2,4,3)$ and $(1,2,18,17)$\\
\hline
\end{tabular}
\caption{A 4-regular prime graph $\Gamma_3$ with trivial automorphism group }\label{tab:4reg}
\end{table}

\begin{table}[H]
\centering
\begin{tabular}{|l|l|}
\hline
$V(\Gamma_4)$ 
& $ \{1,\ldots,28\} $ \\
\hline
$E(\Gamma_4)$ 
& $\{ \{ 1, 3 \}, \{ 1, 10 \}, \{ 1, 12 \}, \{ 1, 15 \}, \{ 2, 3 \}, \{ 2, 9 \}, \{ 2, 13 \}, \{ 2, 16 \},\{ 3, 11 \}, \{ 3, 17 \},$\\
& $ \{ 4, 5 \},\{ 4, 6 \}, \{ 4, 8 \}, \{ 4, 18 \}, \{ 5, 7 \},\{ 5, 9 \}, \{ 5, 19 \}, \{ 6, 7 \}, \{ 6, 10 \}, \{ 6, 20 \}, $\\
&$\{ 7, 14 \}, \{ 7, 21 \}, \{ 8, 9 \},\{ 8, 11 \}, \{ 8, 22 \}, \{ 9, 23 \}, \{ 10, 11 \}, \{ 10, 24 \}, \{ 11, 25 \}, \{ 12, 13 \},$\\
&$\{ 12, 14 \},\{ 12, 26 \}, \{ 13, 14 \}, \{ 13, 27 \}, \{ 14, 28 \}, \{ 15, 17 \}, \{ 15, 24 \}, \{ 15, 26 \}, \{ 16, 17 \},\{ 16, 23 \},$\\
&$\{ 16, 27 \}, \{ 17, 25 \}, \{ 18, 19 \}, \{ 18, 20 \}, \{ 18, 22 \}, \{ 19, 21 \},\{ 19, 23 \}, \{ 20, 21 \}, \{ 20, 24 \}, \{ 21, 28 \},$\\
&$ \{ 22, 23 \}, \{ 22, 25 \}, \{ 24, 25 \}, \{ 26, 27 \}, \{ 26, 28 \}, \{ 27, 28 \} \}$\\
\hline
$M_1$& $\{\{ 1, 3 \}, \{ 2, 9 \}, \{ 4, 5 \}, \{ 6, 10 \}, \{ 7, 14 \}, \{ 8, 11 \}, \{ 12, 13 \},$ \\
&$\{ 15, 17 \}, \{ 16, 23 \}, \{ 18, 19 \}, \{ 20, 24 \}, \{ 21, 28 \}, \{ 22, 25 \}, \{ 26, 27 \}\}$\\
\hline
$M_2$& $\{\{ 1, 10 \}, \{ 2, 13 \}, \{ 3, 11 \}, \{ 4, 8 \}, \{ 5, 9 \}, \{ 6, 7 \}, \{ 12, 14 \},$\\
&$\{ 15, 24 \}, \{ 16, 27 \}, \{ 17, 25 \}, \{ 18, 22 \}, \{ 19, 23 \}, \{ 20, 21 \}, \{ 26, 28 \}\}$\\
\hline
$M_3$& $\{\{ 1, 12 \}, \{ 2, 3 \}, \{ 4, 6 \}, \{ 5, 7 \}, \{ 8, 9 \}, \{ 10, 11 \}, \{ 13, 14 \},$\\
&$ \{ 15, 26 \}, \{ 16, 17 \}, \{ 18, 20 \}, \{ 19, 21 \}, \{ 22, 23 \}, \{ 24, 25 \}, \{ 27, 28 \}\}$\\
\hline
$M_4$& $\{\{ 1, 15 \}, \{ 2, 16 \}, \{ 3, 17 \}, \{ 4, 18 \}, \{ 5, 19 \}, \{ 6, 20 \}, \{ 7, 21 \},$\\
&$\{ 8, 22 \},\{ 9, 23 \},\{ 10, 24 \},\{ 11, 25 \},\{ 12, 26 \},\{ 13, 27 \},\{ 14, 28 \}\}$\\
\hline
$\Aut(\Gamma_4)$ & $\langle (1, 15 )( 2, 16 )( 3, 17 )( 4, 18 )( 5, 19 )( 6, 20 )( 7, 21 )( 8, 22 )( 9, 23 )$\\
&$(10, 24 )( 11, 25 )( 12, 26 )( 13, 27 )( 14, 28 ) \rangle$\\
\hline
prime & This graph is not prime, since it is the product of $K_2$ and\\
&the cubic graph $\Gamma_1$ presented in Table~\ref{tab:3reg}.\\
\hline
\end{tabular}
\caption{A 4-regular graph $\Gamma_4$ with automorphism group of order 2 } 
\end{table}

\begin{table}[H]
\centering
\begin{tabular}{|l|l|}
\hline
$V(\Gamma_5)$ 
& $ \{1,\ldots,64\} $ \\
\hline
$E(\Gamma_5)$ 
&    $\{\{ 1, 2 \}, \{ 1, 3 \}, \{ 1, 12 \}, \{ 1, 17 \}, \{ 1, 33 \}, \{ 2, 4 \}, \{ 2, 15 \}, \{ 2, 18 \}, \{ 2, 34 \}, \{ 3, 4 \},$\\
  &$\{3, 14 \}, \{ 3, 19 \}, \{ 3, 35 \}, \{ 4, 16 \}, \{ 4, 20 \}, \{ 4, 36 \}, \{ 5, 6 \}, \{ 5, 7 \}, \{ 5, 9 \}, \{ 5, 21 \},$\\
  &$\{ 5, 37 \}, \{ 6, 8 \}, \{ 6, 10 \}, \{ 6, 22 \}, \{ 6, 38 \}, \{ 7, 8 \}, \{ 7, 11 \}, \{ 7, 23 \}, \{ 7, 39 \}, \{ 8, 12 \},$\\ 
  & $\{ 8, 24 \}, \{ 8, 40 \}, \{ 9, 10 \}, \{ 9, 13 \}, \{ 9, 25 \}, \{ 9, 41 \}, \{ 10, 15 \}, \{ 10, 26 \}, \{ 10, 42 \},$\\
  &$ \{ 11, 12 \}, \{ 11, 14 \}, \{ 11, 27 \}, \{ 11, 43 \}, \{ 12, 28 \}, \{ 12, 44 \}, \{ 13, 14 \}, \{ 13, 16 \},$\\
  &$ \{ 13, 29 \}, \{ 13, 45 \}, \{ 14, 30 \}, \{ 14, 46 \}, \{ 15, 16 \}, \{ 15, 31 \}, \{ 15, 47 \}, \{ 16, 32 \},$\\
  &$ \{ 16, 48 \}, \{ 17, 18 \}, \{ 17, 25 \}, \{ 17, 32 \}, \{ 17, 49 \}, \{ 18, 19 \}, \{ 18, 26 \}, \{ 18, 50 \},$  \\
  &$ \{ 19, 20 \}, \{ 19, 27 \}, \{ 19, 51 \}, \{ 20, 21 \}, \{ 20, 28 \}, \{ 20, 52 \}, \{ 21, 22 \}, \{ 21, 29 \},$\\
  &$ \{ 21, 53 \}, \{ 22, 23 \}, \{ 22, 30 \}, \{ 22, 54 \}, \{ 23, 24 \}, \{ 23, 31 \}, \{ 23, 55 \}, \{ 24, 25 \},$\\
  &$ \{ 24, 32 \}, \{ 24, 56 \}, \{ 25, 26 \}, \{ 25, 57 \}, \{ 26, 27 \}, \{ 26, 58 \}, \{ 27, 28 \}, \{ 27, 59 \},$\\
  &$ \{ 28, 29 \}, \{ 28, 60 \}, \{ 29, 30 \}, \{ 29, 61 \}, \{ 30, 31 \}, \{ 30, 62 \}, \{ 31, 32 \}, \{ 31, 63 \},$\\
  &$ \{ 32, 64 \}, \{ 33, 34 \}, \{ 33, 37 \}, \{ 33, 61 \}, \{ 33, 64 \}, \{ 34, 35 \}, \{ 34, 38 \}, \{ 34, 62 \},$\\
  &$ \{ 35, 36 \}, \{ 35, 39 \}, \{ 35, 63 \}, \{ 36, 37 \}, \{ 36, 40 \}, \{ 36, 64 \}, \{ 37, 38 \}, \{ 37, 41 \},$\\
  &$ \{ 38, 39 \}, \{ 38, 42 \}, \{ 39, 40 \}, \{ 39, 43 \}, \{ 40, 41 \}, \{ 40, 44 \}, \{ 41, 42 \}, \{ 41, 45 \},$\\
  &$ \{ 42, 43 \}, \{ 42, 46 \}, \{ 43, 44 \}, \{ 43, 47 \}, \{ 44, 45 \}, \{ 44, 48 \}, \{ 45, 46 \}, \{ 45, 49 \},$\\
  &$ \{ 46, 47 \}, \{ 46, 50 \}, \{ 47, 48 \}, \{ 47, 51 \}, \{ 48, 49 \}, \{ 48, 52 \}, \{ 49, 50 \}, \{ 49, 53 \},$\\
  &$ \{ 50, 51 \}, \{ 50, 54 \}, \{ 51, 52 \}, \{ 51, 55 \}, \{ 52, 53 \}, \{ 52, 56 \}, \{ 53, 54 \}, \{ 53, 57 \},$\\
  &$ \{ 54, 55 \}, \{ 54, 58 \}, \{ 55, 56 \}, \{ 55, 59 \}, \{ 56, 57 \}, \{ 56, 60 \}, \{ 57, 58 \}, \{ 57, 61 \},$\\
  &$ \{ 58, 59 \}, \{ 58, 62 \}, \{ 59, 60 \}, \{ 59, 63 \}, \{ 60, 61 \}, \{ 60, 64 \}, \{ 61, 62 \}, \{ 62, 63 \}, \{ 63, 64\}\}$\\
\hline
    $M_1$ &$\{\{ 1, 2 \},\{ 5, 6 \},\{ 7, 11 \},\{ 8, 12 \},\{ 9, 13 \},\{ 10, 15 \},\{ 3, 14 \},\{ 4, 16 \},\{17,18\},$\\
&$\{19,20\},\{21,22\},\{23,24\},\{25,26\},\{27,28\},\{29,30\},\{31,32\},\{ 33, 34 \}, $\\
&$\{ 35, 36 \}, \{ 37, 38 \}, \{ 39, 40 \}, \{ 41, 42 \}, \{ 43, 44 \}, \{ 45, 46 \}, \{ 47, 48 \}, \{ 49, 50 \},$\\
&$ \{ 51, 52 \}, \{ 53, 54 \}, \{ 55, 56 \}, \{ 57, 58 \}, \{ 59, 60 \}, \{ 61, 62 \}, \{ 63, 64 \}\}$\\
\hline 
$M_2$&$\{\{ 1, 3 \},\{ 2, 4 \},\{ 7, 8 \},\{ 5, 9 \},\{ 6, 10 \},\{ 11, 12 \},\{ 13, 14 \},\{ 15, 16 \},\{17,32\}, $\\
&$\{18,19\},\{20,21\},\{22,23\},\{24,25\},\{26,27\},\{28,29\},\{30,31\}, \{ 34, 35 \},$\\
&$\{ 36, 37 \}, \{ 38, 39 \}, \{ 40, 41 \}, \{ 42, 43 \}, \{ 44, 45 \}, \{ 46, 47 \}, \{ 48, 49 \}, \{ 50, 51 \},$\\
&$ \{ 52, 53 \}, \{ 54, 55 \}, \{ 56, 57 \}, \{ 58, 59 \}, \{ 60, 61 \}, \{ 62, 63 \}, \{ 64, 33 \}\}$\\
\hline
$M_3$ & $\{\{ 3, 4 \},\{ 5, 7 \},\{ 6, 8 \},\{ 9, 10 \},\{ 11, 14 \},\{ 1, 12 \},\{ 2, 15 \},\{ 13, 16 \},\{17,25\},$\\
&$\{18,26\},\{19,27\},\{20,28\},\{21,29\},\{22,30\},\{23,31\},\{24,32\},\{ 33, 37 \},$\\
&$ \{ 41, 45 \}, \{ 49, 53 \}, \{ 57, 61\}, \{ 34, 38 \}, \{ 42, 46 \}, \{ 50, 54 \}, \{ 58, 62 \},\{35, 39 \}, $\\
&$\{ 43, 47 \}, \{ 51, 55 \}, \{ 59, 63 \},\{ 36, 40 \}, \{ 44, 48 \}, \{ 52, 56 \}, \{ 60, 64 \}\}$\\
\hline 
$M_4$& $\{\{1,17\}, \{2,18\}, \{3,19\}, \{4,20\}, \{5,21\}, \{6,7\}, \{8,24\}, \{9,25\},\{10,26\},$\\
&$ \{11,13\}, \{12,15\}, \{14,16\}, \{22,27\}, \{23,30\}, \{28,31\}, \{29,32\}, \{ 37, 41 \},$\\
&$ \{ 45, 49 \}, \{ 53, 57 \}, \{ 61, 33 \},\{ 38, 42 \}, \{ 46, 50 \}, \{ 54, 58 \}, \{ 62, 34 \},\{ 39, 43 \},$\\
&$ \{ 47, 51 \}, \{ 55, 59 \}, \{ 63, 35 \},\{ 40, 44 \}, \{ 48, 52 \}, \{ 56, 60 \}, \{ 64, 36 \}\}$\\
\hline
$M_5$ &$\{\{ 1, 33 \}, \{ 2, 34 \}, \{ 3, 35 \}, \{ 4, 36 \}, \{ 5, 37 \}, \{ 6, 38 \}, \{ 7, 39 \}, \{ 8, 40 \}, \{ 9, 41 \},$\\
&$ \{ 10, 42 \}, \{ 11, 43 \}, \{ 12, 44 \}, \{ 13, 45 \}, \{ 14, 46 \}, \{ 15, 47 \}, \{ 16, 48 \}, \{ 17, 49 \},$\\
&$ \{ 18, 50 \}, \{ 19, 51 \}, \{ 20, 52 \}, \{ 21, 53 \}, \{ 22, 54 \}, \{ 23, 55 \}, \{ 24, 56 \}, \{ 25, 57 \},$\\
&$ \{ 26, 58 \}, \{ 27, 59 \}, \{ 28, 60 \}, \{ 29, 61 \}, \{ 30, 62 \}, \{ 31, 63 \}, \{ 32, 64 \}\}$\\
\hline
$\Aut(\Gamma_5)$ & $\{()\}$\\
\hline 
Prime & This graph is prime by Corollary~\ref{cor:prime}, since the vertex $1$ is contained in\\
&exactly three cycles of length 4, namely $(1,2,34,33),(1,2,18,17)$ and $(1,3,4,2).$\\
\hline
\end{tabular}
\caption{A 5-regular prime graph $\Gamma_5$ with trivial automorphism group  }
\end{table}

\begin{table}[H]
\centering
\begin{tabular}{|l|l|}
\hline
$V(\Gamma_6)$ 
& $ \{1,\ldots,64\} $ \\
\hline
$E(\Gamma_6)$ 
&$\{\{ 1, 2\},\{ 1, 3\},\{ 1, 12\},\{ 1, 17\},\{ 2, 4\},\{ 2, 15\}, 
 \{ 2, 18\},\{ 3, 4\},\{ 3, 14\},$\\
& $ \{ 3, 19\},\{ 4, 16\},\{ 4, 20\}, 
 \{ 5, 6\},\{ 5, 7\},\{ 5, 9\},\{ 5, 21\},\{ 6, 8\},\{ 6, 10\},$\\
 & $\{ 6, 22\}, 
 \{ 7, 8\},\{ 7, 11\},\{ 7, 23\},\{ 8, 12\},\{ 8, 24\},\{ 9, 10\}, 
 \{ 9, 13\},\{ 9, 25\},$\\
 &$\{ 10, 15\},\{ 10, 26\},\{ 11, 12\},\{ 11, 14\}, 
 \{ 11, 27\},\{ 12, 28\},\{ 13, 14\},\{ 13, 16\},$\\
 &$\{ 13, 29\},\{ 14, 30\}, 
 \{ 15, 16\},\{ 15, 31\},\{ 16, 32\},\{ 17, 18\},\{ 17, 25\},\{ 17, 32\}, $\\
 &$\{ 18, 19\},\{ 18, 26\},\{ 19, 20\},\{ 19, 27\},\{ 20, 21\},\{ 20, 28\}, 
 \{ 21, 22\},\{ 21, 29\},$\\
 &$\{ 22, 23\},\{ 22, 30\},\{ 23, 24\},\{ 23, 31\}, 
 \{ 24, 25\},\{ 24, 32\},\{ 25, 26\},\{ 26, 27\},$\\
 &$\{ 27, 28\},\{ 28, 29\}, 
 \{ 29, 30\},\{ 30, 31\},\{ 31, 32\},\{ 33, 34\},\{ 33, 35\},\{ 33, 44\}, $\\
 &$\{ 33, 49\},\{ 34, 36\},\{ 34, 47\},\{ 34, 50\},\{ 35, 36\},\{ 35, 46\}, 
 \{ 35, 51\},\{ 36, 48\},$\\
 &$\{ 36, 52\},\{ 37, 38\},\{ 37, 39\},\{ 37, 41 \}, 
 \{ 37, 53\},\{ 38, 40\},\{ 38, 42\},\{ 38, 54\},$\\
 &$\{ 39, 40\},\{ 39, 43\}, 
 \{ 39, 55\},\{ 40, 44\},\{ 40, 56\},\{ 41, 42\},\{ 41, 45\},\{ 41, 57\}, $\\
 &$\{ 42, 47\},\{ 42, 58\},\{ 43, 44\},\{ 43, 46\},\{ 43, 59\},\{ 44, 60\}, 
 \{ 45, 46\},\{ 45, 48\},$\\
 &$\{ 45, 61\},\{ 46, 62\},\{ 47, 48\},\{ 47, 63\}, 
 \{ 48, 64\},\{ 49, 50\},\{ 49, 57\},\{ 49, 64\},$\\
 &$\{ 50, 51\},\{ 50, 58\}, 
 \{ 51, 52\},\{ 51, 59\},\{ 52, 53\},\{ 52, 60\},\{ 53, 54\},\{ 53, 61\}, $\\
 &$
 \{ 54, 55\},\{ 54, 62\},\{ 55, 56\},\{ 55, 63\},\{ 56, 57\},\{ 56, 64\}, 
 \{ 57, 58\},\{ 58, 59\},$\\
 &$\{ 59, 60\},\{ 60, 61\},\{ 61, 62\},\{ 62, 63\}, 
 \{ 63, 64\},\{ 1, 33\},\{ 2, 34\},\{ 3, 35\},$\\
 &$\{ 4, 36\},\{ 5, 37\}, 
 \{ 6, 38\},\{ 7, 39\},\{ 8, 40\},\{ 9, 41\},\{ 10, 42\},\{ 11, 43\}, 
 \{ 12, 44\}, $\\
 &$\{ 13, 45\},\{ 14, 46\},\{ 15, 47\},\{ 16, 48\},\{ 17, 49\}, 
 \{ 18, 50\},\{ 19, 51\},\{ 20, 52\},$\\
& $\{ 21, 53\},\{ 22, 54\},\{ 23, 55\}, 
 \{ 24, 56\},\{ 25, 57\},\{ 26, 58\},\{ 27, 59\},\{ 28, 60\},$\\
 &$\{ 29, 61\}, 
 \{ 30, 62\},\{ 31, 63\},\{ 32, 64\}\}$\\
\hline
$M_1$ & $\{\{ 1, 2 \},\{ 5, 6 \},\{ 7, 11 \},\{ 8, 12 \},\{ 9, 13 \},\{ 10, 15 \},\{ 3, 14 \},\{ 4, 16 \},\{17,18\},$\\
&$\{19,20\},\{21,22\},\{23,24\},\{25,26\},\{27,28\},\{29,30\},\{31,32\},\{ 33, 34\},$\\
&$\{ 37, 38\},\{ 39, 43\},\{ 40, 44\},\{ 41, 45\},\{ 42, 47\}, 
 \{ 35, 46\},\{ 36, 48\},\{ 49, 50\},$\\
 & $\{ 51, 52\},\{ 53, 54\},\{ 55, 56\}, 
 \{ 57, 58\},\{ 59, 60\},\{ 61, 62\},\{ 63, 64\}\}$\\
\hline
$M_2$ & $\{\{ 1, 3 \},\{ 2, 4 \},\{ 7, 8 \},\{ 5, 9 \},\{ 6, 10 \},\{ 11, 12 \},\{ 13, 14 \},\{ 15, 16 \},\{17,32\},$\\
&$\{18,19\},\{20,21\},\{22,23\},\{24,25\},\{26,27\},\{28,29\},\{30,31\},\{ 33, 35 \},$\\
&$\{ 34, 36 \},\{ 39, 40 \},\{ 37, 41 \},\{ 38, 42 \},\{ 43, 44 \},\{ 45, 46 \},\{ 47, 48 \},$\\
&$\{ 49, 64 \},\{ 50, 51 \},\{ 52, 53 \},\{ 54, 55 \},\{ 56, 57 \},\{ 58, 59 \},\{ 60, 61 \},\{ 62, 63 \}\}$\\
\hline
$M_3$& $\{\{ 1, 12 \},\{ 2, 15 \},\{ 3, 4 \},\{ 5, 7 \},\{ 6, 8 \},\{ 9, 10 \},\{ 11, 14 \},\{ 13, 16 \},\{17,25\},$\\
&$\{18,26\},\{19,27\},\{20,28\},\{21,29\},\{22,30\},\{23,31\},\{24,32\},\{ 35, 36\},$\\
&$\{ 37, 39\},\{ 38, 40\},\{ 41, 42\},\{ 43, 46\},\{ 33, 44\}, 
 \{ 34, 47\},\{ 45, 48\},\{ 49, 57\},$\\
 &$\{ 50, 58\},\{ 51, 59\},\{ 52, 60\}, 
 \{ 53, 61\},\{ 54, 62\},\{ 55, 63\},\{ 56, 64\}\}$\\
 \hline
$M_4$& $\{\{1,17\}, \{2,18\}, \{3,19\}, \{4,20\}, \{5,21\}, \{6,7\}, \{8,24\}, \{9,25\},\{10,26\},$\\
&$ \{11,13\}, \{12,15\}, \{14,16\}, \{22,27\}, \{23,30\}, \{28,31\}, \{29,32\},\{ 33, 49\},$\\
&$\{ 34, 50\},\{ 35, 51\},\{ 36, 52\},\{ 37, 53\},\{ 38, 39\}, 
 \{ 40, 56\},\{ 41, 57\},\{ 42, 58\},$\\
 &$\{ 43, 45\},\{ 44, 47\},\{ 46, 48\}, 
 \{ 54, 59\},\{ 55, 62\},\{ 60, 63\},\{ 61, 64\}\}$\\
\hline
$M_5$ &$\{\{ 1, 33 \}, \{ 2, 34 \}, \{ 3, 35 \}, \{ 4, 36 \}, \{ 5, 37 \}, \{ 6, 38 \}, \{ 7, 39 \}, \{ 8, 40 \}, \{ 9, 41 \},$\\
&$ \{ 10, 42 \}, \{ 11, 43 \}, \{ 12, 44 \}, \{ 13, 45 \}, \{ 14, 46 \}, \{ 15, 47 \}, \{ 16, 48 \}, \{ 17, 49 \},$\\
&$ \{ 18, 50 \}, \{ 19, 51 \}, \{ 20, 52 \}, \{ 21, 53 \}, \{ 22, 54 \}, \{ 23, 55 \}, \{ 24, 56 \}, \{ 25, 57 \},$\\
&$ \{ 26, 58 \}, \{ 27, 59 \}, \{ 28, 60 \}, \{ 29, 61 \}, \{ 30, 62 \}, \{ 31, 63 \}, \{ 32, 64 \}\}$\\

\hline
$\Aut(\Gamma_6)$& $\langle (1,33)(2,34)(3,35)(4,36)(5,37)(6,38)(7,39)(8,40)(9,41)(10,42)$\\
&$ (11,43)(12,44)(13,45)(14,46)(15,47)(16,48)(17,49)(18,50)(19,51)$\\
&$(20,52)
  (21,53)(22,54)(23,55)(24,56)(25,57)(26,58)(27,59)(28,60)$\\
  &$(29,61)(30,62)
  (31,63)(32,64)\rangle$\\
\hline
Prime & This graph is not prime, since it is the product of $K_2$ and\\
&the 4-regular graph $\Gamma_3$ presented in Table~\ref{tab:4reg}.\\
\hline
\end{tabular}
\caption{A 5-regular graph $\Gamma_6$ with automorphism group of order 2}
\end{table}

\section*{Acknowledgements}
The authors would like to thank Marco Barbieri for valuable discussions and for drawing our attention to some relevant literature.
M.\ Weiß gratefully acknowledges the funding by the Deutsche Forschungsgemeinschaft (DFG, German Research Foundation) in the framework of the Collaborative Research Centre CRC/TRR 280 “Design Strategies for Material-Minimized Carbon Reinforced Concrete Structures – Principles of a New Approach to Construction” (project ID 417002380).
Furthermore, R.\ Akpanya was supported by a grant from the Simons Foundation (SFI-MPS-Infrastructure-00008650, JV). Tom Goertzen was supported by the Australian Government through the Australian Research Council's Discovery Projects funding scheme (project DP230102982).

\bibliographystyle{plain}
\bibliography{main}

@article{Sabidussi, 
title={Graphs with Given Group and Given Graph-Theoretical Properties}, 
volume={9}, 
DOI={10.4153/CJM-1957-060-7}, 
journal={Canadian Journal of Mathematics}, 
author={Sabidussi, Gert}, 
year={1957}, 
pages={515–525}}

@article {SmallGroupPaper,
    AUTHOR = {Besche, Hans Ulrich and Eick, Bettina and O'Brien, E. A.},
     TITLE = {The groups of order at most 2000},
   JOURNAL = {Electron. Res. Announc. Amer. Math. Soc.},
  FJOURNAL = {Electronic Research Announcements of the American Mathematical
              Society},
    VOLUME = {7},
      YEAR = {2001},
     PAGES = {1--4},
      ISSN = {1079-6762},
   MRCLASS = {20D60},
  MRNUMBER = {1826989},
       DOI = {10.1090/S1079-6762-01-00087-7},
       URL = {https://doi.org/10.1090/S1079-6762-01-00087-7},
}

@article{SurfacesWithAuto,
author = {Akpanya, Reymond and Goertzen, Tom},
year = {2025},
month = {07},
pages = {},
title = {Simplicial surfaces with given automorphism group},
volume = {62},
journal = {Journal of Algebraic Combinatorics},
doi = {10.1007/s10801-025-01424-4}
}

@article{Frucht,
title={Graphs of Degree Three with a Given Abstract Group},
volume={1},
DOI={10.4153/CJM-1949-033-6},
number={4}, 
journal={Canadian Journal of Mathematics}, 
author={Frucht, Robert}, 
year={1949}, 
pages={365–378}}

@article {MR1557026,
    AUTHOR = {Frucht, R.},
     TITLE = {Herstellung von {G}raphen mit vorgegebener abstrakter
              {G}ruppe},
   JOURNAL = {Compositio Math.},
  FJOURNAL = {Compositio Mathematica},
    VOLUME = {6},
      YEAR = {1939},
     PAGES = {239--250},
      ISSN = {0010-437X,1570-5846},
   MRCLASS = {99-04},
  MRNUMBER = {1557026},
       URL = {http://www.numdam.org/item?id=CM_1939__6__239_0},
}

@article {MR4866595,
    AUTHOR = {Ba\v{s}i\'{c}, Nino and Fowler, Patrick W.},
     TITLE = {Nut graphs with a given automorphism group},
   JOURNAL = {J. Algebraic Combin.},
  FJOURNAL = {Journal of Algebraic Combinatorics. An International Journal},
    VOLUME = {61},
      YEAR = {2025},
    NUMBER = {2},
     PAGES = {Paper No. 17, 12},
      ISSN = {0925-9899,1572-9192},
   MRCLASS = {05C25 (05C50)},
  MRNUMBER = {4866595},
MRREVIEWER = {Wenwen\ Fan},
       DOI = {10.1007/s10801-025-01389-4},
       URL = {https://doi.org/10.1007/s10801-025-01389-4},
}

@article {MR1211265,
    AUTHOR = {\v{S}ir\'{a}\v{n}, Jozef and \v{S}koviera, Martin},
     TITLE = {Orientable and nonorientable maps with given automorphism
              groups},
   JOURNAL = {Australas. J. Combin.},
  FJOURNAL = {The Australasian Journal of Combinatorics},
    VOLUME = {7},
      YEAR = {1993},
     PAGES = {47--53},
      ISSN = {1034-4942},
   MRCLASS = {05C10},
  MRNUMBER = {1211265},
MRREVIEWER = {T. R. S. Walsh},
}

@article {graphmult,
    AUTHOR = {Sabidussi, Gert},
     TITLE = {Graph multiplication},
   JOURNAL = {Math. Z.},
  FJOURNAL = {Mathematische Zeitschrift},
    VOLUME = {72},
      YEAR = {1959/60},
     PAGES = {446--457},
      ISSN = {0025-5874,1432-1823},
   MRCLASS = {05.40},
  MRNUMBER = {209177},
MRREVIEWER = {Solomon\ Marcus},
       DOI = {10.1007/BF01162967},
       URL = {https://doi.org/10.1007/BF01162967},
}

@incollection {seymour,
    AUTHOR = {Seymour, P. D.},
     TITLE = {Sums of circuits},
 BOOKTITLE = {Graph theory and related topics ({P}roc. {C}onf., {U}niv.
              {W}aterloo, {W}aterloo, {O}nt., 1977)},
     PAGES = {341--355},
 PUBLISHER = {Academic Press, New York-London},
      YEAR = {1979},
      ISBN = {0-12-114350-3},
   MRCLASS = {05C38 (90B10)},
  MRNUMBER = {538060},
MRREVIEWER = {Peter\ Brucker},
}

@article{szekeres, title={Polyhedral decompositions of cubic graphs},
volume={8}, 
DOI={10.1017/S0004972700042660}, 
number={3}, 
journal={Bulletin of the Australian Mathematical Society}, 
publisher={Cambridge University Press}, 
author={Szekeres, G.}, 
year={1973}, 
pages={367–387}
}

@book {GraphsOnSurfaces,
    AUTHOR = {Mohar, Bojan and Thomassen, Carsten},
     TITLE = {Graphs on surfaces},
    SERIES = {Johns Hopkins Studies in the Mathematical Sciences},
 PUBLISHER = {Johns Hopkins University Press, Baltimore, MD},
      YEAR = {2001},
     PAGES = {xii+291},
      ISBN = {0-8018-6689-8},
   MRCLASS = {05C10 (57M15)},
  MRNUMBER = {1844449},
MRREVIEWER = {Arthur\ T.\ White},
}

@book {TopologicalGraphTheory,
    AUTHOR = {Gross, Jonathan L. and Tucker, Thomas W.},
     TITLE = {Topological graph theory},
 PUBLISHER = {Dover Publications, Inc., Mineola, NY},
      YEAR = {2001},
     PAGES = {xvi+361},
      ISBN = {0-486-41741-7},
   MRCLASS = {05C10 (20F65 57M15)},
  MRNUMBER = {1855951},
}

@article {Bosma,
    AUTHOR = {Bosma, Wieb and Cannon, John and Playoust, Catherine},
     TITLE = {The {M}agma algebra system. {I}. {T}he user language},
      NOTE = {Computational algebra and number theory (London, 1993)},
   JOURNAL = {J. Symbolic Comput.},
  FJOURNAL = {Journal of Symbolic Computation},
    VOLUME = {24},
      YEAR = {1997},
    NUMBER = {3-4},
     PAGES = {235--265},
      ISSN = {0747-7171,1095-855X},
   MRCLASS = {68Q40},
  MRNUMBER = {1484478},
       DOI = {10.1006/jsco.1996.0125},
       URL = {https://doi.org/10.1006/jsco.1996.0125},
}

@article {Cori,
    AUTHOR = {Cori, Robert and Mach\`i, Antonio},
     TITLE = {Construction of maps with prescribed automorphism group},
   JOURNAL = {Theoret. Comput. Sci.},
  FJOURNAL = {Theoretical Computer Science},
    VOLUME = {21},
      YEAR = {1982},
    NUMBER = {1},
     PAGES = {91--98},
      ISSN = {0304-3975,1879-2294},
   MRCLASS = {05C25 (20B25)},
  MRNUMBER = {672104},
MRREVIEWER = {Cheryl\ E.\ Praeger},
       DOI = {10.1016/0304-3975(82)90090-1},
       URL = {https://doi.org/10.1016/0304-3975(82)90090-1},
}

@article {FruchtHow,
    AUTHOR = {Frucht, Roberto W.},
     TITLE = {How {I} became interested in graphs and groups},
   JOURNAL = {J. Graph Theory},
  FJOURNAL = {Journal of Graph Theory},
    VOLUME = {6},
      YEAR = {1982},
    NUMBER = {2},
     PAGES = {101--104},
      ISSN = {0364-9024,1097-0118},
   MRCLASS = {01A70 (05-03)},
  MRNUMBER = {655194},
       DOI = {10.1002/jgt.3190060203},
       URL = {https://doi.org/10.1002/jgt.3190060203},
}

@book {babaiBook,
    AUTHOR = {Lov\'asz, L\'aszl\'o},
     TITLE = {Combinatorial problems and exercises},
   EDITION = {Second},
 PUBLISHER = {North-Holland Publishing Co., Amsterdam},
      YEAR = {1993},
     PAGES = {635},
      ISBN = {0-444-81504-X},
   MRCLASS = {05-01},
  MRNUMBER = {1265492},
  NOTE      = {p. 426},
}

@article {MR406855,
    AUTHOR = {Babai, L\'aszl\'o},
     TITLE = {On the minimum order of graphs with given group},
   JOURNAL = {Canad. Math. Bull.},
  FJOURNAL = {Canadian Mathematical Bulletin. Bulletin Canadien de
              Math\'ematiques},
    VOLUME = {17},
      YEAR = {1974},
    NUMBER = {4},
     PAGES = {467--470},
      ISSN = {0008-4395,1496-4287},
   MRCLASS = {05C25},
  MRNUMBER = {406855},
MRREVIEWER = {Brian\ Alspach},
       DOI = {10.4153/CMB-1974-082-9},
       URL = {https://doi.org/10.4153/CMB-1974-082-9},
}

@article {MR316298,
    AUTHOR = {Babai, L.},
     TITLE = {Groups of graphs on given surfaces},
   JOURNAL = {Acta Math. Acad. Sci. Hungar.},
  FJOURNAL = {Acta Mathematica. Academiae Scientiarum Hungaricae},
    VOLUME = {24},
      YEAR = {1973},
     PAGES = {215--221},
      ISSN = {0001-5954,1588-2632},
   MRCLASS = {05C25},
  MRNUMBER = {316298},
MRREVIEWER = {L.\ V.\ Quintas},
       DOI = {10.1007/BF01894629},
       URL = {https://doi.org/10.1007/BF01894629},
}

@misc{AkpanyaKregularSymmetricGraphs,
  author       = {Akpanya, Reymond},
  title        = {{K-Regular Graphs with Prescribed Automorphism Groups: Magma and GAP Implementations}},
  year         = {2026},
  howpublished = {\url{https://github.com/reymondakpanya/KregularSymmetricGraphs}},
  note         = {GitHub repository}
}

@manual{GAP4,
    organization = "The GAP~Group",
    title        = "{GAP -- Groups, Algorithms, and Programming,
                    Version 4.16.0}",
    year         = 2026,
    url          = "\url{https://www.gap-system.org}",
    }

@article {MR1935567,
    AUTHOR = {Besche, Hans Ulrich and Eick, Bettina and O'Brien, E. A.},
     TITLE = {A millennium project: constructing small groups},
   JOURNAL = {Internat. J. Algebra Comput.},
  FJOURNAL = {International Journal of Algebra and Computation},
    VOLUME = {12},
      YEAR = {2002},
    NUMBER = {5},
     PAGES = {623--644},
      ISSN = {0218-1967,1793-6500},
   MRCLASS = {20D99},
  MRNUMBER = {1935567},
MRREVIEWER = {Robert\ M.\ Guralnick},
       DOI = {10.1142/S0218196702001115},
       URL = {https://doi.org/10.1142/S0218196702001115},
}

@book {MR1788124,
    AUTHOR = {Imrich, Wilfried and Klav\v{z}ar, Sandi},
     TITLE = {Product graphs},
    SERIES = {Wiley-Interscience Series in Discrete Mathematics and
              Optimization},
      NOTE = {Structure and recognition,
              With a foreword by Peter Winkler},
 PUBLISHER = {Wiley-Interscience, New York},
      YEAR = {2000},
     PAGES = {xvi+358},
      ISBN = {0-471-37039-8},
   MRCLASS = {05-01 (05C75 05C85 05C90 68R10)},
  MRNUMBER = {1788124},
MRREVIEWER = {Pranava\ K.\ Jha},
}

\end{document}